\documentclass[11pt,hyp,]{article}
\usepackage{hyperref}
\usepackage{amsmath,amsthm,amsfonts,amssymb,amscd}
\usepackage{amsmath}
\usepackage{enumitem}
\hypersetup{nesting=true,debug=true,naturalnames=true}
\usepackage{graphicx,amssymb,upref}
\usepackage{titling}
\usepackage{xcolor}

\pdfoutput=1

\usepackage[margin=3cm]{geometry}
\usepackage{framed}
\newcommand{\bibliofont}{\footnotesize}

\theoremstyle{definition}
\newtheorem{defn}{Definition}[section]
\newtheorem{remark}[defn]{Remark}

\newtheorem{examp}[defn]{Example}

\newtheorem{question}{Question}[section]

\theoremstyle{theorem}
\newtheorem{prop}[defn]{Proposition}
\newtheorem{cor}[defn]{Corollary}
\newtheorem{thm}[defn]{Theorem}
\newtheorem*{thm*}{Theorem}
\newtheorem*{lem*}{Lemma}
\newtheorem{lem}[defn]{Lemma}

\def\Z{\mathbb{Z}}

\def\wt{\widetilde}
\def\eps{\varepsilon}
\def\ra{\rightarrow}

\def\CP2{\mathbb{CP}^2}

\newcommand{\bigslant}[2]{{\raisebox{.2em}{$#1$}\left/\raisebox{-.2em}{$#2$}\right.}}

\begin{document}
\extrafloats{100}
%\setcounter{section}{8}

%\thispagestyle{empty}
%\title{Diagrams of $\star$-trisections}
%\author{Jos\'e Rom\'an Aranda and Jesse Moeller}
%\date{}
%  %\email{} 
%\maketitle
\begin{center}\Large{Diagrams of $\star$-trisections}\\
\large{Jos\'e Rom\'an Aranda and Jesse Moeller}
\\
{August 2020}
\end{center}
\begin{abstract} 
In this note we provide a generalization for the definition of a trisection of a 4-manifold with boundary. We demonstrate the utility of this more general definition by finding a trisection diagram for the Cacime Surface, and also by finding a trisection-theoretic way to perform logarithmic surgery. In addition, we describe how to perform 1-surgery on closed trisections. The insight gained from this description leads us to the classification of an infinite family of genus three trisections. 
We include an appendix where we extend two classic results for relative trisections for the case when the trisection surface is closed.
\end{abstract}

\tableofcontents
%\setcounter{section}{0}

%%%%%%%%%%%%%%%%%%%%%%%%%%%%%%%%%%%%%%%%%%%
\section{Introduction}
\label{section_intro}
In \cite{trisecting_four_mans}, Gay and Kirby proved that every closed smooth 4-manifold admits a trisection. %. An analogue of the Heegaard splittings in dimension three. 
A trisection of a closed 4-manifold $X$ is a decomposition $X=X_1\cup X_2\cup X_3$ into three 4-dimensional 1-handlebodies so that the pairwise intersections are 3-dimensional 1-handlebodies $X_i\cap X_j$, and the triple intersection is a closed surface $\Sigma=X_1\cap X_2\cap X_3$. 
The genus of the trisection is defined as the genus of $\Sigma$.
%The surface $\Sigma$ is called the trisection surface and its genus is the genus of a trisection. %The minimal genus among all trisection surfaces of $X$ is called the genus of $X$. In the last section of their paper, Gay and Kirby also define trisections of 4-manifolds with boundary called relative trisections.
%When $X$ is a trisected 4-manifold with boundary, the central trisection surface is not necessarily closed. Naturally, the relative trisection of $X$ induces structure on $\partial X$. If the relative trisection surface has non-empty boundary, the trisection restricts to $\partial X$ as an open book decomposition. The definition and diagrammatics of relative trisections in this case can be found in \cite{relative_trisections} and \cite{relative_trisections_2}. When the trisection surface is closed, the binding link is empty and thus the trisection restricts to the boundary of $X$ as a fibration of closed surfaces over $S^1$. 
In recent years, the notion of trisections has been extended to 4-manifolds with several boundaries \cite{relative_trisections_2}, knotted surfaces in 4-manifolds \cite{bridge_trisections_4M} and finitely presented groups \cite{group_trisections}. 
See \cite{trisections_kirby} for an exposition on recent advances in the theory of trisections of 4-manifolds.

The main goal of this paper is to introduce a generalization of trisections of 4-manifolds, called $\star$-trisections, and to develop the diagramatics of this new theory. 
One dificulty of studying trisections of 4-manifolds is the rate at which the genus of the trisections grows under certain operations. % under some operations. 
%This paper is devoted to reduce the trisection genus of such 4-manifolds by using $\star$-trisections.  
Softening the definition of a trisection of a 4-manifold with boundary can potentially reduce this complexity. 
For example, in \cite{trisections_via_lefschetz}, a genus seven trisection of $T^2\times S^2$ was obtained by taking the double of a genus 3 trisection for $T^2\times D^2$. 
In Figure \ref{fig_trisection_T2xS2}, we use a genus one $\star$-trisection for $T^2\times D^2$ to draw a genus four trisection for $T^2\times S^2$. 

We generalize the definition of trisection by relaxing the definition of the 4-dimensional manifolds which make up the pieces of the trisection. The definition of $\star$-trisections can be found in Section \ref{section_all_trisections}. For the interested reader, in Section \ref{standard_star_pieces} we dedicate several remarks, lemmas, and figures to the exposition of these new 4-dimensional pieces which we use to build $\star$-trisections; we discuss two equivalent constructions of these pieces, %modify a previous uniqueness statement for our needs, 
and describe the boundary of a $\star$-trisected 4-manifold. The diagrammatics of these new $\star$-trisections are presented in Section \ref{subsection_trisection_diagrams}.

%One dificulty of obtaining a relative trisection diagram of a given 4-manifold X is the rate at which the genus of the trisection grows. That is, even when the page of the induced open book has small genus and the 4-manifold itself has few handles, the trisection might be of huge genus. Allowing the fiber on the boundary to be closed can potentially reduce this complexity.

%The goal of this paper is to introduce a generalization of trisections of 4-manifolds, called $\star$-trisections, and to develop the diagramatics of this new theory. 
%The utility of this more general definition is apparent when taking complements of embedded submanifolds. 
In Sections \ref{section_curve_complement} and  \ref{section_surface_surgery}, we show how to find $\star$-trisections for the complements of neighborhoods of certain embedded submanifolds. 
Motivated by \cite{trisections_via_lefschetz}, we prove a pasting lemma in Section \ref{section_pasting} which allows us to glue two $\star$-trisected 4-manifolds along connected components of their boundaries. In Section \ref{section_surface_surgery} we use this pasting lemma explicitly to produce a closed trisection diagram for the Cacime Surface. In addition, we use the pasting lemma to decribe how to trisect the Fintushel-Stern knot surgery \cite{FS_knot_surgery} and Logarithmic transforms in the spirit of \cite{round_handles}.

\subsection*{Trisections of genus three}
In \cite{genus_two_std}, Meier and Zupan classified all trisections of genus at most two. 
%proved that only $S^2\times S^2$ is the only closed 4-manifold admitting an irreducible genus two trisection and genus 1 trisections are classified easily enough. 
It is therefore natural to seek a classification for trisections of low genus. 
As a proving ground, in this paper we study an infinite family of genus three trisection diagrams. Consider three rational numbers $\frac{a}{b}$, $\frac{c}{d}$, $\frac{p}{q}$ in reduced form. Let $\alpha_1$ and $\alpha_2$ be the top and middle curves of the left diagram in Figure \ref{fig_farey_trisections_intro} and let $\alpha_3$ be the $\frac{a}{b}$ torus knot in the torus obtained by compressing along $\alpha_1$ and $\alpha_2$. %, making sure the meridians take place to the right of $\alpha_2$.
Take $\alpha$ to be the union of these curves and define $\beta$ and $\gamma$ similarly using $\frac{c}{d}$ and $\frac{p}{q}$, respectively. Observe that %$(\Sigma;\alpha,\beta)$ is a stabilization of a genus one Heegaard splitting. So, when $|ad-bc|\leq 1$, $(\Sigma,\alpha,\beta)$ is a Heegaard diagram for $S^1\times S^2$ or $S^3$. 
$(\Sigma,\alpha,\beta)$ is a Heegaard diagram for $S^1\times S^2$ or $S^3$ whenever $|ad-bc|\leq 1$.
When this condition is satisfied for each pair of fractions, the tuple $(\Sigma;\alpha,\beta,\gamma)$ is a genus three trisection diagram. We call such tuple a Farey diagram $D(\frac{a}{b},\frac{c}{d},\frac{p}{q})$. % or a Farey trisection.
\begin{figure}[h]
\centering
\includegraphics[scale=.35]{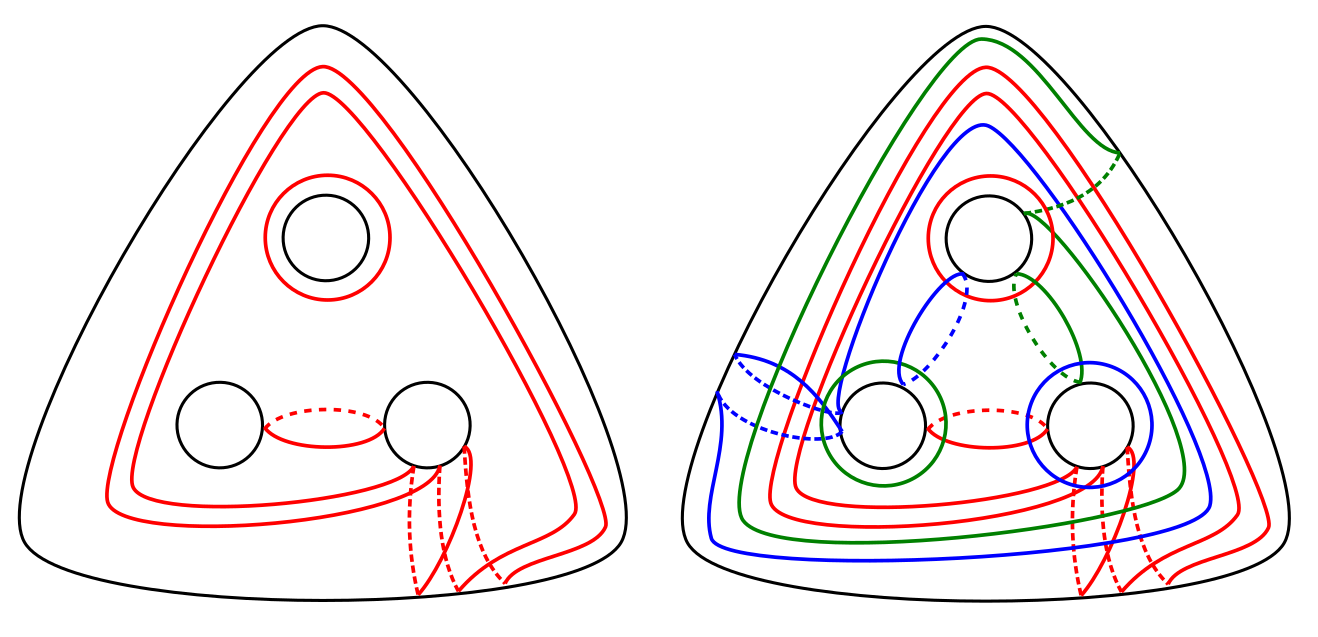}
\caption{The Farey diagram $D(\frac{1}{1}, \frac{1}{2},\frac{2}{3})$.}
\label{fig_farey_trisections_intro}
\end{figure}

The problem of understanding Farey diagrams was proposed during the first day of the 2019 Spring Trisectors Meeting at UGA. By this time, progress had already been made on this problem in \cite{spun_trisections}, where Meier showed that $D(\frac{q}{p},\frac{q}{p},\frac{q}{p})$ is the diagram of a spun lens space $L(p,q)$. Seeking a classification for genus three trisections, he conjectured the following. 
\newtheorem{conj}{Conjecture}
\begin{conj}[Meier \cite{spun_trisections}]\label{spun_conjecture}
Every irreducible 4-manifold with trisection genus three is either the spin of a lens space, or a Gluck twist on a specific 2-knot in the spin of a lens space.
\end{conj}
In Section \ref{section_genus_three}, we show that every Farey diagram is the result of pasting together $\star$-trisections for $S^2\times D^2$ and $X-(S^1\times B^3)$ along their boundaries, where $X\in\{ S^1\times S^3, S^4, \CP2, \overline{\CP2}\}$. 
In particular, Farey diagrams yield trisections for spun lens spaces or reducible 4-manifolds. We can further prove that these diagrams are actually standard. 
Theorem \ref{thm_farey_std}, in conjunction with Meier's results on spun lens spaces, proves Conjecture \ref{spun_conjecture} for the family of Farey trisections.
\newtheorem*{thm:farey_std}{Theorem \ref{thm_farey_std}}
\begin{thm:farey_std}
Let $\{\frac{a}{b},\frac{c}{d},\frac{p}{q}\}\subset\mathbb{Q}\cup\{\frac{1}{0}\}$ with $d(x,y)\leq 1$ for each $x,y\in \{\frac{a}{b},\frac{c}{d},\frac{p}{q}\}$. If at least two of these fractions are distinct, then $D(\frac{a}{b},\frac{c}{d},\frac{p}{q})$ is equivalent to the standard diagram for $T\# S$ where $T\in \lbrace S^4,\CP2,\overline{\CP2}\rbrace$ and $S\in\lbrace S^2\times S^2, S^2\wt \times S^2\rbrace$.
\end{thm:farey_std}

\subsection*{Classic relative trisections}
It is worthy of note that the current work on trisections of 4-manifolds with boundary is restricted to the case of the trisection surface having non-empty boundary. Besides a few remarks in the original trisections paper, not much has been said in the closed case. In order to complete the discussion of the basic theory of trisections of manifolds with boundary, in Appendix \ref{section_classic_diagrams} we offer adaptations of proofs of the main theorems in \cite{relative_trisections} and \cite{relative_trisections_2} by Castro, Gay, and Pinzon-Caicedo. 
We prove that the algorithm which recovers the monodromy $\phi:P\to P$ of the induced circular structure $\partial X=P\times_\phi S^1$ also works in the case where $P$ is a closed surface, and extend the algorithm to obtain a relative trisection from a Kirby diagram of $X$ and a page of an open book decomposition (or fibration over $S^1$) on $\partial X$ within the diagram.

%%%%%%%%%%%%%%%%%%%%%%%%%%%%%%%%%%%%%%%%%%%%%%%%%%%%%%%%%
\textbf{Acknowledgements.}  
The authors of this paper are grateful to the referee and Maggie Miller for their comments on the previous draft. 
The first author would like to thank Jeff Meier and the topology group of the University of Iowa for helpful conversations. The second author would like to thank the Max Planck Institute for Mathematics, and Steve Hamborg, for their hospitality.

%%%%%%%%%%%%%%%%%%%%%%%%%%%%%%%%%%%%%%%%%%%%%%%%%%%%%%%%%%%%%%%%%%%%%%%%%%%%%%%%%%%%%%%%%%%%%%%%%%%%%%%%%%%%%%%%%%%%%%%%%%%%%%%%%%%%%
\section{Trisections of 4-manifolds}\label{section_all_trisections}
This section will be broken into two parts. In the first part, we will review the original definition of a 4-manifold trisection given by Gay and Kirby in \cite{trisecting_four_mans}. We elaborate slightly on the definition by giving an equivalent definition involving 3-dimensional compression bodies as well as by including a description of the induced decomposition of the boundary in the relative case. This elaboration will assist in our exposition of $\star$-trisections. We formally introduce $\star$-trisections in the second part of this section.

\subsection{Original Definitions for Trisections}\label{classic}
A trisection of a closed, connected 4-manifold $X$ is a decomposition of $X$ into three 4-dimensional 1-handlebodies $X=X_1\cup X_2\cup X_3$ such that for each pair $i\neq j$ the intersection $X_i\cap X_j$ is a 3-dimensional handlebody and the common intersection $X_1\cap X_2\cap X_3$ is a closed surface. 
\begin{figure}[h]
\centering
\includegraphics[scale=.6]{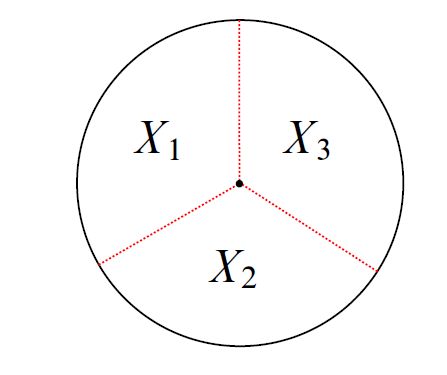}
\caption{A picture two-dimensions lower of a trisection.}
\label{Fig_standard_trisection}
\end{figure}
Here we offer an equivalent definition through the lens of viewing the 4-dimensional pieces in Figure \ref{Fig_standard_trisection} as thickened 3-dimensional handlebodies. Let $C_k$ be a 3-dimensional handlebody with boundary a closed surface of genus $k$, $F_k$. Let $Z_k=[0,1]\times C_k \approx\natural_k \left(S^1\times B^3\right)$. The boundary of $Z_k$ admits a Heegaard splitting $Y_k=\partial Z_k= Y^+_k \cup Y^-_k$ where
\begin{align} Y^+_{k}=\left([1/2,1]\times \partial F_{k}\right)\cup \left(\{1\}\times C_{k}\right) \text{ and } Y^-_{k}=\left(\{0\}\times C_{k}\right) \cup \left([0,1/2]\times F_{k}\right).
\end{align}
Given an integer $g\geq k$, let $Y_k=Y_{k,g}^+\cup Y_{k,g}^-$ be the standard genus $g$ Heegaard splitting of $Y_k$ obtained by stabilizing the genus $k$ Heegaard splitting $g-k$ times. 

\begin{defn}[Trisection of closed 4-manifold]\label{def1}
A trisection of a closed, connected 4-manifold $X$ is a splitting $X=X_1\cup X_2\cup X_3$ and integers $0\leq k,n,g$ with $n\leq k\leq g$ such that each $X_i$ is diffeomorphic to $Z_{k}$ via a diffeomorphism $\varphi_i:X_i \ra Z_{k}$ for which 
\[ \varphi_i(X_i\cap X_{i+1})=Y^+_{k,g} \text{ and } \varphi_i(X_i\cap X_{i-1})=Y^-_{k,g}.\]
\end{defn}

Notice that $Z_k$ is a 4-dimensional 1-handlebody and the pairwise intersections $X_i\cap X_j$ are 3-dimensional handlebodies of genus $g$. This definition is just a slightly more technical rewording of the original presented at the beginning of the section. In \cite{trisecting_four_mans}, Gay and Kirby generalize this definition to allow 4-manifolds which have surface bundles over $S^1$ as boundary. We will again present this definition by presenting the wedges of the trisection as thickened 3-dimensional pieces.

For integers $k,b\geq 0$, let $F_{k,b}$ be a connected orientable surface of genus $k$ with $b$ boundary components. Fix non-negative integers $b,k$ and $n$ with $n< k$ and let $C_{k,b,n}$ denote a 3-dimensional compression body with $F_{k,b}$ as the positive boundary and with $F_{k-n,b}$ as the negative boundary. 
This compression body is built by attaching $n$ 3-dimensional 2-handles to $\{1\}\times F_k\subset [0,1]\times F_k$, yielding a cobordism from $F_{k,b}$ to $F_{k-n,b}$. Now consider the 4-manifold $Z_{k,b,n}:=[0,1]\times C_{k,b,n}$. Part of $\partial Z_{k,b,n}$ is 
\begin{align}
Y_{k,b,n}:=\left( \{0\}\times C_{k,b,n}\right) \cup \left([0,1]\times F_{k,b}\right) \cup \left(\{1\}\times C_{k,b,n}\right),
\end{align}
which has a natural genus $k$ Heegaard splitting into two compression bodies 
\begin{align}
 Y^+_{k,b,n}:=\left([1/2,1]\times F_{k,b}\right)\cup \left(\{1\}\times C_{k,b,n}\right) \text{ and } Y^-_{k,b,n}:=\left(\{0\}\times C_{k,b,n}\right) \cup \left([0,1/2]\times F_{k,b}\right).
 \end{align}
Finally, given any $g\geq k$, let $Y_{k,b,n}=Y^+_{k,b,n,g}\cup Y^-_{k,b,n,g}$ be the genus $g$ Heegaard splitting obtained from the natural genus $k$ splitting by stabilizing $g-k$ times. 

\begin{defn}[Relative trisection]\label{def2}
A trisection of a connected 4-manifold $X$ with non-empty boundary is a splitting $X=X_1\cup X_2\cup X_3$ and integers $0\leq k,b,n,g$ with $n<k\leq g$ such that each $X_i$ is diffeomorphic to $Z_{k,b,n}$ via a diffeomorphism $\varphi_i:X_i \ra Z_{k,b,n}$ for which 
\[ \varphi_i(X_i\cap X_{i+1})=Y^+_{k,b,n,g} \text{ and } \varphi_i(X_i\cap X_{i-1})=Y^-_{k,b,n,g}.\]
\end{defn}
\begin{figure}[h]
\centering
\includegraphics[scale=.13]{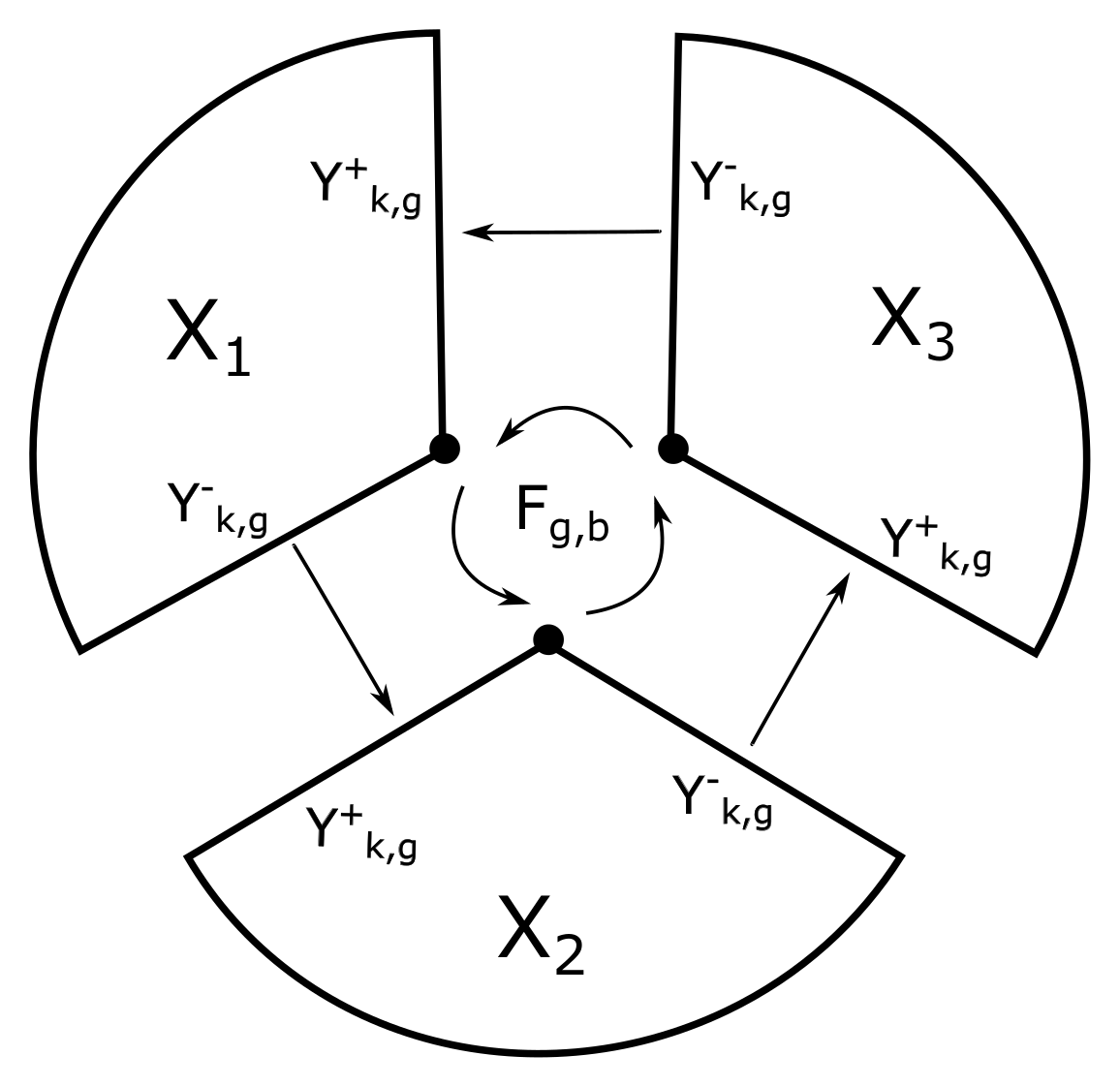}
\caption{A (relative) trisected 4-manifold is built by gluing three standard 4-dimensional pieces along submanifolds of their boundary.}
\label{gluing_maps}
\end{figure}
We can interpret a 3-dimensional handlebody $C_k$ as a cobordism from a closed surface $F_k$ to the empty set. This way, Definition \ref{def2} generalizes Definition \ref{def1}. 
For a detailed discussion on the properties of relative trisections one can see \cite{relative_trisections, relative_trisections_2, trisecting_four_mans}.
The authors of this paper like to think of trisections of 4-manifolds as decompositions into three `standard' pieces $Z_i$ glued along submanifolds of their boundary. 
In this paper we will see some interactions between relative and $\star$-trisections.
In the following subsection we will describe the standard models for the 4-manifold pieces in a $\star$-trisection. 
%%%%%%%%%%%%%%%%%%%%%%%%%%%%%%%%%%%%%%%%%%%%%%%%%%%%%%%%%%%

%%%%%%%%%%%%%%%%%%%%%%%%%%%%%%%%%%%%%%%%%%%%%%%%%%%%%%%%%%
\subsection{$\star$-Trisections of 4-manifolds} \label{section_startrisections}
Here we will define a more general definition of 4-manifold trisection, which we call $\star$-trisection, by first defining analogous 4-dimensional wedges and then gluing these along submanifolds of their boundaries. Informally, these are obtained by gluing two copies of the 4-manifolds $Z_{k,b,n}$ in Subsection \ref{classic}, and removing properly embedded 2-disks in a controlled way.

Let $F_+=F_{k,b}$ and $F_-=\dot\bigcup_{i=1}^{s} F_{k_i,b_i}$ be two orientable surfaces with $k\geq \sum k_i$, $b=\sum b_i$ and $s\geq 0$. $F_+$ is a connected surface and $F_-$ has $s\geq 0$ connected components. 
We will always assume that $F_-$ has no 2-sphere components. Let $C$ be a 3-dimensional connected compression body with positive boundary $\partial_+C=F_+$ and negative boundary $\partial_-C=F_-$. This compression body is obtained by attaching 3-dimensional 2-handles to $F_+\times\{1\} \subset F_+\times [0,1]$, and capping-off the resulting 2-sphere components with 3-handles. This produces a cobordism $C$ from $F_+$ to $F_-$. The 2-handles above are attached along a collection of pairwise disjoint, non isotopic and possibly boundary parallel simple closed curves $\delta \subset F_+$. 
%Notice that, up to diffeomorphism of $F_+$, the smallest set of such loops $\delta$ is determined by the diffeomorphism type of $F_-$. Hence $C$ is determined up to diffeomorphism (preserving the set $F_+$) by $k$, $b_0$, and $\{(k_i,b_i)\}_{i=1}^s$. %, and we write $C=C_{(k,b_0);\{(k_i,b_i)\}}$.
%\begin{figure}[h]
%\centering
%\includegraphics[scale=.5]{Images/gluing_maps.png}
%\caption{\color{red}An example of a model for a compression body with $k=$, $b_0=$, and $\{(k_i,b_i)\}=$. Notice that some $\delta$ curves are redundant.}
%\label{diagrams_S2D2}
%\end{figure}

Let $\wt C$ be a compression body with positive boundary $F_{k,b_0}$. Consider $\wt C^0, \wt C^1\subset \wt C$ two sub-compression bodies spanning $\wt C$ with common part a sub-compression body $\wt C_{all}$; i.e., $\wt C$ is built using a collection of pairwise disjiont simple closed curves $\wt \delta\subset F_{k,b_0}$ such that $\wt \delta=\delta ^0\cup \delta^1$, $\delta_{all}=\delta^0\cap \delta^1$ with $\wt C^i$ determined by $\delta^i$ and $\wt C_{all}$ determined by $\delta_{all}$. %Assume that $\partial_- C_{all}\neq \emptyset$. 
Define 
\begin{align} \wt Z = \wt C^0\times[0,1/2] \bigcup_{\wt C_{all}\times \{1/2\}} \wt C^1\times [1/2,1]. \end{align}

Consider the submanifold of $\partial \wt Z$, 
\begin{align}
\wt Y=\left( \wt C^0 \times \{0\}\right) \cup \left( F_{k,b_0}\times [0,1]\right)\cup  \left(\wt C^{1}\times \{1\}\right).
\end{align}
%\[ \wt Y =  \left(\left(\wt C^{0}-C_{all}\right)\times \{0\}\right) \cup \left( \partial_-C_{all}\times [0,1]\right) \cup \left(\left(\wt C^{1}-C_{all}\right)\times \{1\}\right).\]
Both $\partial \wt Z$ and $\wt Y$ are connected 3-manifolds. Moreover, the surface $F_{k,b_0}\times\{1/2\}$ determines a natural Heegaard splitting of $\wt Y = \wt Y_+ \cup \wt Y_-$ given by 
\begin{align}
\wt Y_-=\left( \wt C^0 \times \{0\}\right) \cup \left( F_{k,b_0}\times [0,1/2]\right) \text{ and } \wt Y_+=\left( F_{k,b_0}\times [1/2,1]\right) \cup \left(\wt C^{1}\times \{1\}\right).
\end{align}
%Notice that $\wt Y$ has as many connected components as $\partial_-C_{all}$. The surfaces of $\partial_-C_{all}\times\{1/2\}$ determine natural Heegaard splittings for the components of $\wt Y = \wt Y_+ \cup \wt Y_-$ given by 
%\[ \wt Y_- =  \left(\left(\wt C^{1/2}-C_{all}\right)\times \{0\}\right) \cup \left( \partial_-C_{all}\times [0,1/2]\right) \text{ and }  \wt Y_+=\left( \partial_-C_{all}\times [1/2,1]\right) \cup\left(\left(\wt C^{1}-C_{all}\right)\times \{1/2\}\right).\]

\begin{figure}[h]
\centering
\includegraphics[scale=.07]{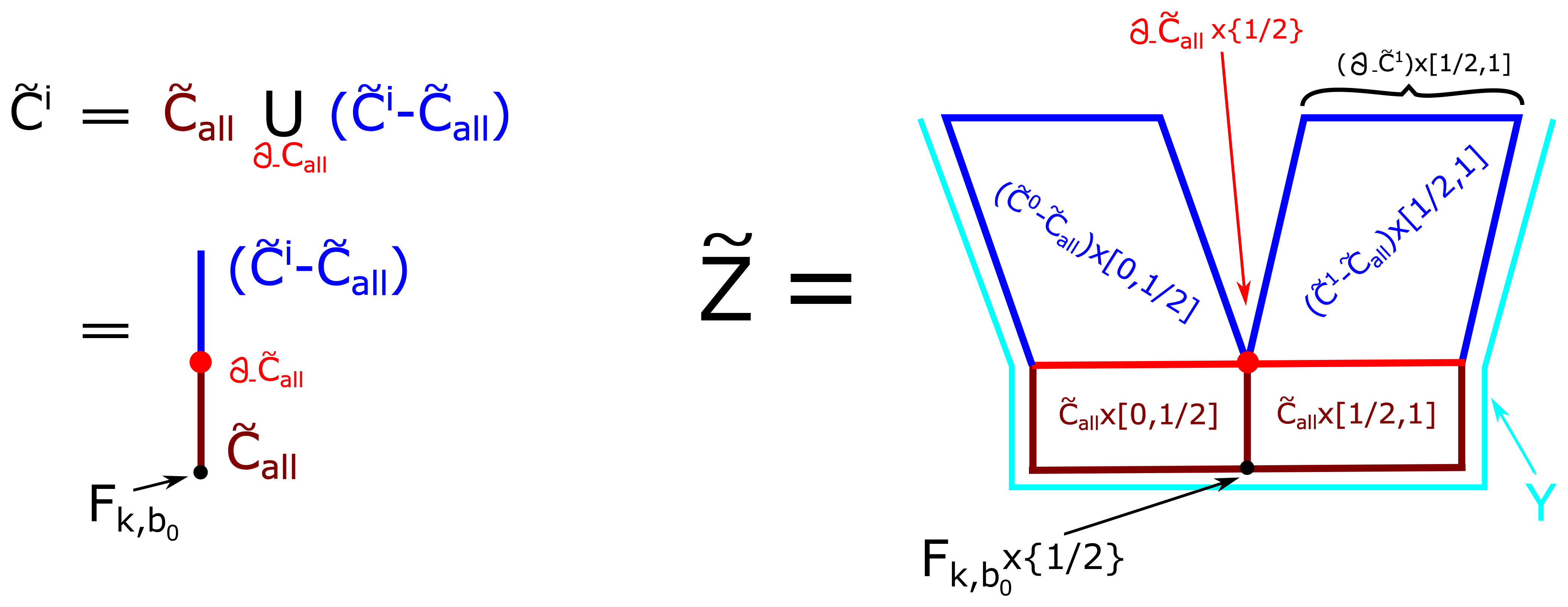}
\caption{A diagram of $\wt Z$ two dimensions down.}
\label{model_Z}
\end{figure}

Define $Z$ to be the 4-manifold obtained from $\wt Z$ by attaching $m\geq 0$ 1-handles with both feet in $\wt Y$. 
Equivalently, $Z$ is obtained by carving a collection of boundary parallel 2-dimensional disks $(D,\partial D)\subset (\wt Z, \wt Y)$.
We will prove in Lemma \ref{lem_unique_trivial_disks} that there is a unique isotopy class of boundary parallel disks in $\wt Z$ with a given boundary. Thus the boundary link $U=D\cap \wt Y$ is an unlink of unknots determining the collection $D$.
Isotope $U$ in bridge position with respect to the Heegaard splitting $\wt Y=\wt Y_- \cup \wt Y_+$; this means that $U\cap \left(\wt C^i\times\{i\}\right)$ is a collection of boundary parallel arcs on each compression body, and $U\cap \left(F_{k,b_0}\times [0,1]\right)$ is a colection of product arcs.

Set $C^i=(\wt C^i\times\{1/2\}) -\eta(U)$, $Y=\wt Y - \eta(U)$ and $Z=\wt Z-\eta(D)$. 
By construction $C^0$ and $C^1$ are compression bodies with positive boundary the surface $F_{k,b}=F_{k,b_0}-\eta(U)$ with $b:=b_0+|U\cap (F_{k,b_0}\times\{1/2\})|$. 
Furthermore, $Y\subset Z$ is given by 
\begin{align}
 Y=Y_{C^0, C^1} := \left( C^{0} \times \{0\} \right) \cup \left( F_{k,b}\times [0,1]\right) \cup \left( C^{1}\times \{1\}\right),
\end{align}
which admits a Heegaard splitting $Y_{C^0,C^1}=Y^+_{C^0, C^1} \cup Y^-_{C^0, C^1}$ as follows
\begin{align*} 
Y^+_{C^0, C^1} :=& \left( F_{k,b}\times [1/2,1]\right) \cup \left( C^{1}\times \{1\}\right)\\
Y^-_{C^0, C^1} :=& \left( C^{0}\times \{0\}\right) \cup \left( F_{k,b}\times [0,1/2]\right)
\end{align*} 
For $g\geq k$, let $Y_{C^0, C^1}=Y^+_{C^0, C^1;g}\cup Y^-_{C^0, C^1;g}$ be the splitting above stabilized $g-k$ times. 
In the next section we will see that the 4-manifold $Z$ can be built from the information of $C^0$, $C^1$ and $C_{all}$, where $C_{all}$ is the comon compression body between $C^0$ and $C^1$ obtained by embedding the loops $\wt \delta_{all}$ into $\Sigma$ and adding some new loops from $U$ (see Remark \ref{remark_constructions_standard_piece}). Thus we write $Z=Z(C^0,C^1,C_{all})$ to emphasize this dependence.

\begin{defn}[$\star$-Trisection] \label{def_star_tris}
A $\star$-trisection of a connected 4-manifold $X$ is a decomposition $X=X_1\cup X_2\cup X_3$ with connected compression bodies $C_{(i)}^0$, $C_{(i)}^1$, $C_{(i),all}$ as above and an integer $g\geq g(\partial_+C_{(i)}^j)$ for $i=1,2,3$, $j=0,1$, such that each $X_i$ is diffeomorphic to $Z(C_{(i)}^0, C_{(i)}^1, C_{(i),all})$ via a map $\varphi_i: X_i\ra Z(C_{(i)}^0, C_{(i)}^1, C_{(i),all})$ for which 
\[ \varphi_i(X_i\cap X_{i+1}) =Y^+_{C_{(i)}^0,C_{(i)}^1;g} \text{ and } \varphi_i(X_{i-1}\cap X_{i}) =Y^-_{C_{(i)}^0,C_{(i)}^1;g}.\] 
The triple intersection is a connected surface $\Sigma$ of genus $g$ with $b\geq 0$ boundary components called the $\star$-trisection surface. 
\end{defn}
\begin{figure}[h]
\centering
\includegraphics[scale=.11]{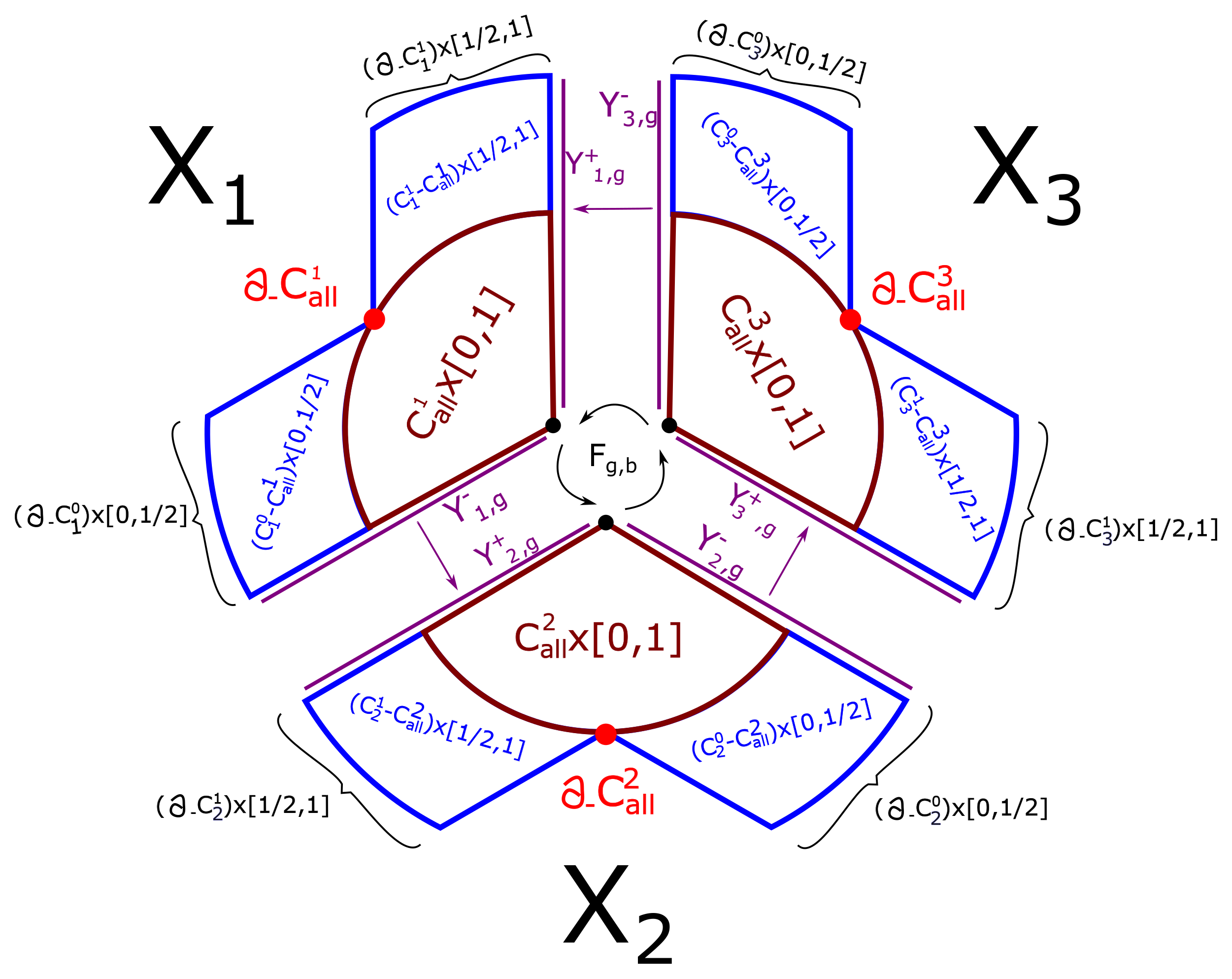}
\caption{A $\star$-trisected 4-manifold $X$ is built by gluing three `standard' 4-dimensional pieces $Z(C_{(i)}^0, C_{(i)}^1, C_{(i),all})$ ($i=1,2,3$) along submanifolds of their boundary. Notice that if for each $i$ the compression bodies $C_{(i)}^0$ and $C_{(i)}^1$ are determined by the same sets of curves, then the extra `blue fins' in the diagram disapear, giving us a foliation of the boundary of $X$ by copies of $\partial_- C_{(1),all}$.}
\label{star_gluing_maps}
\end{figure}
\begin{remark}[Classic trisections]
If all of the compression bodies $C_{(i)}^j$ have empty negative boundary, each $Y^\pm_{C_{(i)}^0,C_{(i)}^1;g}$ is a handlebody and it follows that $X$ is closed. Here, the definition above agrees with the original definition of a trisection when $X$ is closed \cite{trisecting_four_mans}. 
If for all $i$ we have that $C_{(i)}^0$ and $C_{(i)}^1$ are compression bodies determined by the same loops with connected negative boundary or  $Y^\pm_{C_{(i)}^0,C_{(i)}^1;g}$ are determined by the curves in Figure 5 of \cite{relative_trisections_2}, then this decomposition is the same as a trisection of a 4-manifold with boundary in \cite{trisecting_four_mans,relative_trisections,relative_trisections_2}. 
%All prior trisections required that $C_i^0=C_i^1$, inducing an open book decomposition on $\partial X$ with possibly empty binding. 
Such trisections induce an open book decomposition on $\partial X$ with binding a $b$-component link. 
We will refer to all the above as \textbf{classical trisections}, among them we refer to the ones with $\partial X\neq \emptyset$ as \textbf{relative trisections}. We will sometimes wish to distinguish whether or not the binding of the open book $\partial X$ is empty; in these cases we will simply write $b=0$ or $b>0$, referring to the number of boundary components on the relative trisection diagram. 
The most general trisection, or \textbf{$\star$-trisection}, is when $C_{(i)}^0$ is not the same as $C_{(i)}^1$ for some $i$. %Figure \ref{fig_levels_of_generality} contains a diagram from each level of generality mentioned so far. 
The following section is dedicated to present different ways of thinking about the sectors of $\star$-trisections. 
\end{remark}
%%%%%%%%%%%%%%%%%%%%%%%%%%%%%%%%%%%%%%%%%%%%%%%%%%%%%%%%%%%%%%%%%%%%%%%%%

%%%%%%%%%%%%%%%%%%%%%%%%%%%%%%%%%%%%%%%%%%%%%%%%%%%%%%%%%%%%%%%%%%%%%%%%%%
\section{The Standard Pieces} 
\label{standard_star_pieces}
In this section we will describe two equivalent ways of building the `standard' 4-dimensional piece $Z(C_{(i)}^0, C_{(i)}^1, C_{(i),all})$ of a $\star$-trisection. The first construction presents these standard pieces as boundary-conneceted sums of simple 4-dimensional blocks and also establishes the uniqueness of the boundary parallel disks in $\wt Z$, which are carved to form $Z$, up to isotopy. The second construction presents these standard pieces from the handlebody perspective. Remarks \ref{remark_constructions_standard_piece} and \ref{remark_diagrams_standard_piece} summarize the conclusions of this section in the form of four equations (\ref{eq_construction_2_stab}-\ref{equation_construction_4_stab}) and a description of the types of curves needed to build a `standard piece'. 

\textbf{Notation.} Let $\wt C$ be a compression body with positive boundary $F_{k,b_0}$. Consider $\wt C^0, \wt C^1\subset \wt C$ two sub-compression bodies spanning $\wt C$ with fixed common part a sub-compression body $\wt C_{all}$; i.e., there is a collection of simple closed curves $\wt \delta\subset F_{k,b_0}$ determining $\wt C$ such that $\wt \delta=\delta ^0\cup \delta^1$, $\delta_{all}=\delta^0\cap \delta^1$ with $\wt C^i$ determined by $\delta^i$ and $\wt C_{all}$ determined by $\delta_{all}$.
Denote by $\wt F_{all}$ the negative boundary of $\wt C_{all}$. Denote by $\wt C_i=\wt C^i-int(\wt C_{all})$. $\wt C_i$ is a (possibly disconnected) compression body with positive boundary $\wt F_{all}$ determined by the curves\footnote{The curves $\wt \delta_i$ in $F_{k,b_0}$ can also be drawn in $F_{all}$ since they are disjoint from $\wt\delta_{all}$.} $\wt \delta_i = \wt \delta^i -\wt \delta_{all}$. 

We will always choose $\wt \delta_i$ such that the following condition is satisfied: when compressing each component of $\wt F_{all}$ along $\wt \delta_i$, the resulting surface will have at most one sphere component; with equality if and only if the component of $\wt C_i$ is a handlebody.
In particular we will always be able to build $\wt C_{all}$ with no 3-handles unless $\wt F_{all}$ is empty and $\wt C_{all}$ is a handlebody.
\begin{figure}[h]
\centering
\includegraphics[scale=.08]{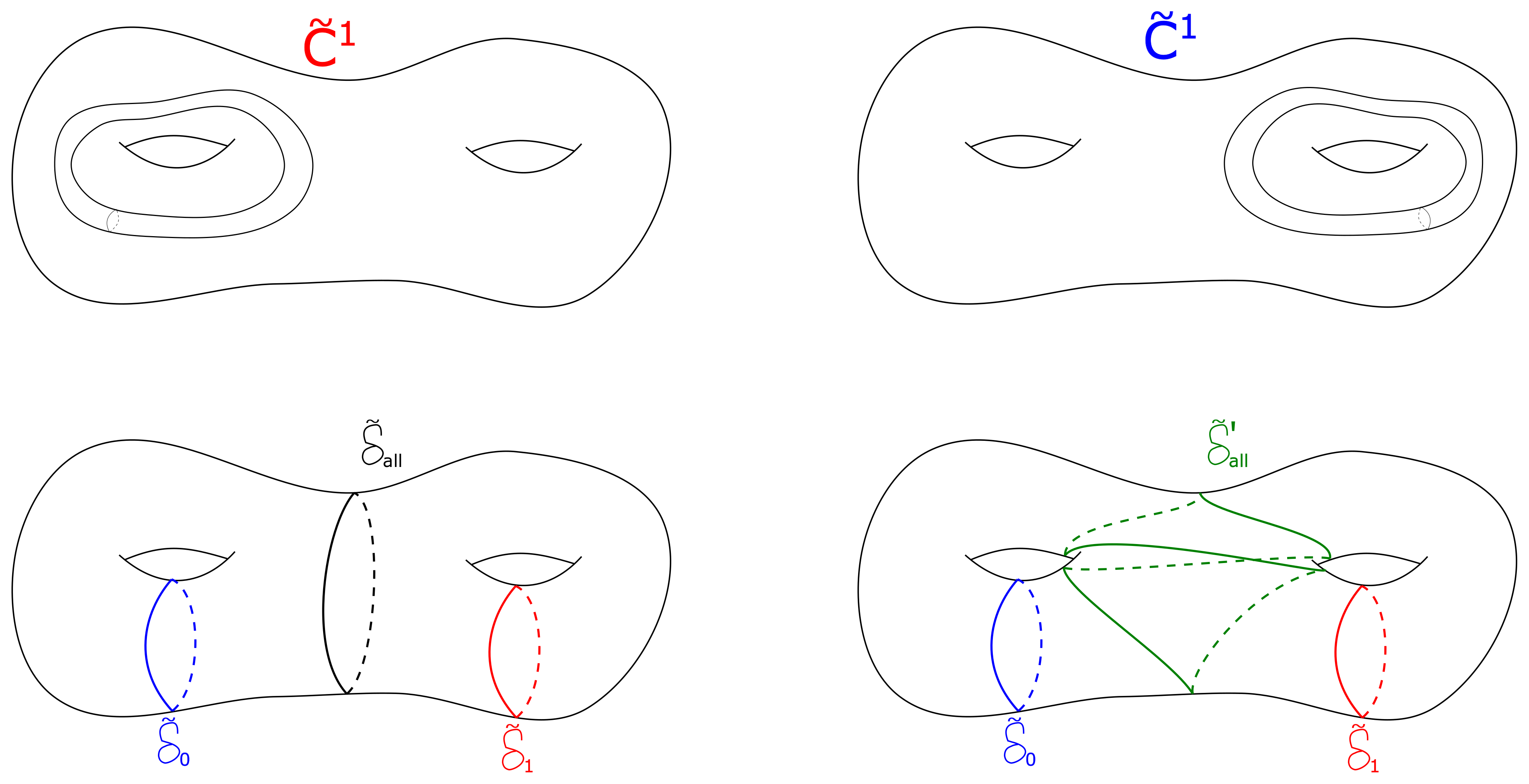}
\caption{Example of loops determining the compression bodies $\wt C^i$, $\wt C_i$, and $\wt C_{all}$. Note in this case $\wt C$ is a handlebody and $\wt\delta_0\cup \wt\delta_1 \cup \wt\delta_{all}$ is not a minimal set of meridians for $\wt C$. The bottom part shows two distinct chioces for loops in $\wt\delta_{all}$.}
\label{Fig_example_loops}
\end{figure}

\begin{remark}[Non-uniqueness of $\wt C_{all}$]
Given two compression bodies $\wt C^0, \wt C^1\subset \wt C$, the compression body $\wt C_{all}$ is not uniquely determined up to isotopy inside $\wt C$. This can be seen by taking the 3-manifold with the Heegaard splitting $W=C^0\cup_{F_{k,b_{0}}} C^1$. Separating loops in $\wt \delta_{all}$ come from separating spheres in $W$, which might not be isotopic depending on the topology of $W$. The bottom part of Figure \ref{Fig_example_loops} shows an example of distinct $\wt \delta_{all}$ sets. 
\end{remark}
%Due to our convention that compression bodies do not have sphererical components on its negative boundary, we only consider $\wt\delta^0$ and $\wt\delta^1$ satisfying following condition: no sphere component of the surface obtained by compressing $F_{k,b}$ along $\wt\delta_{all}$ contains loops in $\wt\delta_i$ for $i=1,2$.  

\subsection{Construction 1} \label{subsection_Construction_1}
Using Figure \ref{model_Z}, we can convince ourselves that $\wt Z$ can be built from $\wt C_{all}\times [0,1]$ by gluing collars of the subcompression bodies $C_i$. More precisely, 
\begin{align}\label{eq_0}
\wt Z \text{ } %:=\wt Z_{C^0,C^1} %:= & \wt C^0\times[0,1/2] \bigcup_{\wt C_{all}\times \{1/2\}} \wt C^1\times [1/2,1] \\ 
= \text{ } & \wt C_0 \times [0,1/2] \bigcup_{\wt F_{all}\times[0,1/2]} \wt C_{all}\times [0,1] \bigcup_{\wt F_{all}\times[1/2,1]}\wt C_1 \times [1/2,1].
\end{align}
We think of the product regions $\wt C_0 \times [0,1/2]$ and $\wt C_1\times [1/2,1]$ as `fins' attached to $\wt C_{all}\times[0,1]$, with the interval directions being horizontal in Figure \ref{model_Z}. Imagine pulling $\wt F_{all}\times \{1/2\}$ `up' in order to horizontally align the fibers $\wt C_i\times \{pt\}$ of the fins (see Figure \ref{isotopy_model_Z}). This makes the interval directions for the product regions $\wt C_0 \times [0,1/2]$ and $\wt C_1\times [1/2,1]$ now vertical in the figure. 
Such isotopy only affects points in $\wt Z$ near $\wt F_{all}\times[0,1]\subset \wt C_{all}\times [0,1]$ so it can be chosen to be the identity in $\left(\wt C_{all}-\eta(\wt F_{all})\right)\times[0,1]$. Thus $\wt Z$ is diffeomorphic to the union 
\begin{align}
\wt Z\text{ }\approx \text{ }\left(\left(\wt C_0\cup_{\wt F_{all}} \wt C_1\right)\times [0,1] \right)\bigcup\wt C_{all}\times [0,1],
\end{align}
where we glue a neighborhood of the surface $\wt F_{all}\times \{0\}$ in $\left(\wt C_0\cup_{\wt F_{all}} \wt C_1\right)\times \{0\}$ with the product $\wt F_{all}\times [0,1]$ in $\wt C_{all}\times [0,1]$.

\begin{figure}[h]
\centering
\includegraphics[scale=.07]{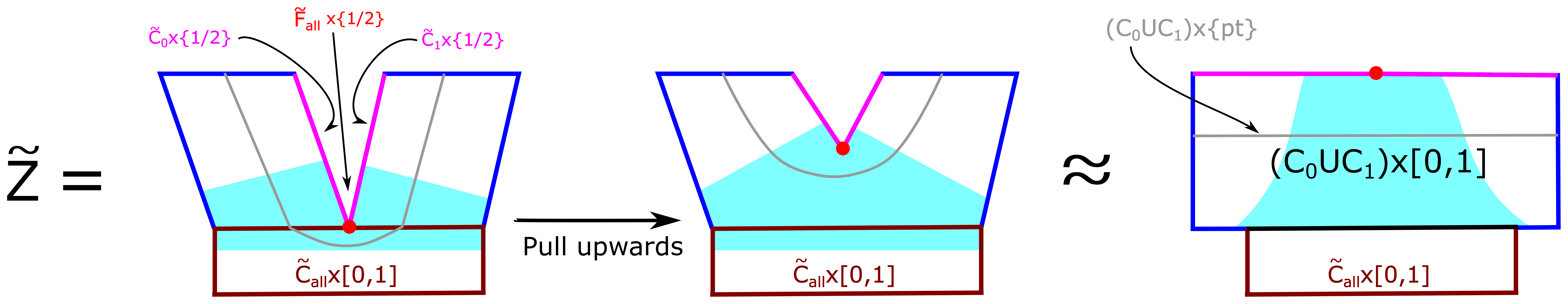}
\caption{How to pull $\wt F_{all}\times \{1/2\}$ `up'. Notice that the isotopy has support in a subset of $\wt Z$ diffeomorphic to $\wt F_{all}\times E$, where $E$ is the shaded disk.}
\label{isotopy_model_Z}
\end{figure}

Recall that a connected compression body $C$ can be built from its negative boundary by adding $l=\frac{1}{2}\left(\chi(\partial_-C)-\chi(\partial_+ C)\right)$ 1-handles to $\partial_-C\times \{1\}\subset \partial_-C\times [0,1]$. If $\partial_-C=\emptyset$, one must add $(l+1)$ 1-handles to a 3-ball (0-handle). Thus, if $\wt F_{all}\neq \emptyset$ then $\wt Z$ is obtained by attaching 1-handles to $\left(\wt C_0\cup_{\wt F_{all}} \wt C_1\right)\times [0,1]$ along $\left(\wt C_0\cup_{\wt F_{all}} \wt C_1\right)\times \{0\}$.
If $\wt F_{all}$ is empty, then $\wt C_{all}$ is a handlebody, $\wt C_0$ and $\wt C_1$ are empty, and $\wt Z$ is a 4-dimensional handlebody. We can safely conclude that 
\begin{align}\label{equation_1}
\wt Z\text{ } \approx \text{ }(\text{0-handle}) \cup\big((\wt C_0\cup_{\wt F_{all}}\wt C_1)\times [0,1]\big) \cup\left(|\delta_{all}| \text{ 1-handles}\right).
\end{align}

\begin{remark}\label{remark_eq_1}
Eventhough the 3-manifold $\wt C_0\cup_{\wt F_{all}}\wt C_1$ might be disconnected\footnote{In fact $\wt C_0\cup_{\wt F_{all}}\wt C_1$ has as many components as $|\wt F_{all}|$.}, each component of $\left(\wt C_0\cup_{\wt F_{all}}\wt C_1\right)\times[0,1]$ has connected 3-manifold boundary. Thus, up to diffeomorphism, each 1-handle in Equation \ref{equation_1} corresponds to a boundary connected sum between two components or with a copy of $S^1\times B^3$. 
\end{remark}

\begin{remark}
Suppose that all of the components of $\partial_-\wt C_0$ and $\partial_-\wt C_1$ are surfaces with boundary. By construction of $\wt\delta^0$ and $\wt\delta^1$, one can show that $\wt C_0 \cup_{\wt F_{all}} \wt C_1$ is also a 3-dimensional handlebody. Therefore, both the standard piece $\wt Z$ and $Z$ are 4-dimensional 1-handlebodies in this specific case. Classical relative trisections satisfy this condition. 
\end{remark}

The following lemma is a refinement of Equation \ref{equation_1}.
\begin{lem}\label{lem_eq_1}
$\wt Z$ is diffeomorphic to a boundary connected sum of a 4-dimensional 1-handlebody with the boundary connected sum of the trivial bundles $F\times D^2$, where $F$ runs through the components of the negative boundary of $\wt C$.
\end{lem}
\begin{proof}
We use the notation described at the begining of the subsection.
Remark \ref{remark_eq_1} allows us to only look at the connected components of $\wt C_0 \cup_{\wt F_{all}} \wt C_1$, which are in correspondance with $|\wt F_{all}|$. We then suppose that $\wt F_{all}$ is a connected non-empty surface.

The fact that $\wt C_{all}$ is the common sub-compression body between $\wt C^0$ and $\wt C^1$ implies that there are no essential simple closed curves in $\wt F_{all}$ bounding disks in both $\wt C_0$ and $\wt C_1$. 
%Haken's Lemma implies that the 3-manifold $\wt C_0 \cup_{\wt F_{all}} \wt C_1$ is irreducible\footnote{Every embedded 2-sphere bounds a 3-ball.}. 
In particular, if both collections of curves $\wt \delta_i$ are non-empty, then the result of compressing $\wt F_{all}$ along $\wt \delta_0 \cup \wt \delta_1$ contains no sphere components, and the compression body $\wt C-int(\wt C_{all})$, determined by $\wt\delta_0 \cup \wt \delta_1$ in $\wt F_{all}$, has non-empty negative boundary. In other words, we do not need 3-dimensional 3-handles to build $\wt C - int(C_{all})$ using the loops $\wt \delta_0 \cup \wt \delta_1$. %This last argument follows from the fact that such sphere $S$ must be obtained by compressing along curves from both $\wt \delta_0$ and $\wt \delta_1$. A belt circle in $S$ separating the disks from both sets of curves will lift to a circle in $\wt F_{all}$ bounding a disk in $\wt C_0$ and $\wt C_1$ simultaneously, producing a contradiction.

Suppose that both $\wt \delta_0$ and $\wt \delta_1$ are non-empty. The observation above tells us that no 3-handles are needed to build each $\wt C_i$. Thus the 3-manifold $\wt C_0\cup \wt C_1$ is built from $\wt F_{all}\times [0,1]$ by attaching 2-handles along $\wt\delta_i\times\{i\}\subset \wt F_{all}\times\{i\}$ for $i=0,1$. 
This implies that $\left(\wt C_0\cup \wt C_1\right)\times [0,1]$ can be built from $\wt F_{all}\times D^2 \approx \wt F_{all}\times [0,1]^2$ by attaching 4-dimensional 2-handles along $\wt \delta_i \times \{(-1)^i\} \subset \wt F_{all}\times \{(-1)^i\}$ for $i=0,1$ with framing given by the fiber surface. Since $\wt\delta_0$ and $\wt \delta_1$ are disjoint in $\wt F_{all}$, we can isotope the attaching regions of the 2-handles in $\wt F_{all}\times S^1$ to lie in the same fiber $\wt F_{all}\times \{1\}$. Using the previous paragraph we get that $\left(\wt C_0\cup \wt C_1\right)\times [0,1]$ is diffeomorphic to $\left(\wt C-int(\wt C_{all})\right)\times [0,1]$. 
Notice that the same conclusion holds if one of $\wt \delta_i$ is empty.

To end, recall that $\widehat C=\wt C-int(\wt C_{all})$ is a compression body satisfying $\partial_+ \widehat C=\wt C_{all}$ and $\partial_- \widehat C=\partial_-\wt C$. 
If $\partial_- \widehat C$ is empty, then $\widehat C$ is a handlebody and $\wt Z\approx \widehat C\times [0,1]$ is a 1-handlebody.
On the other hand, if $\partial_- \widehat C\neq \emptyset$ then $\widehat C$ is built from $\left(\partial_- \widehat C\right)\times [0,1]$ by attaching 1-handles along $\left(\partial_- \widehat C\right)\times \{1\}$ and $\wt Z$ has the right diffeomorphism class.
\end{proof}

In order to build a standard 4-dimensional piece for a $\star$-trisection, we remove boundary parallel disks from the 4-manifold $\wt Z$. We want such disks to be defined solely by their boundary. This is the case when $\wt Z$ is a 4-ball \cite{liv_surfaces_bounding_unlink} or a 1-handlebody \cite{bridge_trisections_4M}. The following technical lemma proves uniqueness for any $\wt Z$. 

\begin{lem}\label{lem_unique_trivial_disks}
Let $\wt Z$ as above, and let $L$ be an unlink in $\partial \wt Z$. Up to isotopy fixing $L$, the unlink $L$ bounds a unique collection of simultaneously boundary parallel disks in $\wt Z$. 
\end{lem}
\begin{proof} 
The argument in the third paragraph of the proof of Lemma 8 of \cite{bridge_trisections_4M} proves the following statement: ``Suppose $A_1$ and $A_2$ are 4-manifolds with connected boundary satisfying (1) $A_i$ has unique trivial disks (up to isotopy relative their boundary), and (2) any embedded sphere in $\partial A_i$ bounds a properly embedded 3-ball in $A_i$. Then $A_1 \natural A_2$ has unique trivial disks.'' 
Also in Lemma 8 of \cite{bridge_trisections_4M} is proven that 1-handlebodies satisfy properties (1) and (2). Hence, using the decomposition of Lemma \ref{lem_eq_1}, it is enough to check uniqueness of trivial disks for $F \times D^2$ where $F$ is a connected closed surface of positive genus. 

Let $F\neq S^2$ be a closed surface and let $L$ be an unlink in $F\times S^1$. Suppose $D$ and $D'$ are two collections of properly embedded trivial disks in $F\times D^2$ with boundary equal to $L$. We want to show that they are isotopic relative to $L$. There exist a collection of disks $D_*$ (resp. $D'_*$) in $F\times S^1$ isotopic to $D$ (resp. $D'$) via an isotopy fixing $L$. Perturb $D_*$ and $D'_*$ to intersect transversely (away from $L$) in a finite set of loops. 
Suppose $int(D_*)\cap int(D'_*)\neq \emptyset$. Let $E\subset D_*$ be a disk with interior disjoint from $D'_*$ and boundary a loop in $int(D_*)\cap int(D'_*)$. Let $E'$ be the disk in $D'_*$ with same boundary as $E$. Since $F\times S^1$ is irreducible, the embedded 2-sphere $E\cup E'$ bounds a 3-ball $B\subset F\times S^1$. Push the interior of $B$ into $int(F\times D^2)$ to make it disjoint from $D_*$ and $D'_*$, we can do this since $D_*\cup D'_*$ lies in $F\times S^1$. Isotope $E'$ through $B$ to remove the intersection loop $E\cap E' \subset int(D_*)\cap int(D'_*)$. We can perform this type of isotopy until $int(D_*)\cap int(D'_*)=\emptyset$. Again, the sphere $D_*\cup D'_*$ bounds a collection of 3-balls in $F\times S^1$ which we can use (as before) to isotope $D_*$ onto $D'_*$ as desired.
\end{proof} 

%\begin{cor}
%The diffeomorphism class of the 4-dimensional pieces $Z$ depend on the following data: (1) the compression bodies $\wt C^0$, $\wt C^1$, $\wt C_{all}$, and (2) the unlink $U$ in $\wt C^0 \cup_{\wt F_{all}} \wt C^1$. 
%\end{cor}

%%%%%%%%%%%%%%%%%%%%%%%%%%%%%%%%%%%%%%%%%%%%%%%%%%%%%%%%%%%%
\subsection{Construction 2} \label{subsection_Construction_2}
Recall that a 3-dimensional compression $C$ body can be built from its positive boundary $F$ by attaching 2-handles along $F\times\{1\}\subset F\times[0,1]$ and 3-handles along the resulting spherical components. 
%SHOULD i ADD THIS.??? 
{Thus $C\times [0,1]$ is obtained from $F\times [0,1]^2$ by attaching 4-dimensional 2-handles along loops in $F\times\{1\}\times \{1/2\}$ with framing given by the interval direction, and 4-dimensional 3-handles along the spheres obtained by compressing $F\times \{1\}\times\{1/2\}$ by the 2-handles. }

Equation \ref{eq_0} of the previous subsection implies that $\wt Z$ can be built from $\wt C_{all}\times [0,1]$ by attaching 4-dimensional 2-handles along the loops $\wt \delta_0\times\{1/4\}\subset \wt F_{all}\times\{1/4\}$ and $\wt \delta_1\times\{3/4\}\subset \wt F_{all}\times\{3/4\}$ with framing given by the surface, and attaching 4-dimensional 3-handles along any 2-sphere obtained from compressing $\wt F_{all} \times \{1/4+i/2\}$ along $\wt \delta_i\times\{1/4+i/2\}$.

\textbf{Suppose $\wt F_{all}$ is non-empty.} We can build $\wt C_{all}\times[0,1]$ as above using the surface $F_{k,b_0}$, the curves $\wt \delta_{all}$, and no 3-handles. Furthermore, we can `double' the 2-handles at the expense of adding 4-dimensional 3-handles as follows
\begin{align}\label{eq_C_all}
\wt C_{all}\times[0,1]\text{ } \approx&\text{ } F_{k,b_0}\times [0,1]^2 \cup \left( \text{2-handles } \wt \delta_{all}\times\{1\}\times\{1/4,3/4\}\right) \cup \left(\text{3-handles}\right).
\end{align}
The cores 2-spheres of the 3-handles in the previous equation are built by connecting the core disks of the 2-handles $\wt \delta_{all}\times\{1\}\times\{1/4,3/4\}$ with the annuli $\wt\delta_{all}\times\{1\}\times[1/4,3/4]$.
For $i=0,1$, $F_{k,b_0}\times\{1\}\times\{1/4+i/2\}$ compresses to a copy of $\wt F_{all}$. We can then glue copies of $\wt C_i\times [0,1]$ to $\wt C_{all}\times[0,1]$ in the form of 4-dimensional 2-handles and 3-handles as in the previous paragraph.
Therefore, via Equation \ref{eq_0} we can conclude
\begin{align}\label{eq_construction_2}
\wt Z \approx& F_{k,b_0}\times [0,1]^2 \cup \left( \text{2-handles } \wt \delta^0\times\{1\}\times\{1/4\} \text{ and } \wt \delta^1\times\{1\}\times\{3/4\}\right) \cup \left(\text{3-handles}\right),
\end{align}
 where the 3-handles are attached along the sphere components obtained from compressing $F_{k,b_0}\times\{1\}\times\{1/4+i\}$ along $\wt \delta^i\times\{1\}\times\{1/4+i/2\}$, and the spheres built from the common curves in $\wt \delta_{all}$.

%In conclusion, one can build $\wt Z$ as follows: start from $F_{k,b_0}\timex D^2$ and attach copies of $\wt C^0 \times I_0$ and $\wt C^1\times I_1$ along $F_{k,b_0}\times I_i$ where $I_0, I_1$ are disjoint intervals in $S^1=\partial D^2$, then attach 3-handles along spheres obtained by curves bounding disks in both $\wt C^0$ and $\wt C^1$.
\begin{remark} \label{remark_if_F_all_empty}
If $\wt F_{all}=\emptyset$, in order to build $\wt C_{all}$ you must add one final 3-dimensional 3-handle after attaching 2-handles to $F_{k,b_0}\times[0,1]$ along $\wt \delta_{all}\times\{1\}$. This final 3-handle becomes a 4-dimensional 3-handle when forming the product $\wt C_{all}\times[0,1]$  (see Figure \ref{fig_C_allxI}), and can be `doubled' at the expense of adding one 4-dimensional 4-handle (in the form of a 3-4-cancelling pair).
In this situation, $\wt C_0$ and $\wt C_1$ are both empty, $F_{k,b_0}=F_{k}$ is a closed surface, and $\wt C=\wt C_{all}$ is a 3-dimensional handlebody. Hence $\wt Z = \wt C_{all}\times[0,1]$ is a 4-dimensional 1-handlebody with boundary admitting a Heegaard splitting of the form $\wt C_{all}\cup_{F_{k}} \wt C_{all}$. 
The resulting model of $\wt Z$ is the standard 4-dimensional piece of a `classic' trisection of a closed, connected 4-manifold.
\begin{figure}[h]
\centering
\includegraphics[scale=.07]{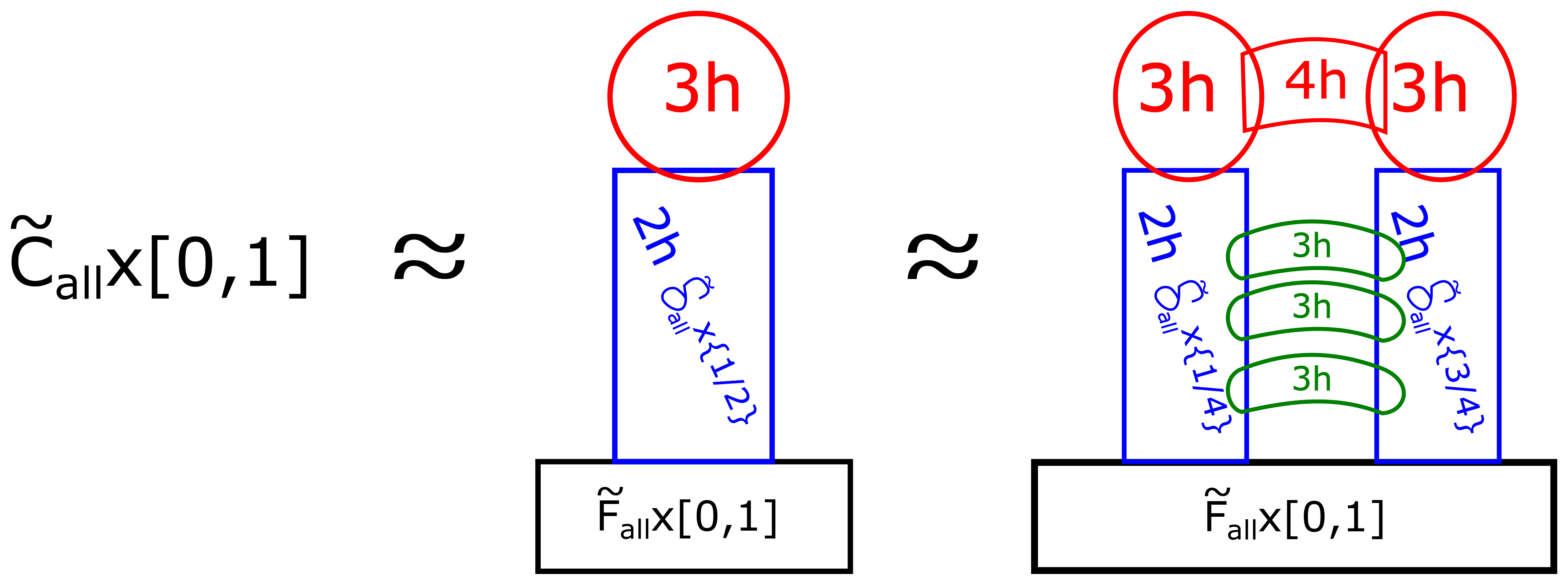}
\caption{A schematic of $\wt C_{all}\times[0,1]$ after `doubling' its 2-handles and 3-handle. In this paper we will normally find $\star$-trisections with $\wt F_{all}\neq \emptyset$. In such cases, we don't add the 3-handles and 4-handle in red in the figure and obtain the description in Equation \ref{eq_C_all}.}
\label{fig_C_allxI}
\end{figure}
\end{remark}

%%%%%%%%%%%%%%%%%%%%%%%%%%%%%%%%%%%%%%%%%%%%%%%%%%%%%%%%%%%%%%%%%%%%%
\subsection{The unlink $U$ and the standard piece $Z$}\label{subsection_unlink_U}

We are ready to discuss the 4-manifold $Z$ in detail. 
We keep using the same notation as before.

Let $\wt Y$ be the 3-manifold given by the Heegaard splitting $\wt Y=\wt C^1 \cup_{F_{k,b_0}} \wt C^0$, and let $U\subset \wt Y$ be an unlink of unknots.
$\wt Y$ embeds into $\partial \wt Z$ as 
\[\wt Y\approx\left( \wt C^0 \times \{0\}\right) \cup \left( F_{k,b_0}\times [0,1]\right)\cup  \left(\wt C^{1}\times \{1\}\right)\subset \partial \wt Z.\]
Lemma \ref{lem_unique_trivial_disks} states that $U \subset \partial \wt Z$ bounds a unique\footnote{Up to isotopy relative to the boundary.} collection of properly embedded disks $D\subset \wt Z$ that can be isotoped (fixing $U$) into embedded disks in $\partial \wt Z$. 
Define $Z$ to be a copy of $\wt Z$ after removing an open neighborhood of $D$. The facts that the disks are parallel to the boundary and $\wt Y$ is connected imply that $Z$ is diffeomorphic to $\wt Z \natural \left(\natural_{|U|} S^1\times B^3 \right)$.
In what follows, we will see that the 4-manifold $Z$ can be built as in Subsection \ref{subsection_Construction_2} using a new pair of compression bodies.
%Lemma \ref{lem_contruction_2_for_Z} states that the 4-manifold $Z$ can be build as in Subsection \ref{subsection_Construction_2} using a new pair of compression bodies which we explain next.

Isotope $U$ in $\wt Y$ to be in bridge position with respect to the given Heegaard splitting. This means that for $i=0,1$ the intersection $U\cap \wt C^i$ consists of properly embedded arcs in $\wt C^i$ which can be projected to embedded arcs in $F_{k,b_0}=\partial_+ \wt C^i$. 
We will refer to such projected arcs in $F_{k,b_0}$ as shadows of $Y\cap \wt C^i$. 
The number of arcs in $U\cap \wt C^0$ is called the number of bridge of $U$. Notice that $U\cap F_{k,b_0}$ is a set of $2|U\cap \wt C^0|$ points. 

Let $Y$ be the complement of $U$ in $Y$ and $C^i:= \wt C^i -\eta (U)$. The intersection $C^0\cap C^1$ becomes a copy of $F_{k,b_0}$ with $2|U\cap \wt C^0|$ open disks removed. Furthermore, each $C^i$ is a compression body with positive boundary the surface $F_{k,b}$, where $b=b_0+2|U\cap \wt C^0|$; and $Y=C^0\cup_{F_{k,b}}C^1$ is a Heegaard splitting.

Since $U$ is an unlink of unknots, we can think of $Y$ as the connected sum of $\wt Y$ with the complement of an unlink $U'$ in $S^3$. Using Haken's lemma we can choose the connected sum sphere $S$ to intersect $F_{k,b}$ in one loop. Thus the loops in $F_{k,b_0}\subset \wt Y$ bounding disks in $\wt C^i$ can be identified with loops in $F_{k,b}$ bounding disks in $C^i$ on one side of such separating sphere.
%In conclusion, we can find sets of curves $\delta^i\subset F_{k,b}$ determining the compression body $C^i$ so that $\wt \delta^i \subset \delta^i$.
We can further use Haken's lemma to find pairwise disjoint spheres $S_*$, away from $S$, decomposing $S^3- \eta(U')$ into a connected sum of complements of unknots in $S^3$ in bridge position, satisfying that each component of $S_*$ intersects $F_{k,b}$ in one loop.

In conclusion, we can find sets of curves $\delta^i\subset F_{k,b}$ determining the compression body $C^i$ so that the loops $\wt \delta^i$ and $\left(S_*\cap F_{k,b}\right)$ are contained in $\delta^i$. Define $\delta_{all}=\wt \delta_{all} \cup \left( S_* \cap F_{k,b}\right)$. % and $C_{all}$ the compression body with positive boundary $F_{k,b}$ determined by $\delta_{all}$. The negative boundary of $C_{all}$ will be denoted by $F_{all}$.
Observe that the curves in $\delta^i$ consist in the loops in $\wt \delta^i$, together with the new loops in $\left(S_*\cap F_{k,b}\right)$ and loops obtained from the complement of each unknot component of $U$. 
The proof of the following lemma (Figure \ref{fig_perturbation}) shows us how complicated these new loops can be.

\begin{lem} \label{lem_contruction_2_for_Z}
The 4-manifold $Z$ can be built as in Equation \ref{eq_construction_2} with $F_{k,b}$ and the loops $\delta^0$, $\delta^1$ and $\delta_{all}$. More precisely, 
\begin{align}\label{eq_construction_2_final}
Z\quad  \approx&\quad  F_{k,b}\times [0,1]^2 \cup \left( \text{2-handles } \delta^0\times\{1\}\times\{1/4\} \text{ and }  \delta^1\times\{1\}\times\{3/4\}\right) \cup \left(\text{3-handles}\right),
\end{align}
 where the 3-handles are attached along the sphere components obtained from compressing $F_{k,b}\times\{1\}\times\{1/4+i\}$ along $\delta^i\times\{1\}\times\{1/4+i/2\}$ for $i=0,1$, and the spheres built from the common curves in $\delta_{all}$.
\end{lem}
\begin{proof} 
Suppose first that each component of $U$ has bridge number one. Since $U$ is an unlink, we can find a set of disks $D_*\subset Y$ bounded by $U$ intersecting $F_{k,b}$ in a single arc per disk in $D_*$. The spheres $S_*$ in $Y$ can be chosen to be $\partial \eta (D_*)$. 
In this case, the curves determining the compression bodies are given by $\delta^i=\wt \delta \cup \left( S_* \cap F_{k,b}\right)$. Using such curves, the right hand side of Equation \ref{eq_construction_2_final} is diffeomorphic to $\wt Z \natural \left(\natural_{|U|}S^1\times B^3\right)\approx Z$, which is what we want. One way to see this is to first note that if $c\subset \delta_{all}$ is a separating curve decomposing $F_{k,b}$ as $F_{k_1,b_1}\#_c F_{k_2,b_2}$, then the 4-manifold from the right hand side decomposes as a boundary connected sum of 4-manifolds built in the same way from the surfaces $F_{k_j,b_j}$ ($j=1,2$); and then check that the 4-manifold obtained from an annulus $F_{0,2}$ and $\delta^0=\delta^1=\delta_{all}=\emptyset$ is $S^1\times B^3$.

We now proceed with the general case. First recall that any bridge position of an unlink of unknots in a 3-manifold is always perturbed. This result was proven for $S^3$ in \cite{Otal}, and for any 3-manifold in \cite{hayashi_unknot}.
A perturbation is a local isotopy of a link in bridge position around the Heegaard surface shown in Figure \ref{fig_perturbation}. It increases the bridge number by one. Figure \ref{fig_perturbation} describes how the complement of the link changes after a perturbation, and shows us how the $\delta^i$ sets increase by one loop. The main observation is that the 4-manifolds built from the right hand side of Equation \ref{eq_construction_2_final} are diffeomorphic because the operation of removing two disks to $F_{k,b}$ and adding two curves (not in $\delta_{all}$) corresponds to add two 4-dimensional 1-2-cancelling pairs of handles (See Figure \ref{fig_perturbation}). 

Hence, for any bridge position of $U$ in $\wt Y$, after fixing a set of spheres $S_*$, we can slide the curves of $\delta^i-\delta_{all}$ over the curves in $\delta^i$ until `detecting' a pair of curves intersecting like in Figure \ref{fig_perturbation}. This pattern determines a perturbation of $U$ which we undo to decrease the bridge number, without changing the resulting 4-manifold. We can repeat this procedure until each component of $U$ has bridge number one, in which case the result is known.
Therefore, for any bridge position of $U$, Equation \ref{eq_construction_2_final} holds.
\begin{figure}[h]
\centering
\includegraphics[scale=.2]{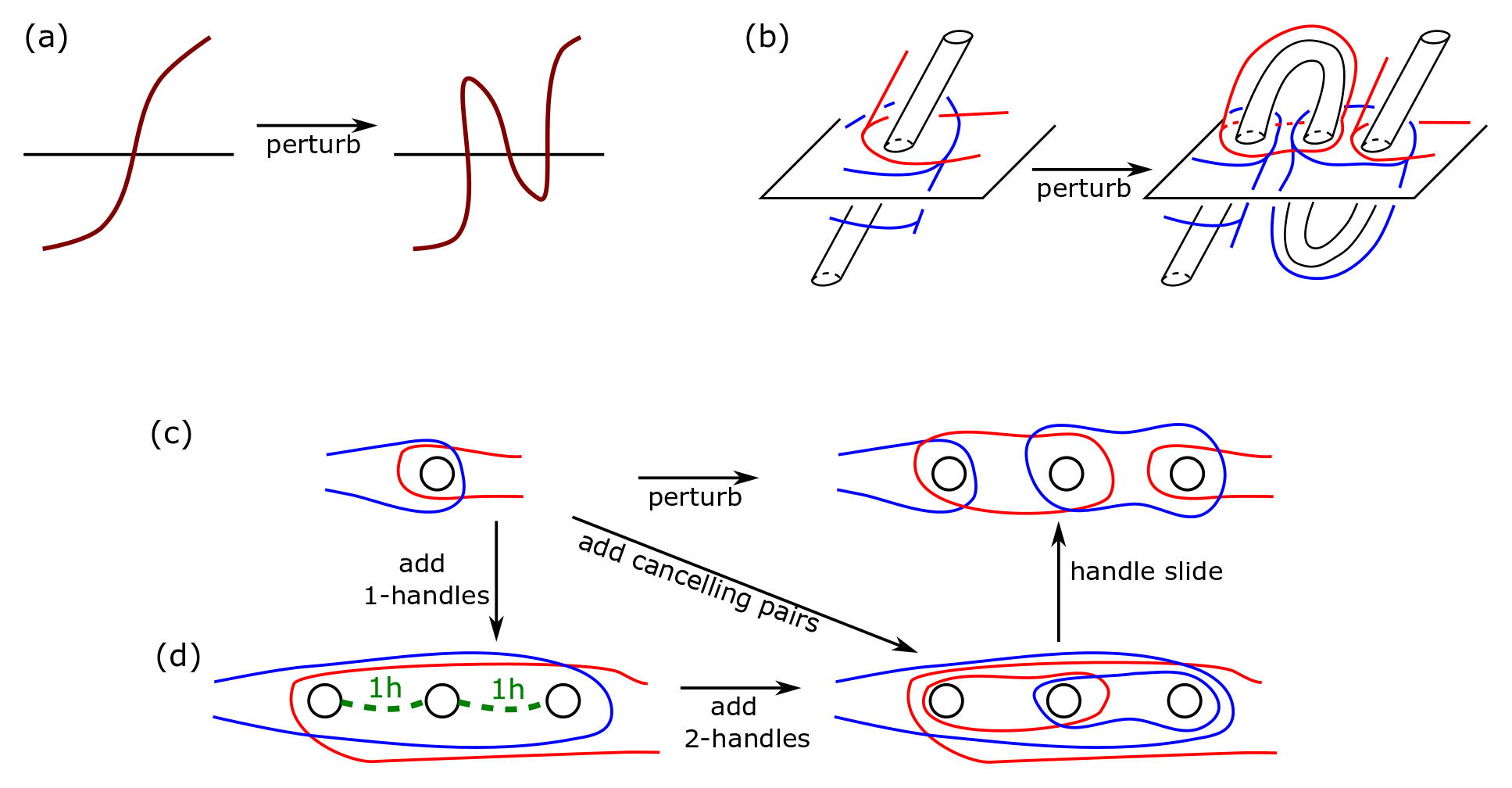}
\caption{(a) An image of a perturbation of $U$ around a point $U\cap F$. (b) How a perturbation changes the complement of $U$ in both compression bodies. (c) A perturbation changes the collections of disks $\delta^i$ by adding a curve to each of them. (d) A perturbation changes the decomposition of the 4-dimensional piece by adding two 1-2-cancelling pairs of handles.}
\label{fig_perturbation}
\end{figure}
\end{proof}

\begin{remark}[Stabilization]\label{remark_constructions_standard_piece}
For an integer $g\geq k$, we can stabilize the Heegard splitting of $Y=C^0\cup_{F_{k,b}} C^1$ ($g-k$) times to obtain the decomposition $Y=Y^-_{C^0, C^1} \cup Y^+_{C^0, C^1}$. This changes the Heegaard diagram $(F_{k,b}; \delta^0, \delta^1)$ by connected sum with ($g-k$) copies of the genus one diagram $(F_{1,0}; a_0, a_1)$, where $a_0$ and $a_1$ are loops intersecting once. 
We write the new diagram as $(\Sigma;\Delta^0, \Delta^1)$ where $\Sigma=F_{g,b}$ and $\Delta^i=\delta^i\cup \delta^i_{stab}$ and $(\delta^0_{stab}, \delta^1_{stab})$ correspond to the new dual loops. 
Note that stabilizations on the diagram correspond to adding two 1-2-cancelling pairs of handles to the 4-manifold built in Equation \ref{eq_construction_2_final}. Therefore, using Lemma \ref{lem_contruction_2_for_Z} we conclude  %$Z$ can be built from $(F_{g,b},\delta^0\cup \delta^0_{stab}, \delta^1\cup \delta^1_{stab})$ as in Equation \ref{eq_construction_2_final}.
\begin{align}\label{eq_construction_2_stab}
Z\quad  \approx&\quad  \Sigma\times [0,1]^2 \cup \left( \text{2-handles } \Delta^0\times\{1\}\times\{1/4\} \text{ and }  \Delta^1\times\{1\}\times\{3/4\}\right) \cup \left(\text{3-handles}\right),
\end{align}
where the 3-handles are attached along the sphere components obtained from compressing $\Sigma\times\{1\}\times\{1/4+i\}$ along $\Delta^i\times\{1\}\times\{1/4+i/2\}$ for $i=0,1$, and the spheres %built from the common curves in $\delta_{all}=\Delta^0\cap \Delta^1$.
built by connecting the core disks of the 2-handles $\delta_{all}\times\{1\}\times\{1/4,3/4\}$ with the annuli $\delta_{all}\times\{1\}\times[1/4,3/4]$. Recall that we need to add extra 3-handles and 4-handles if $\wt F_{all}$ is empty (see Remark \ref{remark_if_F_all_empty}). 
Thus, if both $C^0$ and $C^1$ are handlebodies (i.e., have empty negative boundary) then one of the boundary components of the resulting 4-manifold is a copy of $S^4$, which we cap-off uniquelly with one 4-handle. This last case occurs for classical trisections of closed 4-manifolds.

Denote by $C_{all}$ the compression body built from $\Sigma\times [0,1]$ using the loops $\Delta_{all}=\Delta^0\cap \Delta^1$ in $\Sigma$, and by $F_{all}$ the negative boundary of $C_{all}$. 
By construction $F_{all}$ is homeomorphic to a copy of $\wt F_{all}$, together with the disjoint union of planar surfaces (complements of the bridge spheres for each component of $U$). 
Denote by $C_0$ (resp. $C_1$) the compression body obtained from $F_{all}$ and the curves in $\Delta_i=\Delta^i-\Delta_{all}$. 
%The key observation is that w
Starting from Equation \ref{eq_construction_2_stab}, we can undo the `doubling' procedure in Subsection \ref{subsection_Construction_2} and conclude that $Z$ is diffeomorphic to the union
\begin{align}\label{eq_construction_1_stab}
Z \quad \approx& \quad C_0 \times [0,1/2] \bigcup_{F_{all}\times[0,1/2]} C_{all}\times [0,1] \bigcup_{F_{all}\times[1/2,1]} C_1 \times [1/2,1].
%\approx &\quad Y^-_{C^0, C^1}\times[0,1/2] \bigcup_{C_{\Delta_{all}}\times \{1/2\}} Y^+_{C^0, C^1}\times [1/2,1].
\end{align}
Moreover, using Figure \ref{model_Z}, we can check that $Z$ is diffeomorphic to 
\begin{align}\label{eq_construction_0_stab}
Z\quad \approx &\quad Y^-_{C^0, C^1}\times[0,1/2] \bigcup_{C_{all}\times \{1/2\}} Y^+_{C^0, C^1}\times [1/2,1].
\end{align}
From here, we can perform the isotopy described in Subsection \ref{subsection_Construction_1} and conclude that $Z$ is diffeomorphic to the union 
\begin{align}\label{equation_construction_4_stab}
Z\text{ }\approx \text{ }\left(\left(C_0\cup_{F_{all}} C_1\right)\times [0,1] \right)\bigcup C_{all}\times [0,1],
\end{align}
where we glue a neighborhood of the surface $F_{all}\times \{0\}$ in $\left( C_0\cup_{F_{all}} C_1\right)\times \{0\}$ with the product $F_{all}\times [0,1]$ in $C_{all}\times [0,1]$.
\end{remark}

\begin{remark}[Diagramatics]\label{remark_diagrams_standard_piece}
The previous equations show us several equivalent ways of building the standard 4-manifold piece $Z$ using the compression bodies $C^0$, $C^1$ and $C_{all}$ which are determined by the tuple $(\Sigma;\Delta^0, \Delta^1, \Delta_{all})$ up to isotopies and handle slides. The combinatorics of the loops $\Delta^0$ and $\Delta^1$ in $\Sigma$ is summarized in Figure \ref{Fig_standard_star_diagram}. 
The curves on each set $\Delta^i$ are pairwise disjoint and the sets decompose as $\Delta^i = \delta^i_{stab} \cup\Delta_{all}\cup \Delta_i$. The curves in $\delta^0_{stab}$ and $\delta^1_{stab}$ are in bijective correspondence and only the paired loops intersect in exactly one point while the rest are pairwise disjoint.
The curves in $\Delta_{all}$ are in both $\Delta^0$ and $\Delta^1$ and are pairwise disjoint. 
There are two types of families of loops in $\Delta_0\cup \Delta_1$: the fist type is formed of non-isotopic pairwise disjoint loops (identified with loops in $\wt\delta_i\subset F_{k,b_0}$); the loops in the second type form Heegaard splittings of the complement of an unknot in bridge position (obtained from the trivial arcs of each component of $U$).
\begin{figure}[h]
\centering
\includegraphics[scale=.5]{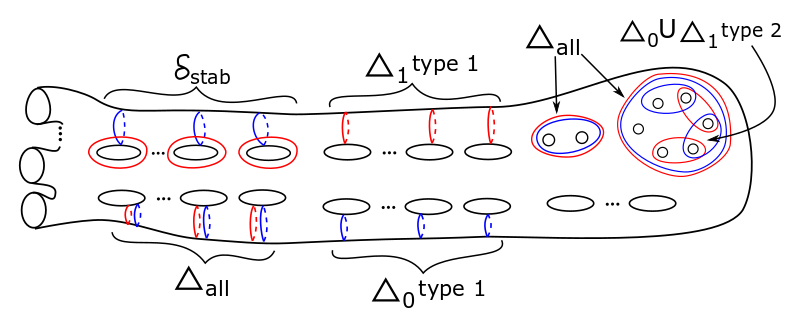}
\caption{The combinatorics of the tuple $(\Sigma; \Delta^0, \Delta^1, \Delta_{all})$. For this example, observe that $U$ has two components, one of which has bridge number one and the other bridge number three.}
\label{Fig_standard_star_diagram}
\end{figure}
\end{remark}

%{\color{red}\begin{remark}
%Make a remark summarizinig the types of curves in a standard picture necessary to build $Z$.... Don't forget the conditions on the delta-curves like no non-trivial spherical components.... Maybe say something about hoy the only thing that is not 'determined by the compressionbodies' (I think) are the loops from the spheres $S_*$ which must be determined in our pictures.....
%
%"These three different ways to think of $Z$, as a 4-manifold built from the compression bodies $C^0$, $C^1$ and $C_{all}$ as the input data."
%
%\end{remark}}

%%%%%%%%%%%%%%%%%%%%%%%%%%%%%%%%%%%%%%%%%%%%%%%%

\subsection{The boundary of $Z$}\label{subsection_boundary_Z}
Let $C$ be a compression body built from $F\times[0,1]$ by attaching some 3-dimensional 2-handles along $F\times\{1\}$. Recall that the boundary of $C$ is given by copies of its positive boundary ($F\times \{0\}$) and its negative boundary (the non-spherical components of $F\times\{1\}$ after compression), connected by annuli of the form $\partial_0 C:= \left(\partial F\right)\times [0,1]$. 

Let $B:= \left(\partial F\right)\times\{1/2\}\times\{1/2\}$ be the core of the solid torus $\left(\partial_0 C\right)\times[0,1]$ inside $\partial\left(C\times[0,1]\right)$. The 3-manifold $\partial \left(C\times [0,1]\right) -\eta(B)$ can be built by taking two copies of $C$ and gluing their negative (and positive) boundaries together using the identity map. 
Such splitting induces a circular handle decomposition of $\partial \left( C\times [0,1]\right)$ with binding $B$, built from $F\times[0,1]$ by attaching 2-handles on both sides $F\times\{0,1\}$ and identifying the resulting non-spherical components\footnote{Recall we always attach 3-handles along the new sphere components. In particular, if $C$ is a handlebody this contruction yields the classical handle decomposition of $\# S^1\times S^2$ which we still think of circular.} (see Figure \ref{Fig_boundary_part_1}). 
%Adopting the notations on the previous subsections (see Remark \ref{remark_constructions_standard_piece}), 
Using Equation \ref{eq_construction_0_stab}, we can glue the circular handle decompositions for $\partial\left(Y^-\times[0,1/2]\right)$ and $\partial\left(Y^+\times[1/2,1]\right)$ together along $C_{all}\times\{1/2\}$ to find a circular handle decomposition of the boundary of a standard piece $Z(C^0,C^1,C_{all})$. We draw them in Figures \ref{Fig_boundary_part_1} and \ref{Fig_boundary_part_2}. 
\begin{figure}[h]
\centering
\includegraphics[scale=.1]{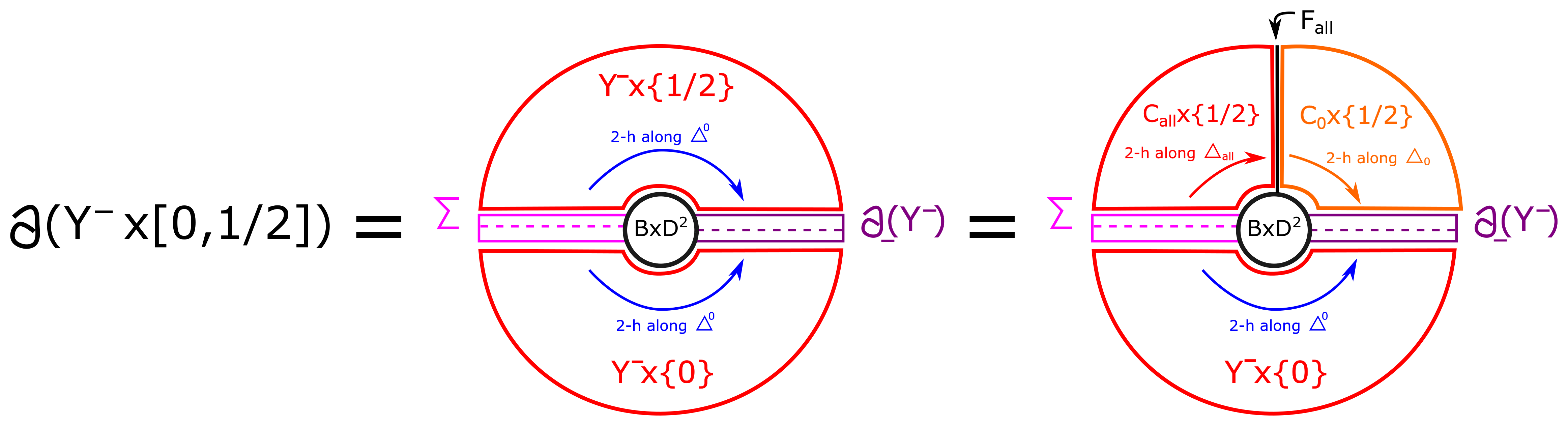}
\caption{Circular handle decomposition on $\partial \left(Y^-\times[0,1/2]\right)$. We can further decompose one copy of $Y^-$ as $C_{all}\cup_{F_{all}} C_0$.}
\label{Fig_boundary_part_1}
\end{figure}

\begin{figure}[h]
\centering
\includegraphics[scale=.075]{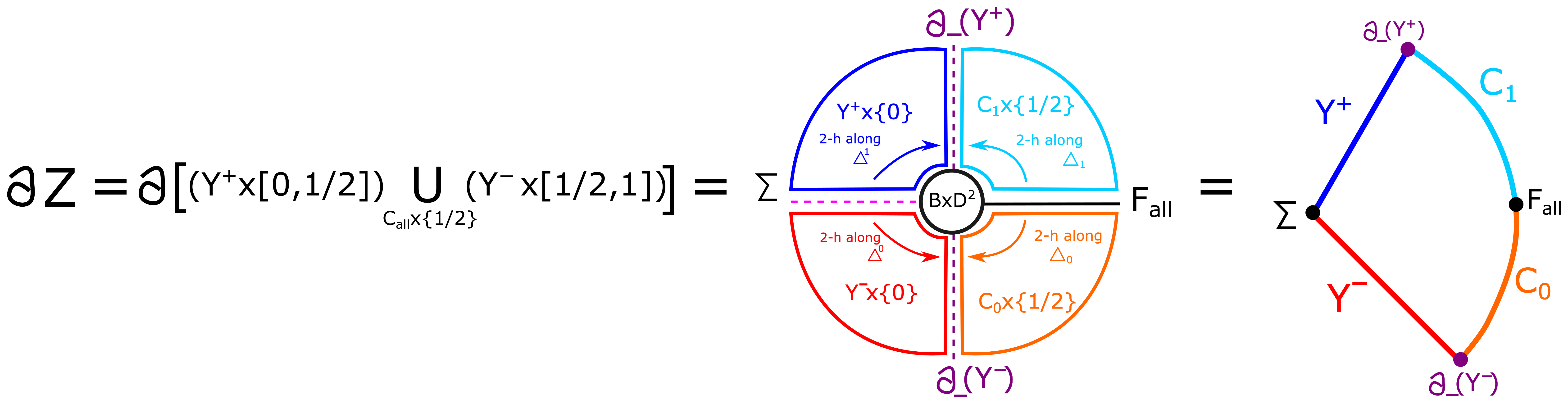}
\caption{How to build $\partial Z$ using the tuple $\left(\Sigma;\Delta^0, \Delta^1\right)$. Compare with right side of Figure \ref{model_Z}.}
\label{Fig_boundary_part_2}
\end{figure}

\begin{remark}[Boundary of $X$]\label{remark_boundary_of_X}
In a $\star$-trisected 4-manifold $X=X_1\cup X_2\cup X_3$, the pieces labeled $Y^\pm$ on the boundary of all standard pieces $X_i$ are glued together. This implies that $X_i\cap \partial X$, away from $\eta(B)$, is of the form $C_0\cup_{F_{all}} C_1$ (see Figure \ref{Fig_boundary_part_2}) which is the disjoint union of Heegaard splittings with Heegaard surfaces being the components of $F_{all}$. 
Thus the boundary of $X$ admits a circular handle decomposition with binding the boundary of the trisection surface $B=\partial \Sigma$. The decomposition on $\partial X- \eta(B)$ can be obtained by taking copies of the 3-manifolds $C_0\cup_{F_{all}}C_1$ obtained from each $X_i$ and gluing $\partial_-C_0$ from $X_i$ with $\partial_-C_1$ from $X_{i+1}$ as in Figure \ref{star_gluing_maps}.
%If $M$ is a connected component of $\partial X$
\end{remark}

%\begin{remark}[Boundary of a $\star$-trisection]
%Before stating the Pasting Lemma, we review the structure induced in the boundary of a $\star$-trisected 4-manifold (see Remark \ref{remark_boundary_of_X}). 
%Let $X=X_1\cup X_2\cup X_3$ be a $\star$-trisected 4-manifold with trisection surface $\Sigma=X_1\cap X_2\cap X_3$. By definition, there exist compression bodies $C_{(i)}^0$, $C_{(i)}^1$ and $C_{(i),all}$ with positive boundary $\Sigma$ such that $X_i\approx Z(C_{(i)}^0,C_{(i)}^1,C_{(i),all})$. By definition $C_{(i),all}$ is a subcompression body of both $C_{(i)}^0$ and $C_{(i)}^1$, so the 3-manifolds $C_{(i),0}:=C_{(i)}^0-\eta(C_{(i),all})$ and $C_{(i),1}:=C_{(i)}^1-\eta(C_{(i),all})$ are (posibly disconnected) compression bodies with positive boundary the surface $F_{(i),all}:=\partial_-C_{(i),all}$.
%Let $B\subset \partial X$ be the loops $\partial \Sigma$ thought of as a subset of the boundary of $X$.
%Suppose $X$ has non-empty boundary and let $M\subset \partial X$ be a connected component.
%\end{remark}

%%%%%%%%%%%%%%%%%%%%%%%%%%%%%%%%%%%%%%%%%%%%%%%%%%%%%%%%%%%%%%%%%%%%%%%%%%%%%%%%%%%%%%%%%%%
\section{Trisection diagrams}\label{subsection_trisection_diagrams}
In this section we will describe the diagramatics of $\star$-trisections. We will compare our notion of trisection diagrams with the existent literature. We end this section by discussing a diagramatic operation on $\star$-trisections with boundary called poking.

Let $X=X_1\cup X_2\cup X_3$ be a $\star$-trisected 4-manifold. The trisection surface $\Sigma=X_1\cap X_2 \cap X_3$ is a connected orientable surface of genus $g$ and $b\geq 0$ boundary components. The pairwise intersection $X_i\cap X_j$ is a compression body with positive boundary the surface $\Sigma$, denote the compresion bodies by $H_{\alpha}=X_1\cap X_3$, $H_{\beta}=X_2\cap X_3$ and $H_{\gamma}=X_1\cap X_2$. 
By definition, the 4-manifold piece $X_1$ is diffeomorphic to the 4-manifold $Z(H_{\gamma},H_{\alpha},H_{\gamma\cap \alpha})$ for a given comon subcompression body $H_{\gamma\cap\alpha}$ of $H_{\gamma}$ and $H_{\alpha}$. Let $\Delta_{\gamma\cap\alpha}$ be loops in $\Sigma$ determining the compression body $H_{\gamma,\alpha}$. We define $H_{\alpha\cap \beta}$, $\Delta_{\alpha\cap \beta}$, $H_{\beta\cap \gamma}$ and $\Delta_{\beta\cap\gamma}$ in a similar way. Let $\alpha$, $\beta$ and $\gamma$ be loops in $\Sigma$ determining the compression bodies $H_{\alpha}$, $H_{\beta}$ and $H_{\gamma}$, respectively. 

Due to Equation \ref{eq_construction_2_stab}, we can use the tuple $(\Sigma; \alpha, \beta, \gamma; \Delta_{\gamma\cap \alpha}, \Delta_{\alpha\cap \beta}, \Delta_{\beta\cap\gamma})$ to build the 4-manifold $X$ as follows: 
Let $p_\alpha^\pm$, $p_\beta^\pm$, $p_\gamma^\pm$ be six distinct points in the unit circle as in Figure \ref{Fig_building_X}. Start with $\Sigma\times D^2$ and glue copies of collars of the compresion bodies\footnote{This is equivalent to add 4-dimensional 2-handles to $\Sigma\times D^2$ along $\eps\times \{p_\eps\}$ with framing given by the surface, and 3-handles along the sphere components obtained by compressing $\Sigma\times\{pt_\eps\}$ for $\eps\in\{\alpha,\beta,\gamma\}$.} 
$H_\eps\times[p_\eps^-,p_\eps^+]$ % $H_\beta\times[p_\beta^-,p_\beta^+]$ and $H_\gamma\times[p_\gamma^-,p_\gamma^+]$ 
along the arcs $[p_\eps^-,p_\eps^+]\times \Sigma$ for $\eps \in\{\alpha,\beta,\gamma\}$. 
For each pair $(\eps,\mu)\in\{(\gamma,\alpha), (\alpha,\beta), (\beta, \gamma)\}$, the compression bodies $H_\eps \times\{p_\eps^\pm\}$ are contained in the boundary of the resulting 4-manifold, and the `comon' curves $\Delta_{\eps\cap\mu} \times\{p_\eps^-,p_\mu^+\}$ bound disks on them. 
Attach 4-dimensional 3-handles with core spheres given by connecting such disks with the annuli $\Delta_{\eps\cap\mu}\times[p_\eps^-,p_\mu^+]$.
%$D_\alpha^- \cup \left(\Delta_{\alpha\cap\beta}\times[p_\alpha^-,p_\beta^+]\right)\cup D_\beta^+$.
If one of the common compression bodies has empty negative boundary (see Remark \ref{remark_if_F_all_empty}), say $\partial_-H_{\alpha\cap\beta}=\emptyset$, after adding the 3-handles, the new boundary will have a component homeomorphic to $S^3$ in the interval $[p_\alpha^-, p_\beta^+]$. In such case, we attach a 4-dimensional 4-handle to cap-off the $S^3$ component.
The end result is our trisected 4-manifold $X$.
\begin{figure}[h]
\centering
\includegraphics[scale=.1]{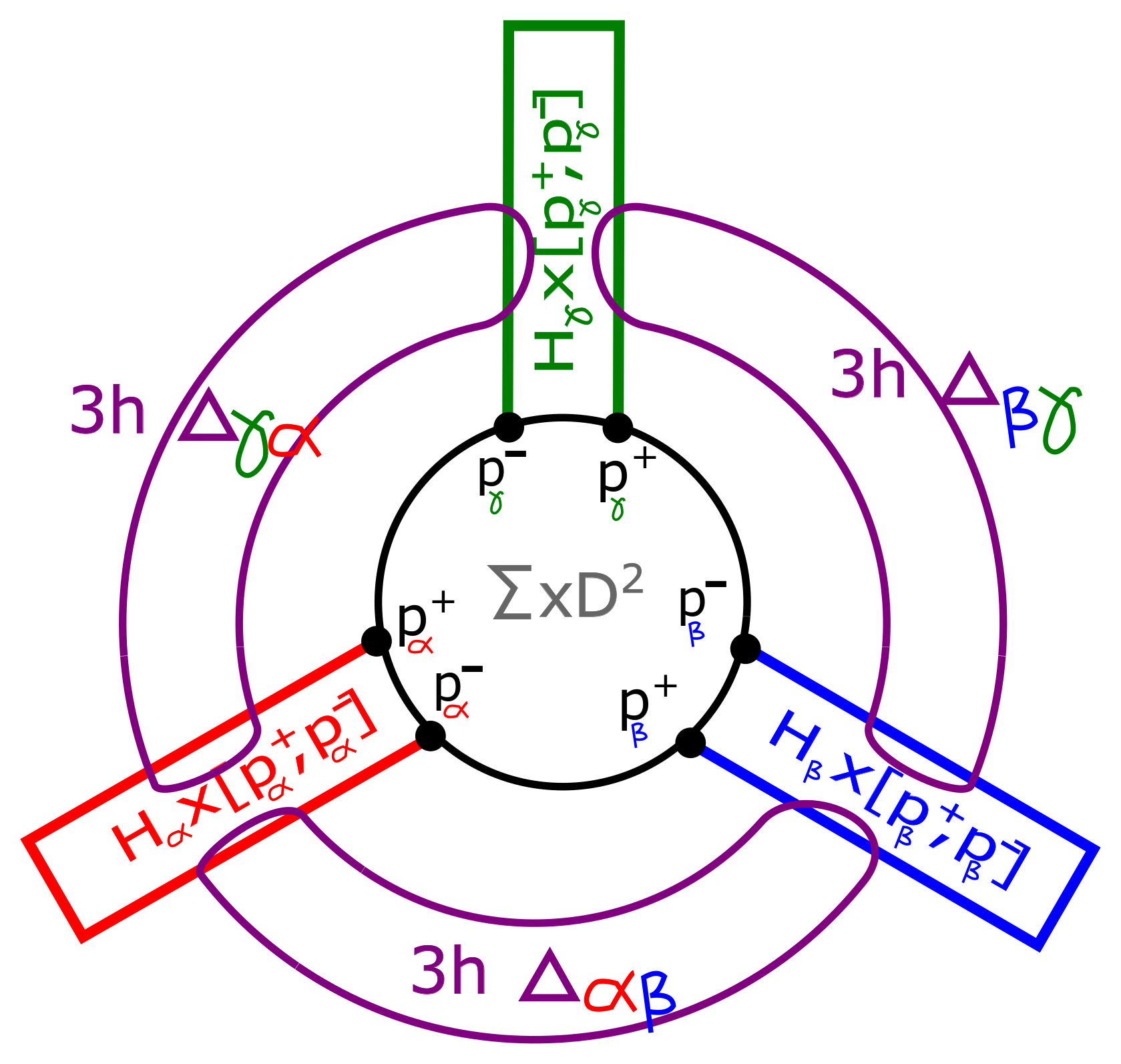}
\caption{A schematic of how to build $X$ from the tuple $(\Sigma; \alpha, \beta, \gamma; \Delta_{\gamma\cap \alpha}, \Delta_{\alpha\cap \beta}, \Delta_{\beta\cap\gamma})$. Compare this with Figure \ref{fig_C_allxI}. }
\label{Fig_building_X}
\end{figure}

\begin{remark}[Diagramatics] \label{remark_diagramatics}
We will refer to the data $(\Sigma; \alpha, \beta, \gamma; \Delta_{\gamma\cap \alpha}, \Delta_{\alpha\cap \beta}, \Delta_{\beta\cap\gamma})$ as a $\star$-trisection diagram.
%Since we are only interested about the compression bodies, 
The curves of a given diagram can be modified via handle slides along curves of the same label, or via homeomorphisms of the surface $\Sigma$ without changing the diffeomorphism class of the $\star$-trisected 4-manifold. %Thus there is a unique 4-manifold determined by a $\star$-trisection diagram 
The authors of this paper find the following interpretation of $\star$-trisection diagrams useful: A $\star$-trisection diagram is a tuple $(\Sigma;\alpha,\beta,\gamma)$ such that each pair $(\eps, \mu)\in\{(\gamma,\alpha), (\alpha,\beta), (\beta,\gamma)\}$, after adding some redundant curves to the $\eps$, $\mu$ sets, is handle slide equivalent to a standard pair $(\Delta^0, \Delta^1)$ from Remark \ref{remark_diagrams_standard_piece} with $\Delta_{\eps\cap \mu}=\Delta_{all}$.
Figure \ref{fig_levels_of_generality} shows examples of various $\star$-trisection diagrams. 
\end{remark}
\begin{figure}[h]
\centering
\includegraphics[scale=.43]{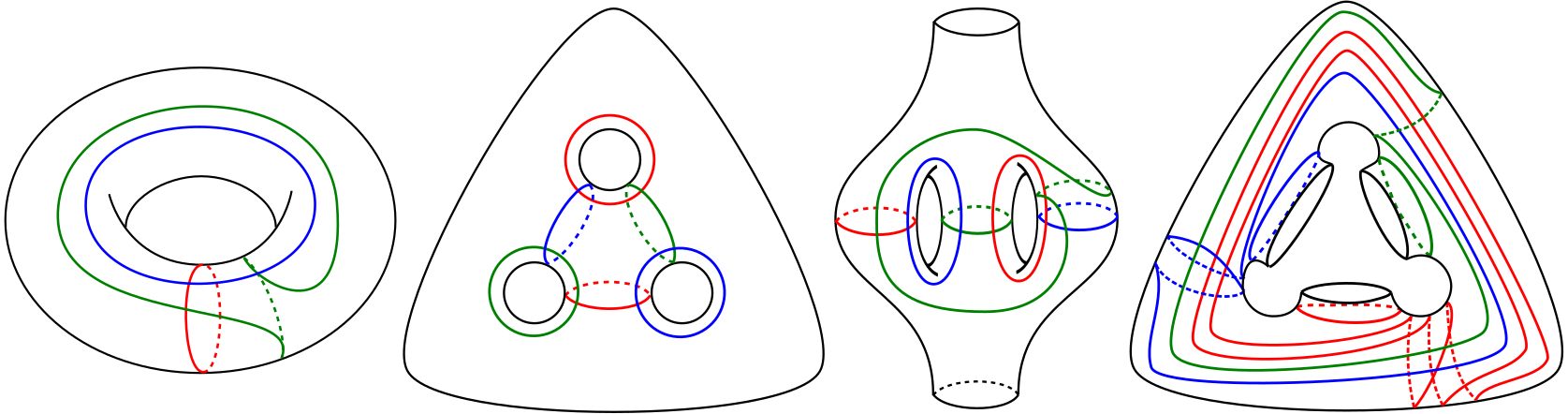}
\caption{(left to right) A trisection of $\CP2$, a relative trisection of the spin of $S^1\times D^2$, a relative trisection of a disk bundle over $S^2$ with $e=-1$, a $\star$-trisection of $\CP2\setminus\eta(S^1)$.}
\label{fig_levels_of_generality}
\end{figure}
\begin{remark}[Classic diagrams]
For \textbf{classical} trisections, each pair of curves $(\eps,\mu)$ in a relative trisection diagram is slide equivalent to a standard pair $(\Delta^0,\Delta^1)$ diffeomorphic to the one in Figure 5 of \cite{relative_trisections_2}. In particular the $\Delta^i$ sets decompose as $\Delta^i=\Delta_{all}\cup \delta^i_{stab}$, making the information of the $\Delta_{\eps\cap\mu}$ redundant. In this case the tuple $(\Sigma; \alpha, \beta,\gamma)$ is enough to build the 4-manifold $X$.
For general compression bodies $H_\eps$ and $H_\mu$ the choice of the common compression body $H_{\eps\cap \mu}$ may not be unique, forcing us to record $\Delta_{\eps\cap\mu}$ in the $\star$-trisection diagram. 
Fortunately, the choice of common curves will be easy to see for most of the $\star$-trisection diagrams that appear in this paper.
\end{remark}

%%%%%%%%%%%%%%%%%%%%%%%%%%%%%%%%%%%%%%%%%%%%%%

%%%%%%%%%%%%%%%%%%%%%%%%%%%%%%%%%%%%%%%%%%%%%%%%%%%%%%%%%
\subsection{New Relative Diagrams from Old}\label{subsection_poking}

Let $\mathcal{D}=(\Sigma; \alpha, \beta, \gamma; \Delta_{\gamma\cap \alpha}, \Delta_{\alpha\cap \beta}, \Delta_{\beta\cap\gamma})$ be a $\star$-relative trisection diagram (see Remark \ref{remark_diagramatics}) for a 4-manifold $X$ with non-empty boundary. 
Let $p_\alpha, p_\beta, p_\gamma\subset \Sigma-(\alpha\cup \beta\cup \gamma)$ be three finite collections of points. Define $\mathcal{D}'=(\Sigma';\alpha', \beta', \gamma'; \Delta_{\gamma\cap \alpha}, \Delta_{\alpha\cap \beta}, \Delta_{\beta\cap\gamma})$ be a new tuple given by $\Sigma'=\Sigma-\overset{\circ}{\eta}(p_\alpha\cup p_\beta\cup p_\gamma)$ and $\eps' = \eps \cup \partial \eta (\eps)$ for $\eps =\alpha, \beta, \gamma$. %Suppose that $\mathcal{D}'$ is also a $\star$-trisection diagram. 
Using Equation \ref{eq_construction_2_stab} one can show that if $\mathcal{D}'$ is a $\star$-trisection diagram, then it is a diagram for $X$. We will refer to the operation of replacing $\mathcal{D}$ by $\mathcal{D}'$ by \textbf{poking}. 

Poking a $\star$-trisection diagram adds new boundary components to the trisection surface and to the negative boundaries of the compression bodies $C_{\alpha}$, $C_\beta$ and $C_\gamma$. This operation comes handy when gluing $\star$-trisections. One condition for two $\star$-trisections to be able to be glued along their boundaries is that the surfaces $\partial_-C_\eps$ ($\eps=\alpha,\beta,\gamma$) have all non-empty boundary. 
Figure \ref{fig_poking_example_1} shows an example of how to poke a $\star$-trisection diagram in order to satisfy this condition.

In practice, one can poke a diagram by drawing one pair of curves, say $(\alpha, \beta)$, into standard position $(\Delta^0,\Delta^1)$, and then pick the poking points $p_\gamma\subset \Sigma-(\alpha \cup \beta \cup \gamma)$ whenever needed. The location of the points in $p_\gamma$ with respect to the curves in $\gamma$ is irrelevant since we think of the curves in a diagram defined up-to handle slides and we can always slide $\partial \eta (p_\gamma)$ over $\gamma$ as required. 

\begin{figure}[h]
\centering
\includegraphics[scale=.3]{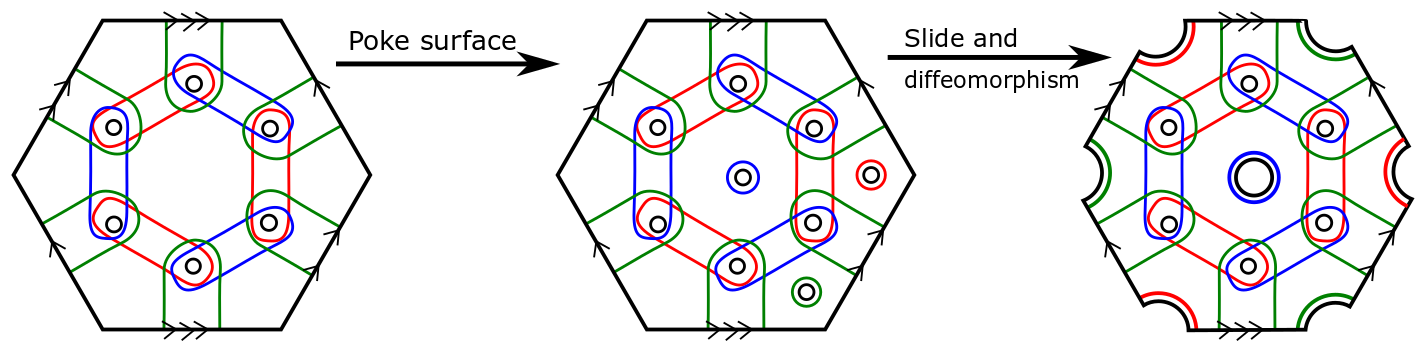}
\caption{Example of how to poke a $\star$-trisection diagram.}
\label{fig_poking_example_1}
\end{figure}

%\begin{lem}[Poking Lemma] 
%Let a $\star$-trisection diagram $\mathcal{D}$ for a 4-manifold with non-empty boundary boundary. There exists a $\star$-trisection diagram for the same 4-manifold, obtained by poking $\mathcal{D}$ enough times, with the condition that all the pages are surfaces with boundary. (write this better)
%\end{lem}

\begin{remark}[Relative Double Twists]
If we can pick points $|p_\alpha|=|p_\beta|=|p_\gamma|=1$ so that all lie in the same component of $\Sigma-(\alpha\cup \beta\cup \gamma)$, poking corresponds to connect sum the given $\star$-trisection diagram with the diagram $\tau_1$ in Figure \ref{diagrams_S2D2}. 
In this case, one could also connect sum the diagram $\tau_2$ instead and still obtain a $\star$-trisection of the same 4-manifold. 
The new diagram will change the circular handle decomposition in the boundary by taking a solid torus neighborhood of a regular circle fiber and replacing it with the Seifert fibered space $S(0,1;+1,-1)$. For relative trisection diagrams, the new monodromy is given by composing one positive and one negative Dehn twist along the two new boundary components, respectively (see Section 3.1 of \cite{OBD_from_fibrations}). 
The operation of connect summing (when possible) a relative trisection diagram with $\tau_2$ is exactly the relative double twist move defined by Castro, Islambouli, Miller and Tomova in \cite{relative_L_inv}. Using this move, the authors were able to show that any two relative trisections for the same 4-manifold with connected boundary and non-closed trisection surface ($b>0$) are related by sequence of ambient isotopies, stabilizations, relative stabilizations, relative double twists, and the inverses of these moves. 
%We can use Lemma \ref{poking_lem} to poke any relative trisection with closed trisection surface with $\tau_2$ and turn it into a relative trisectioin with trisection surface having non-empty boundary $b>0$. Hence, we obtain the same result for any two relative trisection of $X$ with $\partial X$ connected.
\end{remark}

\begin{figure}[h]
\centering
\includegraphics[scale=.3]{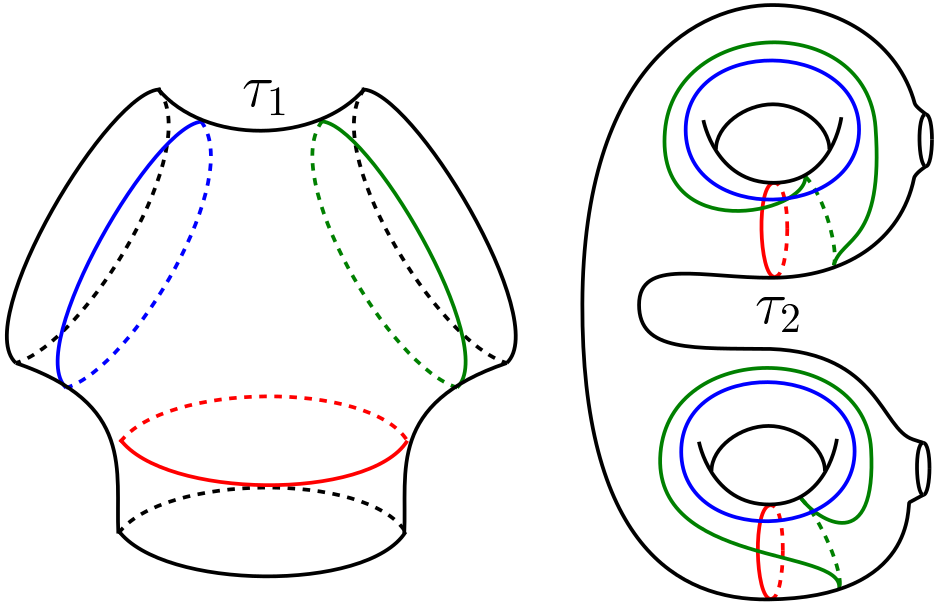}
\caption{Two relative trisection diagram for $S^2\times D^2$.}
\label{diagrams_S2D2}
\end{figure}

%%%%%%%%%%%%%%%%%%%%%%%%%%%%%%%%%%%%%%%%%%%%%%%%%%%%%%%%%%%%%%%%%%%%%%%%%%%%%%%%%%%%%%%%%%%%%%%%%%%%%%%%%%%%%%%%%%%%%%%%%%%%%%%%%%%%%%%%%%%%%
\section{Pasting $\star$-Trisections}\label{section_pasting}

Motivated by the work of Castro and Ozbagci \cite{pasting_rel_trisections,trisections_via_lefschetz}, we want to describe a method to paste two compatible $\star$-trisections along their boundaries.
The main technical observation that makes Theorem \ref{pasting_lemma_star} work is that when we paste two standard pieces $Z(C^0,C^1,C_{all})$ along some connected components of $C_0 \cup_{F_{all}} C_1$, the resulting 4-manifold is also a standard piece whenever all the glued surfaces in $F_{all}$ have boundary. 
The interested reader might note that our proof follows the same outline that the one of Theorem 2 of \cite{pasting_rel_trisections} applied to the setting of $\star$-trisections. %The conclusions of Remark \ref{remark_constructions_standard_piece} will be useful for us.

%As an application, we can use this pasting lemma to show that the so-called Farey trisections describe 1-surgery on simply connected closed 4-manifolds of trisection genus one. This insight allows us to subsequently standardize the diagrams with a sequence of handle slides.
%\begin{thm} [Pasting Lemma for $\star$-relative trisections] \label{pasting_lemma_star}
%Let $W=W_1\cup W_2 \cup W_3$ and $W'=W_1'\cup W_2' \cup W_3'$ be two $\star$-trisected 4-manifolds with boundary. Let $Y\subset W$ and $Y'\subset W'$ be connected components of the boundary of each 4-manifold and suppose $f:Y\ra Y'$ is a homeomorphism of the boundaries respecting the circular handle decompositions induced by the $\star$-trisections. Suppose that each component of the negative boundary of the compression bodies $W_i\cap W_j \cap Y$ and $W_i'\cap W_j' \cap Y'$ has boundary for all $i\neq j$. Then $X=W \cup_f W'$ admits a $\star$-trisection with pieces given by $W_i\cup_f W_i'$ ($i=1,2,3$), and $\star$-trisection surface the result of gluing the trisection surface of $W$ and $W'$ along the boundary components corresponding to $Y$ and $Y'$. 
%\end{thm}

We first review the circular handle decomposition on a $\star$-trisected 4-manifold. % Remark \ref{remark_boundary_of_X}. 
Let $X=X_1\cup X_2\cup X_3$ be a 4-manifold with a fixed $\star$-trisection. %and let $M\subset  \partial X$ be a connected component of its boundary. 
Let $B:=\partial \Sigma$ be the boundary of the trisection surface embedded as a link in $\partial X$; $B$ might be empty. %Remark \ref{remark_boundary_of_X} states that the complement $\partial X -\eta(B)$ decomposes as the union of three 3-manifolds of the form $C_0\cup_{F_{all}} C_1$, one per 4-dimensional piece.
%each of them obtained from $X_i \cap \partial X$. 
%For each $i=1,2,3$, the 4-manifold piece is of the form $X_i\approx Z(C_{(i)}^0,C_{(i)}^1,C_{(i),all})$ for some compression bodies with positive boundaty $\Sigma$. ......
Observe that $\partial X-\eta(B)$ decomposes as the union of three 3-manifolds, one per 4-dimensional piece. Remark \ref{remark_boundary_of_X} states that for each $i=1,2,3$, the 3-manifold $X_i \cap (\partial X-\eta(B))$ admits a Heegaard splitting of the form $C_{(i),0}\cup_{F_{(i),all}} C_{(i),1}$. The compression bodies in this decomposition are obtained from the data in the identification $X_i\approx Z(C_{(i)}^0,C_{(i)}^1,C_{(i),all})$ and have as many connected components as $F_{(i),all}$ (see Remark \ref{remark_constructions_standard_piece} for more details). 
%If $M$ is a connected component of the boundary of $X$, the 3-manifold $X_i \cap (Y-\eta(B))$ does not have to be connected. In general, each component of the surface $F_{(i),all}$ inside $Y$ will determine a connected component of $X_i \cap (Y-\eta(B))$.

\begin{thm} [Pasting Lemma for $\star$-relative trisections] \label{pasting_lemma_star}
Let $W=W_1\cup W_2 \cup W_3$ and $W'=W_1'\cup W_2' \cup W_3'$ be two $\star$-trisected 4-manifolds with non-empty boundary. Let $M\subset W$ and $M'\subset W'$ be closed 3-manifolds of the boundary of each 4-manifold. Let $f:M\ra M'$ be a homeomorphism satisfying $f(W_i\cap M)=W'_i\cap M'$ %, $f(W_i\cap W_j \cap Y) =W'_i\cap W'_j\cap Y'$ 
and preserving the Heegaard decompositions on $W_i\cap M$ and $W'_i\cap M'$ induced by the trisections for all $i$.
%More precisely, if $W_i=Z(C_{(i)}^0,C_{(i)}^1,C_{(i),all})$ (similarly $W'_i$) then $f$ sends the Heegaard splitting on $W_i\cap Y\subset C_{(i),0}\cup_{F_{(i),all}} C_{(i),1}$ to the Heegaard splitting on $W'_i\cap Y'\subset C'_{(i),0}\cup_{F'_{(i),all}} C'_{(i),1}$, 
%satisfying $f(W_i\cap Y)=W'_i\cap Y'$ and $f(W_i\cap W_j \cap Y) =W'_i\cap W'_j\cap Y'$ for all $i,j$. 
Suppose that each component of the surfaces %negative boundary of the compression bodies 
$W_i\cap W_j \cap M$ and $W_i'\cap W_j' \cap M'$ has boundary for all $i\neq j$. Then $X=W \cup_f W'$ admits a $\star$-trisection with pieces given by $W_i\cup_f W_i'$ ($i=1,2,3$), and $\star$-trisection surface the result of gluing the trisection surfaces of $W$ and $W'$ along their boundary components intersecting with $M$ and $M'$, respectively.
\end{thm}
\begin{proof} 
Let $X_i = W_i \cup_f W_i'$. We will show that $X=X_1\cup X_2\cup X_3$ is a $\star$-trisection. 

We will first focus on the pairwise intersection $X_i\cap X_j$. Fix $i\neq j$ and let $C=W_i\cap W_j$ and $C'=W_i'\cap W_j'$. By definition $C$ and $C'$ are compression bodies with positive boundary the connected $\star$-trisection surfaces $\Sigma$ and $\Sigma'$, respectively. 
Let $P_M=\partial_-C\cap M$ and $P_{M'}=f(P_M)=\partial_-C'\cap M'$ be the negative boundary components of $C$ and $C'$ lying inside the gluing regions $M$ and $M'$. By assumption each component of $P_M$ and $P_{M'}$ is a surface with boundary. 
This implies that $H=C\cup_{f} C'$ is also a compression body with positive boundary $\wt \Sigma=\Sigma\cup_f\Sigma'$. 
In particular, the triple intersection $X_1\cap X_2\cap X_3$ is a copy of $\wt \Sigma$. 
%A set of curves in $\wt \Sigma$ bounding disks in $\wt C$ is given as follows: l
Let $\eps \subset \Sigma$ and $\eps'\subset \Sigma'$ be loops determining $C$ and $C'$. Let $a\subset P_M$ be pairwise disjoint properly embedded arcs which cut $P_M$ into disks and let $a'=f(a)$ be the corresponding arcs in $P_{M'}$. 
We can think of the arcs in $a$ (resp. $a'$) as subsets of $\Sigma$ (resp. $\Sigma'$), disjoint from $\eps$ (resp. $\eps'$).
The curves in $\wt \Sigma=\Sigma\cup_f \Sigma'$ determining $H$ are given by $\eps\cup \eps'$ and the glued arcs $a\cup_f a'$.
%Recall that a $\star$-trisection is determined by six compression bodies: $H_{i,j}=X_i\cap X_j$ and the common subcompression bodies $C_{(i),all}$. In order to talk about the common ones, we need to do some more work as follows.

We now show that each piece $X_i$ is diffeomorphic to a standard piece. %In order to do this, we need to find curves $. % for some compression bodies. 
Fix $i\in \{1,2,3\}$, by definition $W_i\approx Z(C^0,C^1, C_{all})$ and $W'_i\approx Z(C'^0,C'^1,C'_{all})$ for some compression bodies. We use the notation of Remark \ref{remark_diagrams_standard_piece}; i.e., $C^0$ and $C^1$ are determined by essential loops $\Delta^0,\Delta^1\subset\Sigma$, and the common curves $\Delta_{all}=\Delta^0\cap \Delta^1$ determine the $C_{all}$. %(similary $C'_{all}$).
Let $F_{all}$ be the negative boundary of $C_{all}$. For $j=0,1$ let $C_j=C^j-int(C_{all})$ be the compression body with positive boundary the surface $F_{all}$, and let $P_j$ be the negative boundary of $C_j$. 
Recall that, due to Remark \ref{remark_boundary_of_X}, $W_i\cap \partial W$ contains a copy of $C_0\cup_{F_{all}} C_1$. 

If $F_{all}$ is disconnected, we can pick $\Delta_{all}\subset \Sigma$ having separating simple closed curves. Each such curve decomposes the diagram $(\Sigma;\Delta^1,\Delta^0, \Delta_{all})$ as a connected sum of standard diagrams in lower genus surfaces. This decomposes $W_i$ into a boundary connected sum of the corresponding standard 4-manifolds. The same argument holds for $W'_i$. Hence, it is enough to check that $W_i\cup_f W'_i$ is a standard piece when $F_{all}$ and $F'_{all}$ are connected surfaces. 

Suppose $F_{all}$ is a connected surface. In particular, the curves in $\Delta_{all}\subset \Sigma$ do not separate the surface and $C_{all}$ is built from $F_{all}\times[0,1]$ by attaching $|\Delta_{all}|$ 1-handles\footnote{No 3-handles are needed in this case since we are assuming $F_{all}$ is non empty.} along $F_{all}\times\{1\}$. 
Then, using Equation \ref{equation_construction_4_stab} we conclude that $W_i$ is diffeomorphic to the result of adding $|\Delta_{all}|$ 1-handles to $(C_0\cup_{F_{all}} C_1)\times [0,1]$. Since $F_{all}$ is connected, Remark \ref{remark_eq_1} explains that the location of the feet of the 1-handles is not necessary. 
Furthermore, since $(\Sigma; \Delta^0, \Delta^1, \Delta_{all})$ is a standard pair and $F_{all}$ is connected, it follows that $\Delta_0\cup\Delta_1$ only has curves of one type (see Remark \ref{remark_diagrams_standard_piece}). In any case, we get that $(C_0\cup_{F_{all}}C_1)\times [0,1])$ is a 1-handlebody of genus $N$. Here $N=\left( 1-\chi(F_{all})-|\Delta_0|-|\Delta_1|-2|\delta^0_{stab}|\right)$ if $\Delta_0\cup\Delta_1$ is of type 1 or $N=1$ if type 2.  
One can see this by building $W_i$ like in Equation \ref{eq_construction_2_stab} and noticing that each circle in $\Delta_0\cup \Delta_1$ cancels a 1-handles from the 1-handlebody $F_{all}\times[0,1]^2$ (recall $F_{all}$ has boundary).
On the other hand, to build $X_i$ we must glue $(C_0\cup_{F_{all}}C_1)\times[0,1]$ and $(C'_0\cup_{F'_{all}}C'_1)\times[0,1]$ %along $(C_0\cup_{F_{all}}C_1)\times\{1\}$ and $(C'_0\cup_{F'_{all}}C'_1)\times\{1\}$ 
using $f\times\{1\}$, and then add copies of $C_{all}\times[0,1]$ and $C'_{all}\times[0,1]$ as in Equation \ref{equation_construction_4_stab} (see right side of Figure \ref{isotopy_model_Z}).
Therefore, $X_i=W_i\cup_f W'_i$ is a 1-handlebody of genus $K=|\Delta_{all}|+|\Delta'_{all}|+N$. 

To end, observe that the boundary of $X_i=W_i\cup_f W'_i$ has a decomposition of the form 
$H^0\cup_{\wt \Sigma} H^1$ where $H^j=(C^j\cup_f C'^j)$. In particular, both $H^0$ and $H^1$ are handlebodies and this is a Heegaard splitting for $\partial X_i = \#_K S^1\times S^2$. Waldhaussen's Theorem \cite{waldhausen} implies that this splitting is a $g(\wt\Sigma)-K$ times stabilization of the standard genus $K$ Heegaard splitting. We can take this standard splitting to be $H_{all}$ and conclude that $X_i$ is diffeomorphic to the standard piece $Z(H^1,H^0, H_{all})$. This finishes the proof.
\end{proof}

\begin{remark} [The Diagrammatics] \label{remark_diagramatics_pasting_lemma}
Take two $\star$-trisected 4-manifolds $W$, $W'$ with a diffeomorphism $f:M\ra M'$ between closed submanifolds of their boundaries respecting the handle respective handle decompositions. Suppose that all components of $W_i\cap W_j \cap M$ have non-empty boundary. 
Let $(\Sigma; \alpha, \beta, \gamma; \Delta_{\gamma\cap \alpha}, \Delta_{\alpha\cap \beta}, \Delta_{\beta\cap\gamma})$ be a $\star$-trisection diagram for $W$ (resp. $W'$). Using Lemma \ref{pasting_lemma_star}, we get a $\star$-trisection diagram $(\wt \Sigma; \wt\alpha,\wt \beta, \wt\gamma; \Delta_{\wt\gamma\cap \wt\alpha}, \Delta_{\wt\alpha\cap \wt\beta}, \Delta_{\wt\beta\cap\wt\gamma})$ for $X=W\cup_f W'$ with trisection surface $\wt \Sigma= \Sigma\cup_f \Sigma'$. 
For $\wt\eps = \wt\alpha, \wt\beta, \wt\gamma$, the curves in the compression body $H_{\wt\eps}$ are given by:
\begin{enumerate} 
\item Loops in $\eps\subset \Sigma$ and $\eps'\subset \Sigma'$.
\item Disks obtained from pairwise disjoint arcs filling the compressed page $P_\eps$, glued along their boundaries to their images on $P_{\eps'}$ under $f$. 
\end{enumerate}
In the proof of Lemma \ref{pasting_lemma_star} we showed that the curves in $(\wt \eps,\wt \mu)$ obtained from the second item above are slide equivalent to a collection of isotopic curves and curves intersecting in one point (coming from stabilizations). 
Thus the curves in $\Delta_{\wt \eps\cap \wt \mu}$ are given by the curves in $\Delta_{\eps\cap \mu}\cup \Delta_{ \eps'\cap\mu'}$, together with this new parallel curves. 
\end{remark}

Since relative trisections with non-empty binding ($b>0$) satisfy the conditions of Lemma \ref{pasting_lemma_star}, our result extends the gluing theorems of \cite{pasting_rel_trisections} and \cite{trisections_via_lefschetz}. In this case ($b>0$), the monodromy algorithm of \cite{relative_trisections} is helpful to find the image of the filling arcs under $f$. 
In general, in order to glue two $\star$-trisections with Theorem \ref{pasting_lemma_star}, one might need to poke the given diagrams (see Subsection \ref{subsection_poking}) enough times to ensure that every connected component $\partial_- (W_i\cap W_j)$ is a surface with boundary. In particular, for relative trisections with $b=0$, it is sufficient to poke the trisection surface three times. 

\begin{cor} [Pasting Lemma for Relative Trisections with Empty Binding]\label{pasting_lemma}
Let $W=W_1\cup W_2\cup W_3$ and $W'=W'_1\cup W'_2\cup W'_3$ be two trisected 4-manifolds with non-empty connected boundary and \textbf{closed} trisection surfaces $\Sigma$ and $\Sigma'$, respectively. Let $P$ and $P'$ be the pages of the fibration over $S^1$ on $\partial W$ and $\partial W'$ induced by the trisections, respectively. Let $f:\partial W \ra \partial W'$ be a homeomorphism between the boundaries respecting the pages; i.e., $f(P)=P'$. Then the glued closed 4-manifold $X=W\cup_f W'$ admits a $(G;K)$-trisection where \[G=g(\Sigma) + g(\Sigma') +2 \text{ and } K_i=k_i+k_i'+2g(P).\]  Here $k_i$ denote the number of common curves in the trisection diagrams. 
\end{cor} 
 
\begin{remark}[Diagramatics for relative trisections with empty binding]
Suppose now that we are given two relative trisection diagrams with \textbf{closed trisection surface} for two 4-manifolds with connected boundary. Suppose that we are also given a homeomorphism $f:\partial W\ra \partial W'$ between connected components of the boundaries preserving the pages of the open book decomposition. 
To build the trisection diagram for the union $W\cup_f W'$ we need to find three poking points $p_\alpha, p_\beta, p_\gamma \subset \Sigma-(\alpha\cup \beta\cup \gamma)$. These points can be thought as lying in a page $P$ for the open book decomposition in $\partial W$. Thus we can take their image under $f$ to be points in the page $P'\subset \partial W'$. Denote by $p'_\alpha, p'_\beta, p'_\gamma \subset \Sigma'-(\alpha'\cup \beta' \cup \gamma')$ the corresponding points in $\Sigma'$. It follows that these are also poking points and the two new trisection diagrams for $W$ and $W'$ (obtained by poking) satisfy the conditions of Lemma \ref{pasting_lemma_star}. 
Hence, $W\cup_f W'$ admits a trisection diagram $(\wt \Sigma; \wt\alpha,\wt \beta, \wt\gamma)$ with trisection surface
\[\wt \Sigma = \left( \Sigma -\cup_\eps \overset{\circ}{\eta}(\eps)\right) \bigcup_{\partial \eta(\eps) = \partial \eta(\eps')} \left( \Sigma' -\cup_{\eps'} \overset{\circ}{\eta}(\eps')\right).\] 
For $\wt \eps = \wt\alpha, \wt \beta, \wt \gamma$, the curves in $\wt \eps$ are given by:
\begin{enumerate} 
\item Loops in $\eps\subset \Sigma$ and $\eps'\subset \Sigma'$.
\item Disks obtained from pairwise disjoint arcs filling the compressed page $P_\eps$, glued along their boundaries to their images on $P_{\eps'}$ under $f$. 
\item The loop $\partial \eta (p_\eps)$.
\end{enumerate}
\end{remark} 
\begin{figure}[h]
\centering
\includegraphics[scale=.04]{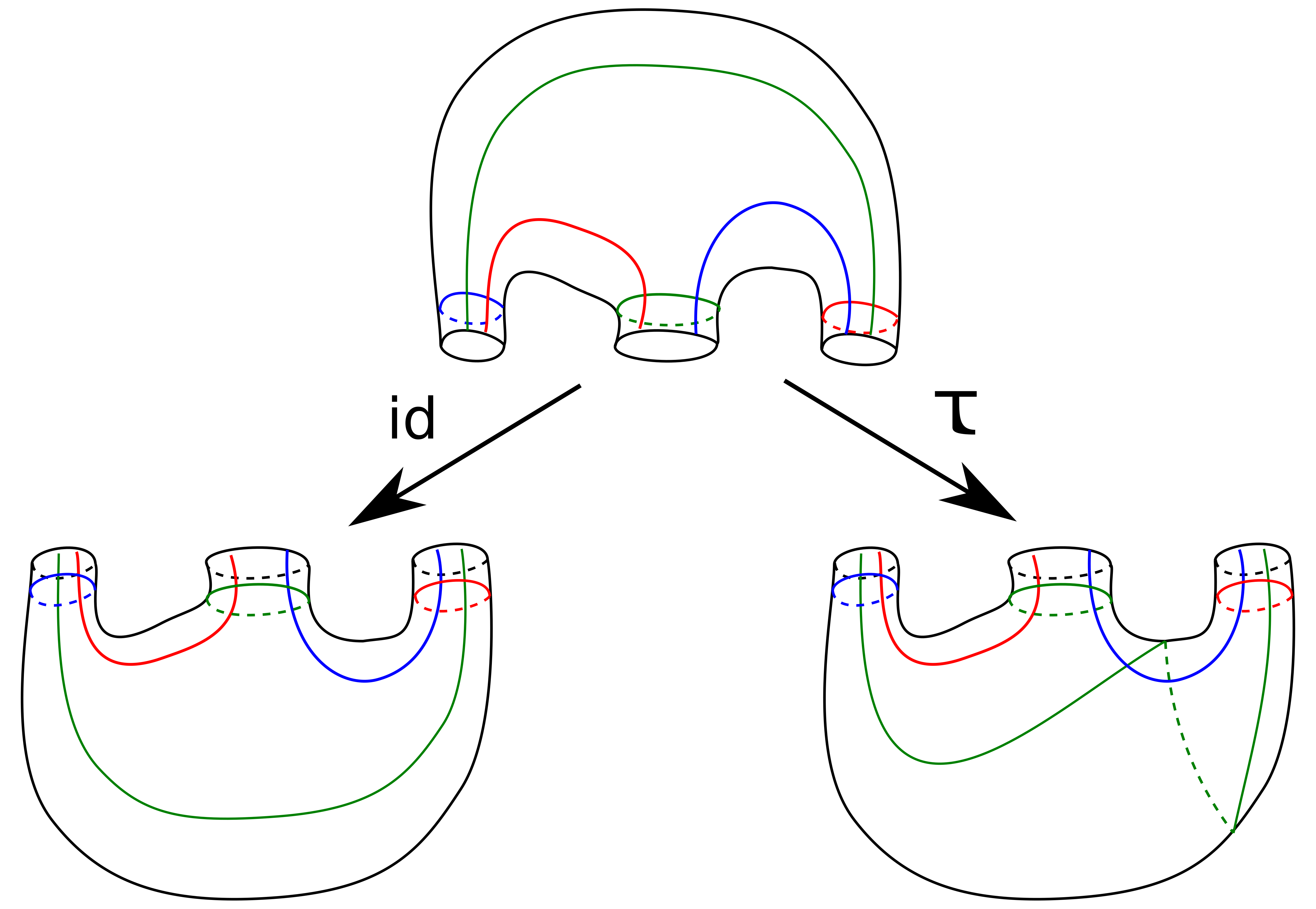}
\caption{Two distinct ways of gluing a pair of thickened spheres. The map $\tau$ twists the $S^2$ fiber once while traversing the $S^1$ direction.}
\end{figure}
\begin{figure}[h]
\centering
\includegraphics[scale=.13]{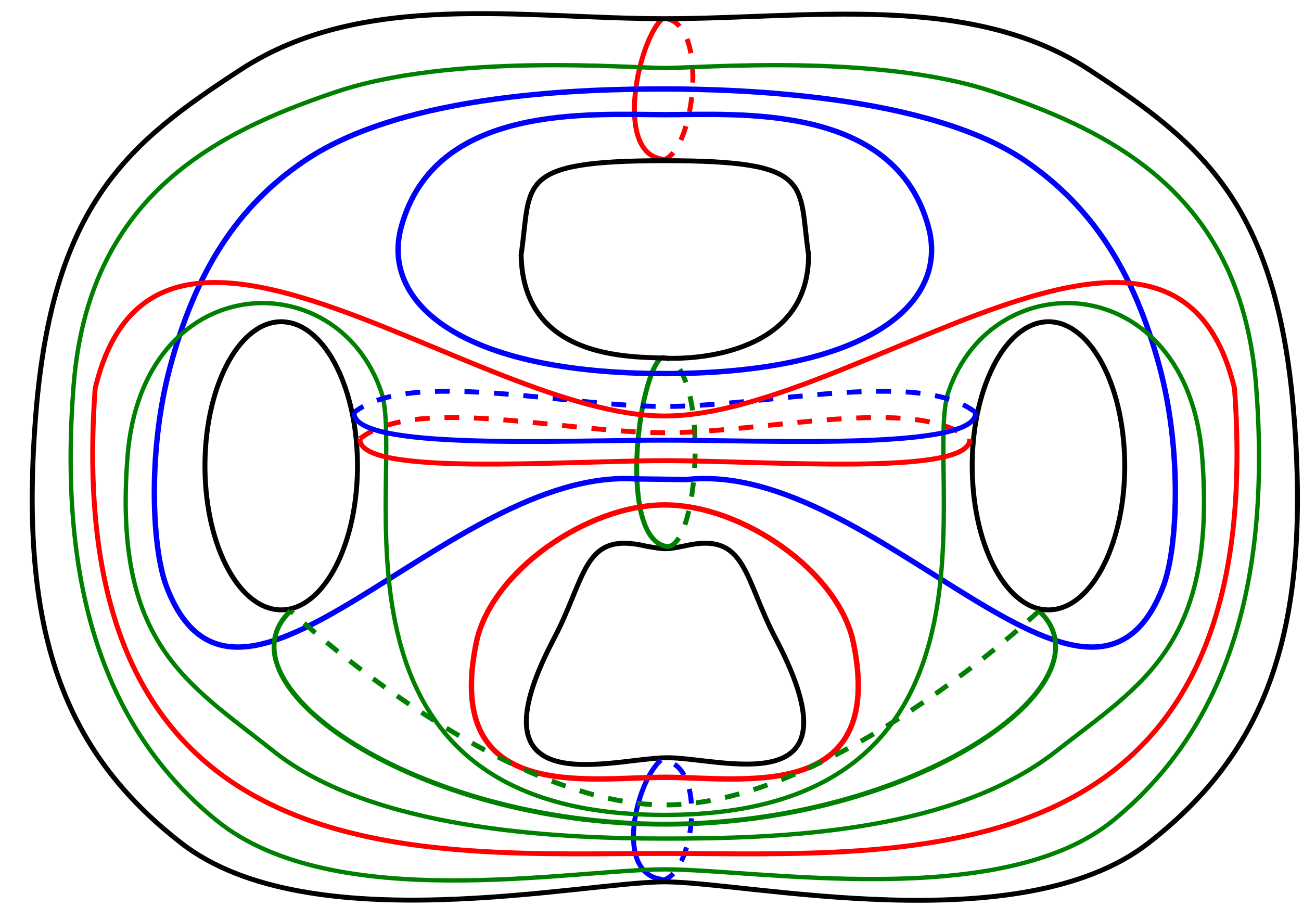}
\caption{Trisection diagram for $T^2\times S^2$. One can see two thrice punctured tori (left and right) corresponding to each copy of $T^2\times D^2$.}
\label{fig_trisection_T2xS2}
\end{figure}
%%%%%%%%%%%%%%%%%%%%%%%%%%%%%%%%%%%%%%%%%%%%%%%%%%%%%%%%%

%%%%%%%%%%%%%%%%%%%%%%%%%%%%%%%%%%%%%%%%%%%%%%%%%%%%%%%%%
\begin{examp}[Sphere bundles over $RP^2$]
During the last day of the 2019 Spring Trisectors Meeting at UGA, the pair of trisection diagrams of Figure \ref{RP2_1} was discussed. Work of Gay and Meier in \cite{doubly} shows that $B$ is a Gluck twist of $A$ along some embedded 2-sphere. We can use Theorem \ref{pasting_lemma_star} to decompose the 4-manifolds $A$ and $B$ as the union of two 4-manifolds along glued along their boundary: $A=(S^2\times D^2)\cup_f X$ and $B=(S^2\times D^2)\cup_g X$. Remark \ref{remark_complement_loop} and Figure \ref{complement_loop_3} show that $X$ is the complement of a circle in $S^1\times S^3$ representing twice the generator of first homology. It is a nice exercise to see that $X$ is diffeomorphic to the product $S^2\times M^2$ where $M^2$ is a Mobius band. To end, the pasting map $f$ does not twist the $S^2$ component (see Figure \ref{RP2_2}). Hence $A$ is a trisection for the product $S^2\times RP^2$ and $B$ is a Gluck twist of $A$ along a fiber $S^2\times \{pt\}$. Concluding that $B=S^2\wt\times RP^2$. 
\end{examp}
\begin{figure}[h]
\centering
\includegraphics[scale=.06]{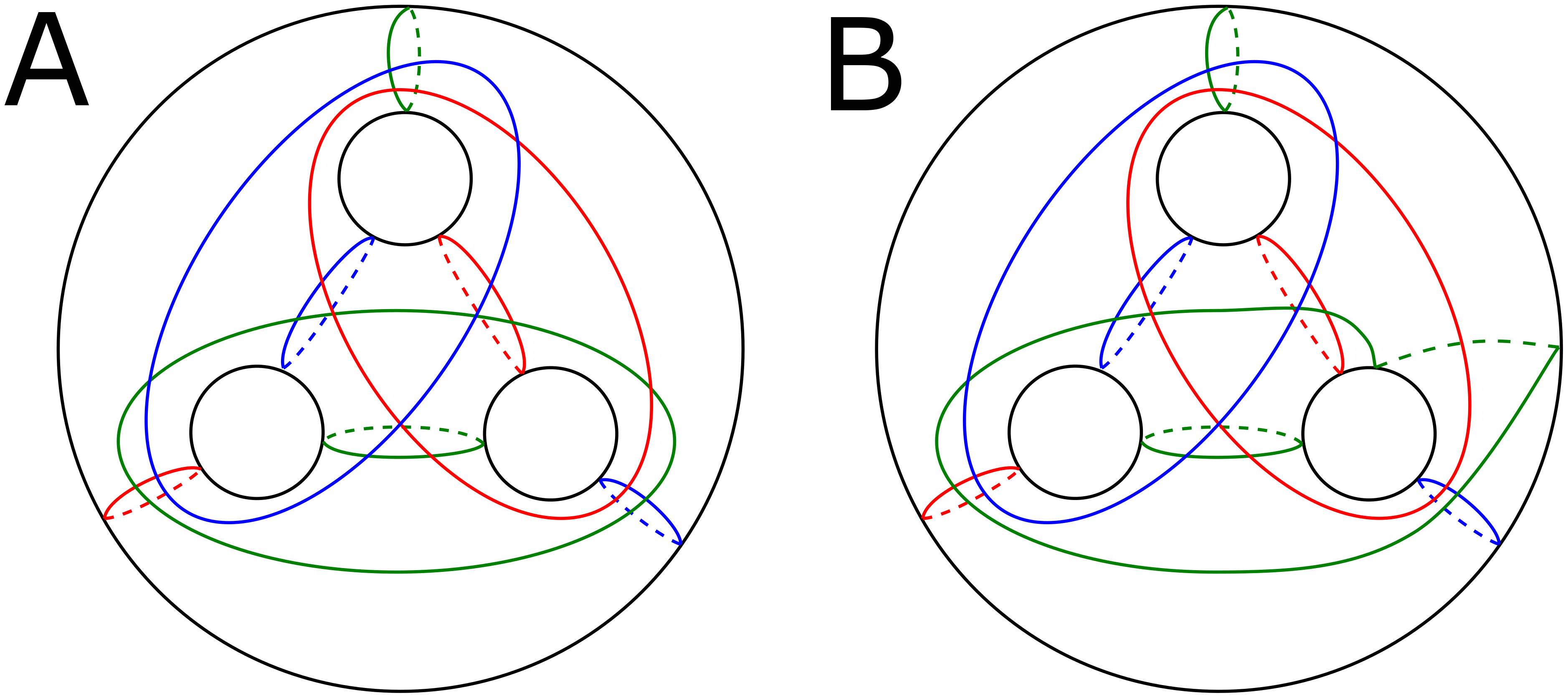}
\caption{A pair of genus 3 trisection diagrams that differ by Gluck twist.}
\label{RP2_1}
\end{figure}
\begin{figure}[h]
\centering
\includegraphics[scale=.06]{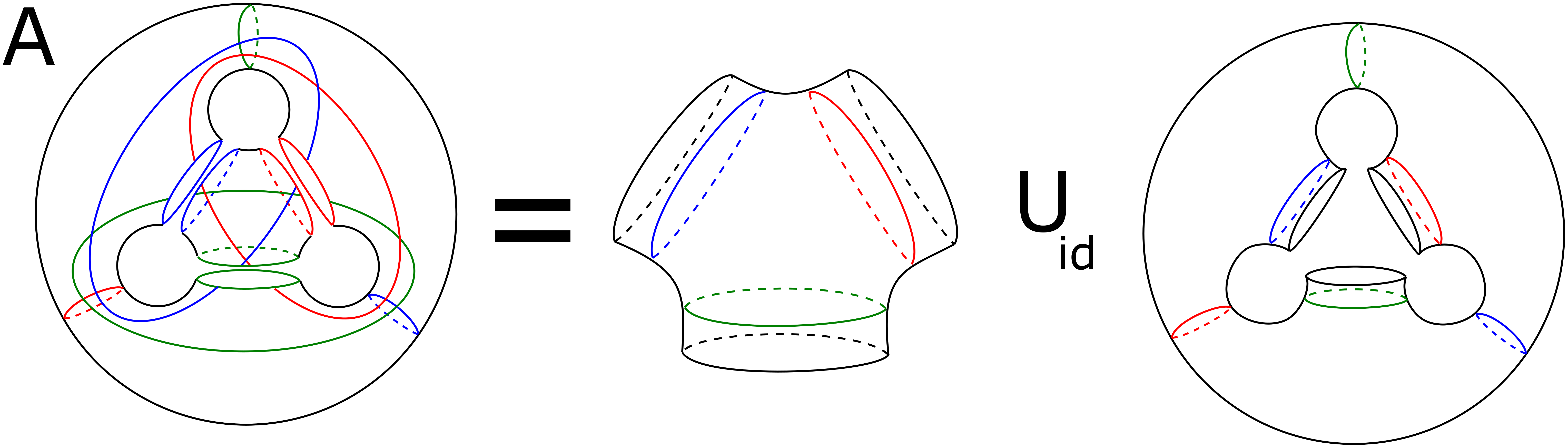}
\caption{Decomposing $A$ as $A=(S^2\times D^2)\cup_f X$. Note that $X$ is a trisection diagram for the complement of a circle in $S^1\times S^3$ representing twice the generator of first homology. (see Section \ref{loops_genus_one}).}
\label{RP2_2}
\end{figure}
%%%%%%%%%%%%%%%%%%%%%%%%%%%%%%%%%%%%%%%%%%%%%%%%%%%%%%%%

%%%%%%%%%%%%%%%%%%%%%%%%%%%%%%%%%%%%%%%%%%%%%%%%%%%%%%%
\section{The Complement of a Simple Closed Curve} 
\label{section_curve_complement}
Let $X$ be a compact 4-manifold with $\star$-trisection $X=X_1\cup X_2\cup X_3$. Let $\Sigma$ be the trisection surface and consider $c\subset X$ a simple closed curve in $X$. Since $\pi_1(\Sigma)\twoheadrightarrow \pi_1(X)$, the curve $c$ is homotopic to an immersed curve $S^1\looparrowright \Sigma$. Given an immersed curve as such, we are interested in finding a $\star$-trisection for $X-\eta(c)$. To accomplish this, we decompose the immersed curve into a union of embedded arcs, push the arcs into the handlebodies, and then remove the tubular neighborhood of each arc.
\begin{defn} \label{def_decomposed_curve}
Given a $\star$-trisection $(\Sigma;\alpha_1,\alpha_2,\alpha_3)$, we say that an immersed curve $c\subset \Sigma$ is decomposed if $c$ is the union of three collections of embedded arcs $c=a_1\cup a_2\cup a_3$ with the property that $a_1\cap\alpha_1=a_2\cap\alpha_2=a_3\cap\alpha_3=\emptyset$ and that each arc in $a_i$ is connected to one arc from each of $a_{i-1}$ and $a_{i+1}$. Denote the discrete set of points $a_{i-1}\cap a_{i+1}$ by $b_i$.
\end{defn}
Starting with a decomposed curve $c$, push each arc of $a_i$ into $H_{\alpha_i}$, leaving the endpoints fixed. We claim that the decomposition $X-\eta(c) = \cup_{i=1}^3 X_i-\eta(c)$ is a $\star$-trisection. Since $c\cap X_i=a_{i-1}\cup a_i \subset \partial X_i$ is a collection of disjoint arcs in the boundary of $X_i$, the complement $\wt X_i :=X_i-\eta(c)$ is diffeomorphic to $X_i$. By construction, the arcs $a_i$ are simultaneously parallel to the boundary of $H_i$, thus $\wt X_i \cap \wt X_{i+1}$ is also a compression body. Hence we have a $\star$-trisection of $X-\eta(c)$. 

We now describe the $\star$-trisection diagram for $X-\eta(c)$ resulting from this procedure. The trisection surface $\wt \Sigma =\cap_i \wt X_i$ is a copy of $\Sigma$ with open disks removed around the endpoints of all the arcs. Start by drawing a $\star$-trisection diagram $(\Sigma; \alpha_1, \alpha_2, \alpha_3;\Delta_{\alpha_1\cap \alpha_2},\Delta_{\alpha_2\cap \alpha_3},\Delta_{\alpha_3\cap \alpha_1})$ for $X$ together with the immersed decomposed curve $c=a_1\cup a_2\cup a_3$. Let $\wt\Sigma$ be the punctured surface $\Sigma-\cup_{i=1}^3\eta(b_i)$. Then the compression body $\wt H_i=\wt X_i \cap \wt X_{i+1}$ can be built from $\wt \Sigma$ by attaching 2-handles along the following curves (see Figure \ref{complement_loop}). 
\begin{enumerate} 
\item The original curves $\alpha_i$, 
\item the boundary parallel curves $\partial\eta(b_i)$ and
\item the non-boundary parallel components of $\partial \eta(a_i)$. 
\end{enumerate} 
For each pair $(i,i+1)$, the new standard picture is given by adding extra loops in $\Delta_{\alpha_i\cap \alpha_{i+1}}$ as in Part (b) of Figure \ref{remark_complement_loop}. 
\begin{figure}[h]
\centering
\includegraphics[scale=.076]{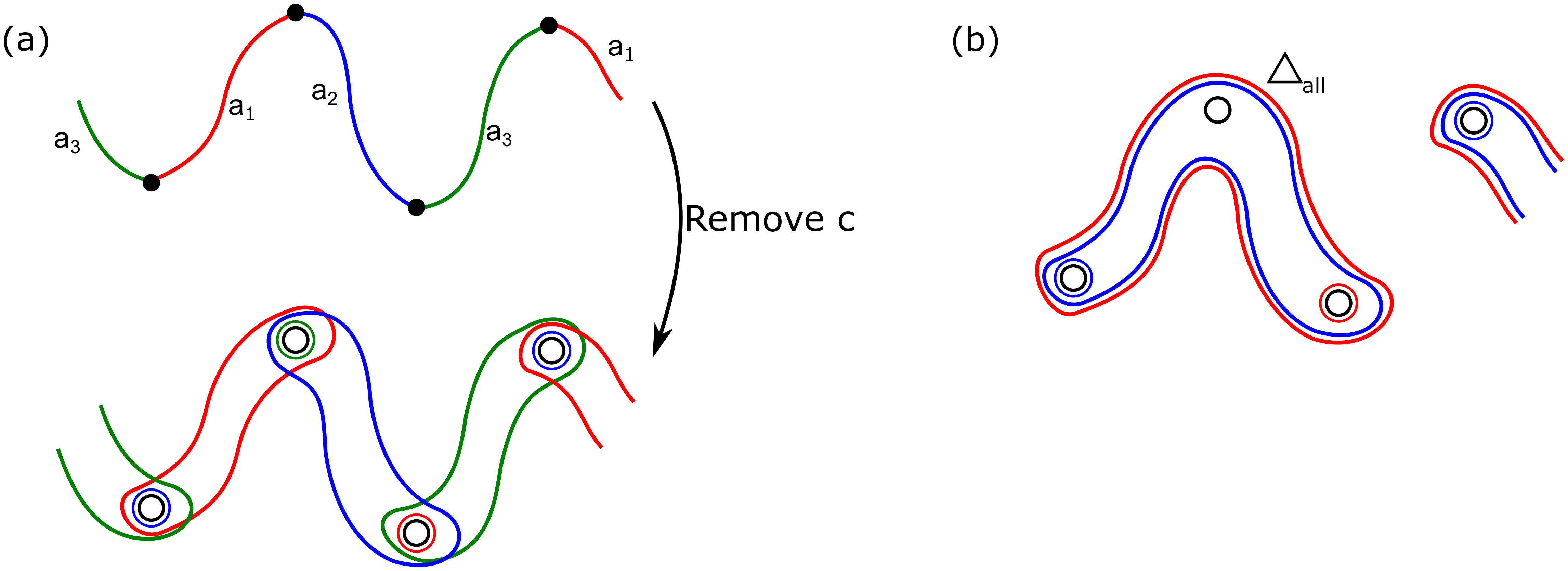}
\caption{How a $\star$-trisection diagram changes when taking the complement of a decomposed curve in $\Sigma$.}
\label{complement_loop}
\end{figure}
\begin{remark} \label{remark_complement_loop} 
If $X$ is closed and we can decompose the curve $c$ so that $|a_i|=1$, then the $\star$-trisection diagram given by the procedure above introduces an unnecessary curve which we can remove. The new $\star$-trisection diagram for the complement is depicted in Figure \ref{complement_loop_2}. To see why this is true, compress $\wt\Sigma$ along all of the original $\alpha$ curves. What remains is a thrice punctured surface. The two new $\alpha$ curves introduced by the procedure above become parallel, thus we can remove one of them.
Figure \ref{complement_loop_3} shows a concrete example of this operation.
\begin{figure}[h]
\centering
\includegraphics[scale=.4]{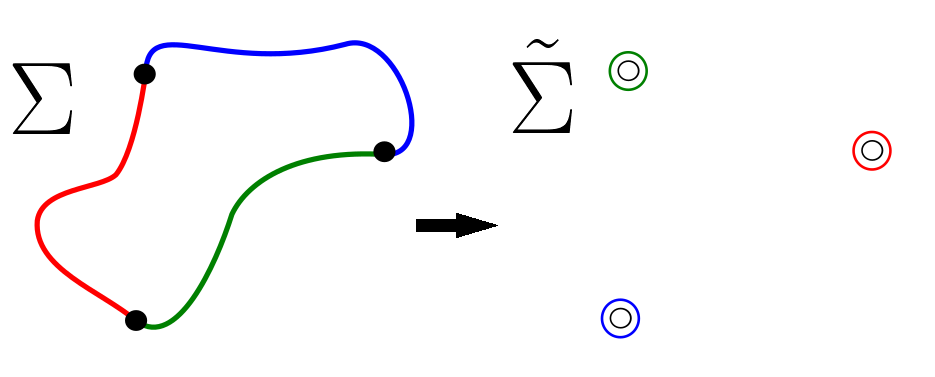}
\caption{How the $\star$-trisection diagram changes if each $|a_i|=1$ and $\partial X=\emptyset$.}
\label{complement_loop_2}
\end{figure}
\begin{figure}[h]
\centering
\includegraphics[scale=.07]{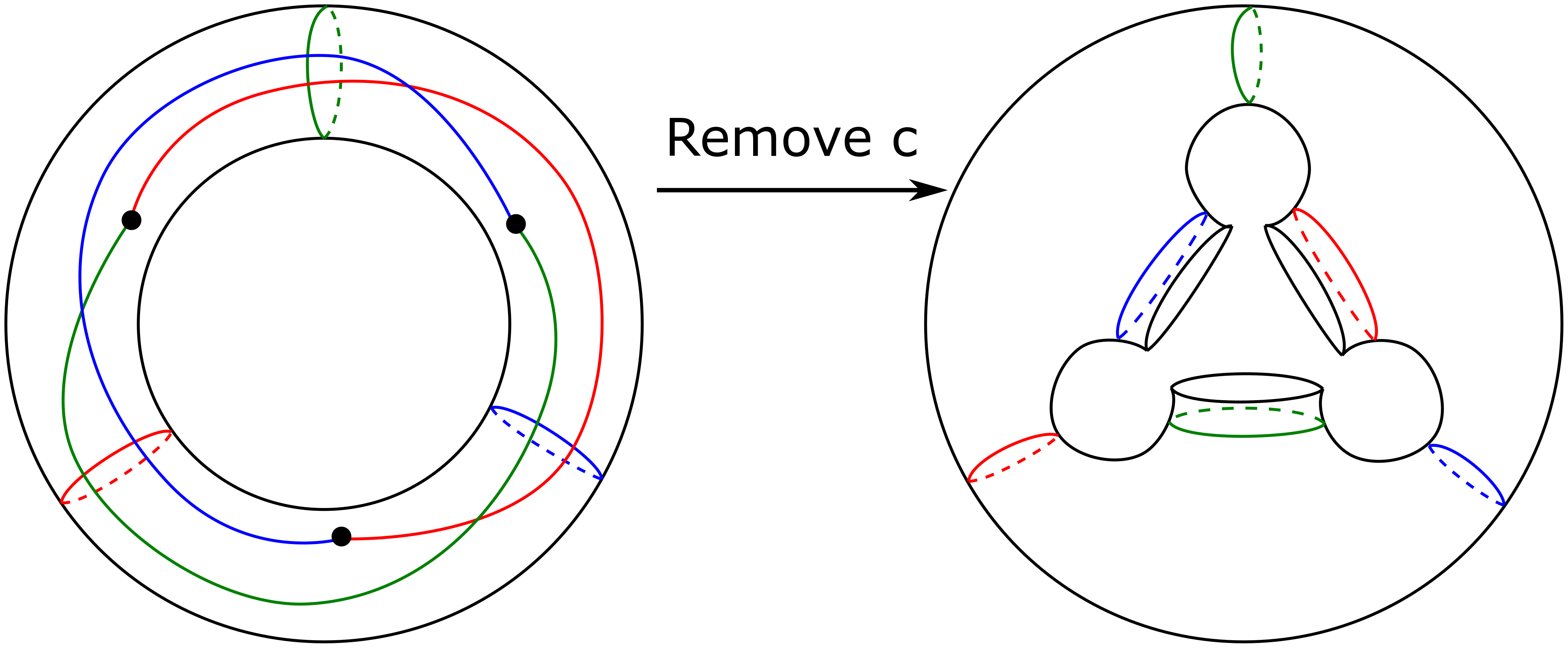}
\caption{A trisection diagram for $S^1\times S^3$ together with a decomposed curve $c$ homotopic to twice the generator $S^1\times \{pt\}$. The right picture is a $\star$-trisection diagram for the complement.}
\label{complement_loop_3}
\end{figure}
\end{remark} 

%%%%%%%%%%%%%%%%%%%%%%%%%%%%%%%%%%%%%%%%%%%%%%%%%%%%%%%%
\subsection{Loops in Genus One Trisections} \label{loops_genus_one} 
Recall that all simple closed curves in a 4-manifold $X$ representing a fixed homotopy class $[c]\in\pi_1(X)$ are isotopic. It is therefore natural to wonder if, for simple $\star$-trisections, any two decomposed curves representing a given class $[c]\in\pi_1(X)$ are slide equivalent in the $\star$-trisection surface. We prove that this is the case for embedded curves in genus one classical trisections. 
The following technical proposition is key to show in Section \ref{section_genus_three} that an infinite family of genus three trisection diagrams is standard.

\begin{prop} \label{prop_genus_one}
Let $(\Sigma; \alpha, \beta, \gamma)$ be a genus one trisection diagram for a closed 4-manifold. Let $c$ be an embedded decomposed curve with $|a_i|=1$ in $\Sigma$. Let $\lbrace \mu,\lambda\rbrace$ be a basis for $\pi_1(\Sigma)$ with $[\alpha]=[\mu]$. If $[c]=[m\mu+n\lambda]$, then by sliding the arcs $a_i$ over the boundaries $b_i=a_{i+1}\cap a_{i-1}$, and sometimes sliding $a_1$ over $\alpha$, $c$ is slide equivalent to an immersed curve representing $[n\lambda]$ with $a_1$ twisting around $b_3$ a total of $m$ times.
\end{prop}
\begin{proof} 

Throughout this argument $\beta$ and $\gamma$ will be pushed around as needed. If $\gamma$ or $\beta$ are in the way of sliding $a_i$ over $b_i$ for $i=1,2$, simply include them in the slide as in the left of Figure \ref{starting_position}.  Since $c$ is embedded and $a_1\cap \alpha=\emptyset$, the $n$ intersections of $c$ with $\alpha$ occur on $a_2\cup a_3$. Isotope $c$ such that $a_2\cap\alpha=\emptyset$. Since $a_1$ and $a_2$ both miss $\alpha$, we may isotope them such that $a_1$ is a small segment leaving $b_2$, $a_2$ is a small segment leaving $b_3$ with $a_1\cup a_2$ being contained in a small disk $D$ disjoint from $\lambda\cup\mu$. This way, all of the intersections of $c$ with $\lambda\cup\mu$ occur on $a_3$. Therefore, since $c$ is embedded in $\Sigma$, $a_3-D$ is a properly embedded arc in $\Sigma-D$. Thus we may assume from the beginning that the trisection as well as the embedded decomposed curve are isotopic to the model in the right of Figure \ref{starting_position} where we have suppressed the  $\beta$ and $\gamma$ curves. At this point in the argument, the $\beta$ and $\gamma$ curves cannot be assumed to lie in any kind of special position with respect to $c$.

\begin{figure}[!h]
\centering
\includegraphics[scale=.5]{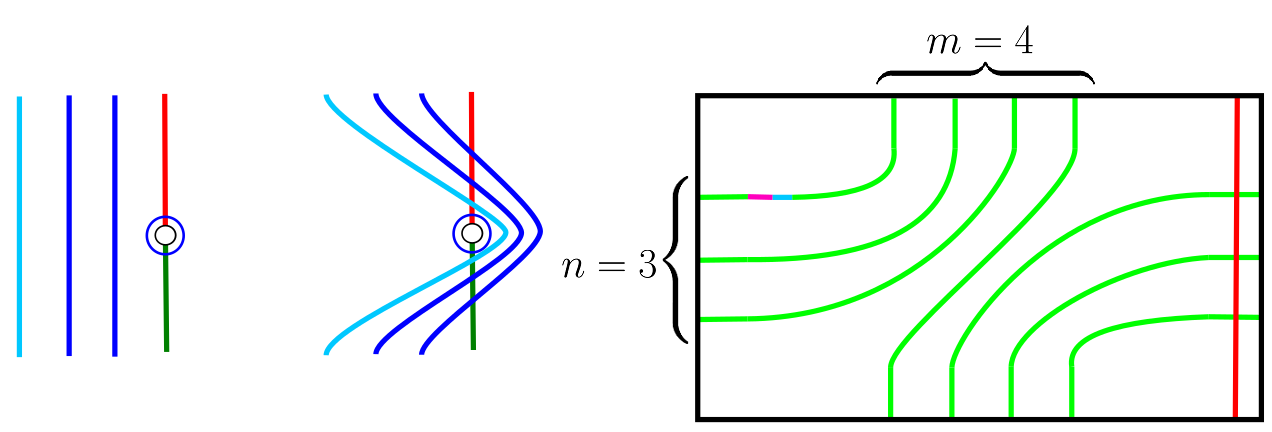}
\caption{(left) How the $\beta$ and $\gamma$ curves behave as we manipulate $c$. (right) The initial position of embedded $c$ with $[c]=[4\mu+3\lambda]$.}
\label{starting_position}
\end{figure}
Let $w$ be the word in the alphabet $\lbrace \mu,\lambda \rbrace$ which represents $[c]$ and decompose $w$ into three subwords $w_1,w_2,w_3$ such that $w=w_1w_2w_3$ and each $w_i$ records the intersections of $a_i$ with $\mu$ and $\lambda$. Initially, as described above, we have that $w_1$ and $w_2$ are empty words and $w_3$ is a certain permutation of the multiset $\lbrace m\mu,n\lambda\rbrace$ which allows $a_3$ to be an embedded arc. We claim that a decomposed $c$ representing $[m\mu+n\lambda]$, $m>0$, in such a way that $w_1$ and $w_2$ are empty words, can be slid to be a representative of $[(m-1)\mu+n\lambda]$ with $w_1$ and $w_2$ being empty words.

Suppose that $c$ represents $[m\mu+n\lambda]$, $m>0$ with $w_1$ and $w_2$ empty. Then since $m>0$, let $j\geq 0$ such that $\lambda^j\mu$ is a prefix of $w_3$. The endpoint of $a_2$ connected to $a_3$ is $b_1$. We can make $b_1$ ``move past" $\alpha$ and $a_1$ by performing the local move in Figure \ref{local_move1}.

\begin{figure}[!h]
\centering
\includegraphics[scale=.4]{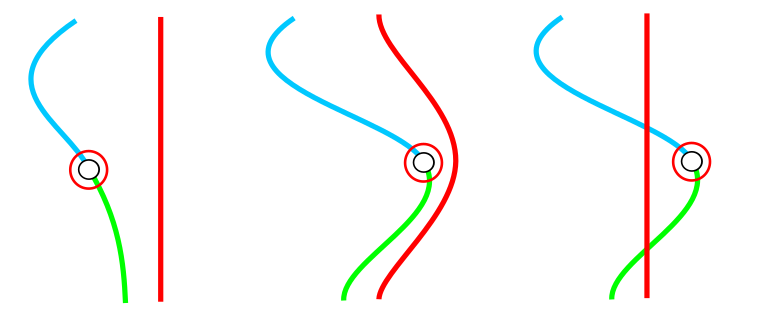}
\caption{How $b_1$ moves past $\alpha$ or $a_1$ in this argument.}
\label{local_move1}
\end{figure}
 We can therefore move $b_1$ along $a_3$ until $w_2$ reads $\lambda^j\mu$, Figure \ref{local_move2}. Using the fact that this is a genus one trisection, and the fact that $\lambda^j\mu$ is embedded, by sliding $a_2$ over $b_2$ when necessary as in Figure \ref{local_move3} we can commute $\mu$ past $\lambda^j$. After an isotopy of $a_1$ and $a_2$, reading from left to right in the two leftmost frames in \ref{local_move4}, we achieve $w_1=\mu$ and $w_2=\lambda^j$. From this position, we can slide $a_1$ against $\alpha$ to yield $w_1=\emptyset$ at the expense of adding a single twist of $a_1$ around $b_3$. This introduces a bigon between $a_1$ and $\lambda$ which is easily resolved. This slide and subsequent isotopy are shown in the last 2 frames.
\begin{figure}[!h]
\centering
\includegraphics[scale=.4]{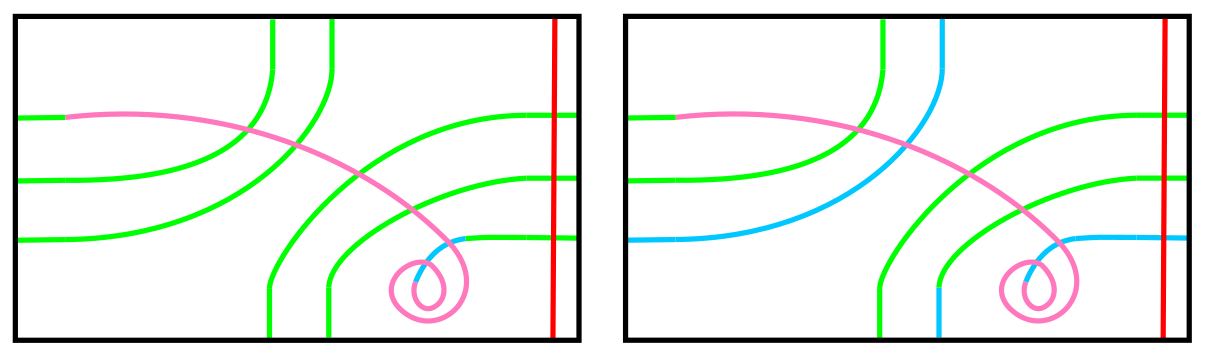}
\caption{In this example $w_2$ consumes $\lambda\mu$ from $w_3$, accomplished by extending $b_2$.}
\label{local_move2}
\end{figure} 
Now it is possible to commute $\mu$ past $\lambda^j$ by sliding $a_2$ over $b_2$ when necessary.

\begin{figure}[!h]
\centering
\includegraphics[scale=.4]{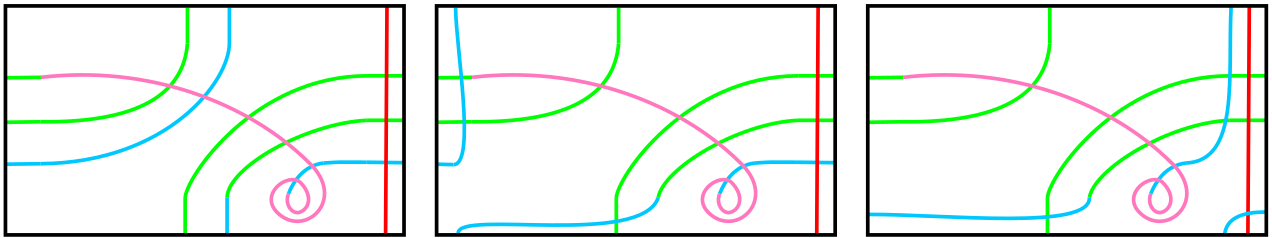}
\caption{A depiction of an isotopy of $a_2$ which corresponds to the rewriting $\lambda\mu\to\mu\lambda$ in $w_2$.}
\label{local_move3}
\end{figure} 

\begin{figure}[!h]
\centering
\includegraphics[scale=.4]{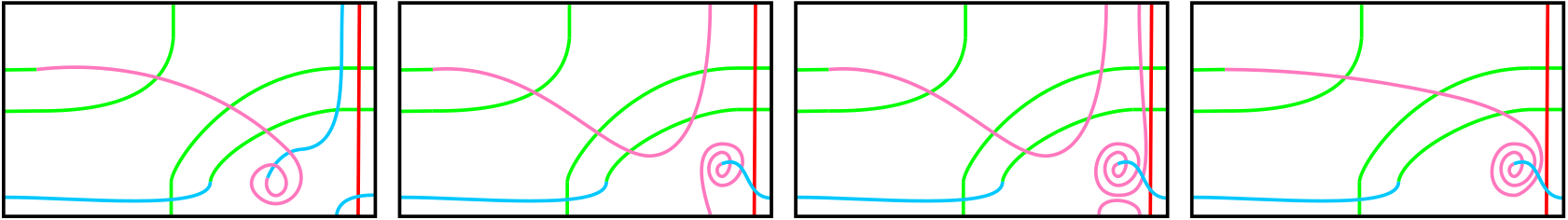}
\caption{Eliminating $\mu$ from $w$ by sliding $a_1$ over $\alpha$.}
\label{local_move4}
\end{figure}  After completing this slide to eliminate $\mu$, we can shrink $a_2$, removing $\lambda^j$ from $w_2$ and appending $\lambda^j$ to the front of $w_3$. The decomposed curve at this stage is a representative of $[(m-1)\mu+n\lambda]$ with $w_1=w_2$ empty, so the claim is proved. By repeating this process, we can slide $c$ to the model representative of $[n\lambda]$ below with $w_1=w_2$ being empty and $w_3=\lambda^l$.

\begin{figure}[!h]
\centering
\includegraphics[scale=.4]{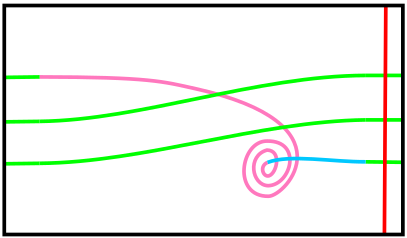}
\caption{The final product.}
\label{final_product}
\end{figure} 
\end{proof} 
Continuing the argument in Proposition \ref{prop_genus_one}. Suppose that $(\Sigma,\alpha,\beta,\gamma)$ is a genus one trisection for a simply connected 4-manifold with $[\lambda]=[\beta]$. Then we can say even more: Since $w=\lambda^n$, we may extend $a_2$ so that $w_2=\lambda^n$ and $w_1=w_3$ are empty words. Putting $\beta$ back into the picture, it is now clear that we can slide $\lambda$ off as well, see Figure \ref{final_product2}. This proves the following. 

\begin{cor}
\label{standard_cor}
Let $(\Sigma; \alpha, \beta, \gamma)$ be a genus one trisection diagram for a simply connected closed 4-manifold. Let $c$ be an embedded decomposed curve in $\Sigma$. Let $\lbrace \mu,\lambda\rbrace$ be a basis for $\pi_1(X)$ with $[\alpha]=[\mu]$ and $[\beta]=[\lambda]$. If $c=m\mu+n\lambda$, then by sliding the arcs $a_i$ over the boundaries $b_i=a_{i+1}\cap a_{i-1}$, sometimes sliding $a_1$ over $\alpha$, and sometimes sliding $a_2$ over $\beta$, $c$ is slide equivalent to an immersed curve representing $1\in\pi_1(\Sigma,b_2)$ with $a_1$ twisting around $b_3$ a total of $m$ times and $a_2$ twisting around $b_1$ a total of $n-1$ times.
\end{cor}

\begin{figure}[!h]
\centering
\includegraphics[scale=.3]{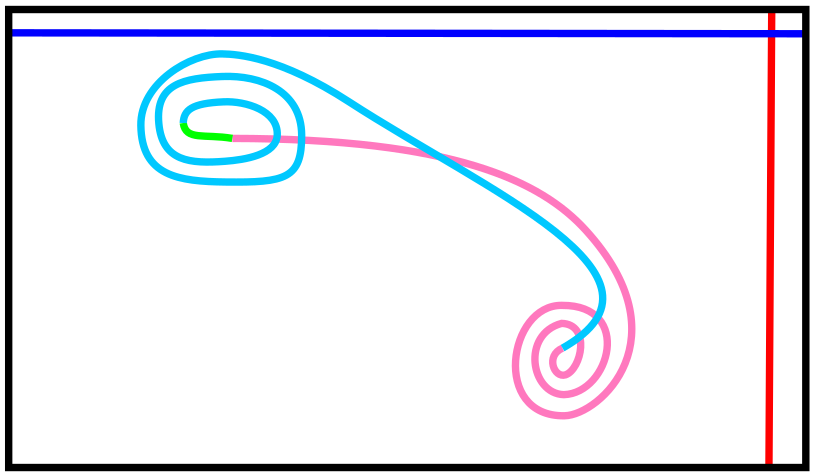}
\caption{ An immersed curve representing $1\in\pi_1(\Sigma,b_2)$ with twists around boundary points.}
\label{final_product2}
\end{figure}

%%%%%%%%%%%%%%%%%%%%%%%%%%%%%%%%%%%%%%%%%%%%%%%%%%%%%%%%%
\section{Trisections of Genus 3}
\label{section_genus_three}
For two irreducible fractions $\frac{a}{b}, \frac{c}{d} \in \mathbb{Q}\cup\{\frac{1}{0}\}$, define $d(\frac{a}{b}, \frac{c}{d}):=\det\big( \begin{smallmatrix}a&c\\b&d\end{smallmatrix}\big)$. 
Given an ordered triple of rational numbers $\frac{a}{b},\frac{c}{d},\frac{p}{q}\in \mathbb{Q}\cup\{\frac{1}{0}\}$, we can consider the diagram $D(\frac{a}{b},\frac{c}{d},\frac{p}{q})$ as in the right side of Figure \ref{fig_farey_trisections}. Here one curve of each $\alpha$, $\beta$, $\gamma$ set has slope $\frac{a}{b}$, $\frac{c}{d}$, $\frac{p}{q}$, respectively, in the torus obtained by compressing the genus three surface along the two central curves of the same color as in the left side of Figure \ref{fig_proof_farey_trisections}. 
The diagram $D(\frac{a}{b},\frac{c}{d},\frac{p}{q})$ is a trisection diagram for some closed smooth 4-manifold if and only if each pair $x,y\in \{\frac{a}{b},\frac{c}{d},\frac{p}{q}\}$ satisfies the inequality $|d(x,y)|\leq 1$. 
If the three numbers in the triplet are all distinct with $d(x,y)=\pm 1$ for all $x\neq y\in \{\frac{a}{b},\frac{c}{d},\frac{p}{q}\}$, then we call $\{\frac{a}{b},\frac{c}{d},\frac{p}{q}\}$ a \textbf{Farey triplet}. In this case, $\{\frac{a}{b},\frac{c}{d},\frac{p}{q}\}$ corresponds to a triangle in the Farey graph. 
\begin{figure}[h]
\centering
\includegraphics[scale=.5]{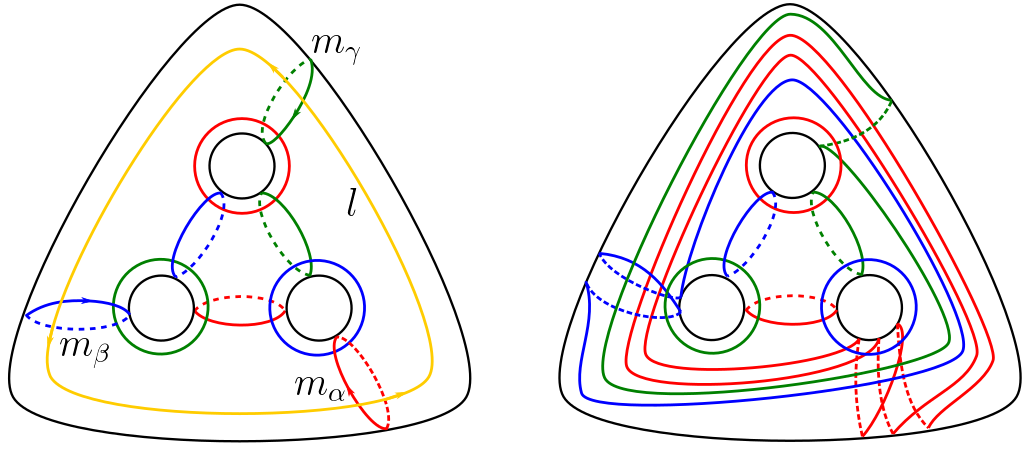}
\caption{(left) The longitude (l) and meridian (m) for each of the three tori. (right) The diagram $D(\frac{2}{3}, \frac{1}{1},\frac{1}{2})$.}
\label{fig_farey_trisections}
\end{figure}

The question we will discuss now is what 4-manifolds the diagrams $D(\frac{a}{b},\frac{c}{d},\frac{p}{q})$ represent. Meier proved in \cite{spun_trisections} that $D(\frac{q}{p},\frac{q}{p},\frac{q}{p})$ is the diagram of a spun lens space $L(p,q)$. He conjectured that the only 4-manifolds admitting genus three trisections are spun lens spaces and certain connected sums of combinations of $S^1\times S^3$, $S^2\times S^2$, $\CP2$, and $\overline{\CP2}$. We will call the latter combinations standard manifolds. 
Note that $D(\frac{a}{b},\frac{c}{d},\frac{p}{q})$ is simply connected whenever $\{\frac{a}{b},\frac{c}{d},\frac{p}{q}\}$ contains two or three distinct numbers. Thus, such diagrams must represent standard manifolds if we expect the conjecture to be true. We prove that this is indeed the situation. 

\begin{thm} \label{thm_farey_trisections}
Let $\frac{a}{b},\frac{c}{d},\frac{p}{q}\in \mathbb{Q}\cup \{\frac{1}{0}\}$ satisfying $|d(x,y)|\leq 1$ for all $x, y \in \{\frac{a}{b},\frac{c}{d},\frac{p}{q}\}$. Then $D(\frac{a}{b},\frac{c}{d},\frac{q}{p})$ describes a trisection diagram for
\begin{enumerate} 
\item either $\CP2\#\CP2\#\overline{\CP2}$ or $\CP2\#\overline{\CP2}\#\overline{\CP2}$ if $\{\frac{a}{b},\frac{c}{d},\frac{p}{q}\}$ is a Farey triplet,
\item either $S^2\times S^2$ or $S^2\wt \times S^2$ if $\{\frac{a}{b},\frac{c}{d},\frac{p}{q}\}=\{x,y\}$ with $d(x,y)=\pm 1$,
\item a spun lens space if $\{\frac{a}{b},\frac{c}{d},\frac{p}{q}\}=\{x\}$.
\end{enumerate}
\end{thm}
To find out the specific 4-manifold the above diagrams represent, it is enough to compute its intersection matrix using \cite{homology_trisection} or \cite{torsion_trisection}. 
\begin{proof} 
The third part was done by Meier in \cite{spun_trisections}. Denote by $X$ the 4-manifold represented by the diagram $D(\frac{a}{b},\frac{c}{d},\frac{p}{q})$. Notice that we can decompose the genus three surface in Figure \ref{fig_farey_trisections} into a thrice punctured sphere and a thrice puntured torus glued together along their boundaries, see Figure \ref{fig_proof_farey_trisections}. Theorem \ref{pasting_lemma} implies that $X$ decomposes as the union $X=(S^2\times D^2) \cup_\partial Y$ for some 4-manifold $Y$.
\begin{figure}[h]
\centering
\includegraphics[scale=.5]{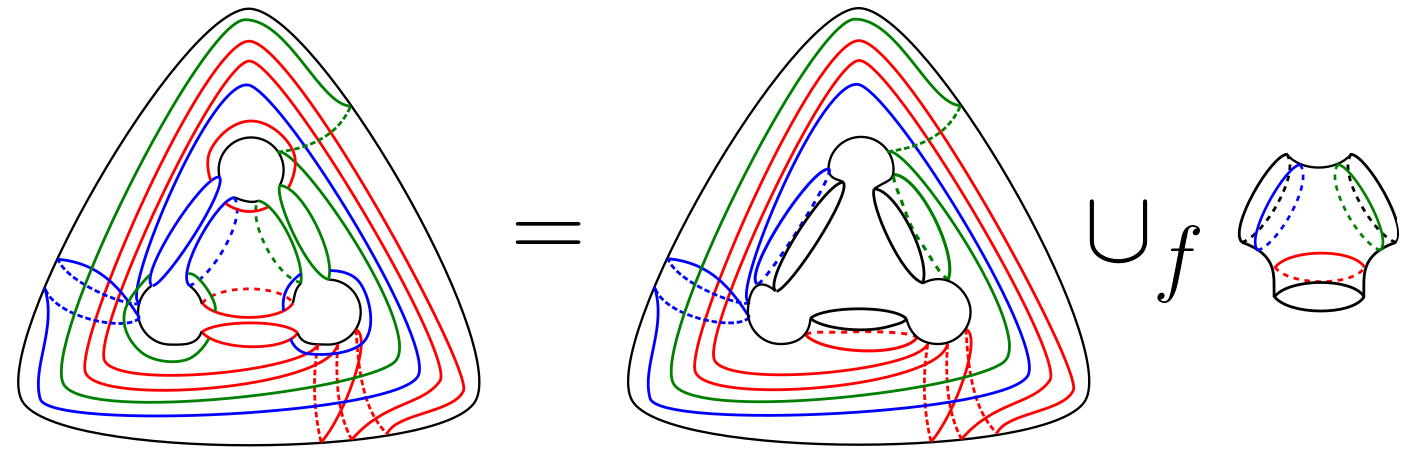}
\caption{Decomposing $X$ as $(S^2\times D^2) \cup_\partial Y$ glued via some map $f:S^1\times S^2 \ra \partial Y$.}
\label{fig_proof_farey_trisections}
\end{figure}

Remark \ref{remark_complement_loop} implies that the $\star$-trisection diagram corresponding to $Y$ can be viewed as the result of taking the complement of a decomposed curve in a genus one trisection. So if $\{\frac{a}{b},\frac{c}{d},\frac{p}{q}\}$ is a Farey triplet, the diagram for $Y$ implies that $Y$ is the complement of a curve $c$ in $\CP2$ (or $\overline{\CP2}$). Since $\CP2$ is simply connected, we can isotope $c$ to lie inside a small 4-ball and so $Y=\CP2\# (S^2\times D^2)$. Hence $X$ is the connected sum of $\CP2$ with a sphere bundle over the sphere. 
Suppose now $\{\frac{a}{b},\frac{c}{d},\frac{p}{q}\}=\{x,y\}$ with $d(x,y)=\pm1$, then Remark \ref{remark_complement_loop} implies that $Y$ is the complement of a decomposed loop in a genus one trisection for $S^4$. Thus $Y = S^4-(S^1\times B^3)$ and $X$ is a copy of $S^2\times S^2$ or $S^2\wt \times S^2$. 

We have shown the 4-manifolds associated to the trisection diagrams $D(\frac{a}{b},\frac{c}{d},\frac{p}{q})$ are diffeomorphic to standard ones. To figure out which ones specifically, it is sufficient to compute the intersection form $Q_X$.
The intersection form of $X$, computed from $D(\frac{a}{b},\frac{c}{d},\frac{p}{q})$ using \cite{homology_trisection}, is given by $$Q_X=
\begin{bmatrix}
 \frac{b d}{a d-b c} & -1 & \frac{b (c q-d p)}{b c-a d} \\
 -1 & 0 & 0 \\
 \frac{b (c q-d p)}{b c-a d} & 0 & \frac{(b p-a q) (c q-d p)}{b c-a d} \\
\end{bmatrix}.$$
In the case that $\frac{c}{d}=\frac{p}{q}$, notice the third column and third row become zero and we are left with the intersection form $Q_X=
\begin{bmatrix}
 \frac{b d}{a d-b c} & -1  \\
 -1 & 0 
\end{bmatrix}$ which is equivalent to the intersection form for $S^2\times S^2$ when $bd$ is even and to $S^2\wt \times S^2$ when $bd$ is odd. If all rationals are distinct, then without loss of generality suppose that $d(\frac{a}{b},\frac{c}{d})=1$. Because $\{\frac{a}{b},\frac{c}{d},\frac{p}{q}\}$ is a Farey triple, we know that $\frac{p}{q}=\frac{a\pm c}{b \pm d}$. This gives $$Q_X=
\begin{bmatrix}
 bd & -1 & b \\
 -1 & 0 & 0 \\
 b & 0 & \mp 1 \\
\end{bmatrix}.$$
By inspection, one determines that $Q_X$ is equivalent to $\langle 1\rangle\oplus \langle \mp 1\rangle\oplus \langle -1\rangle$.
\end{proof} 

%%%%%%%%%%%%%%%%%%%%%%%%%%%%%%%%%%%%%%%%%%%%%%%%%%%%%%%%%%
\subsection{Farey Trisections are Standard}
We will now demonstrate that in cases 1 or 2 of Theorem \ref{thm_farey_trisections}, the diagrams $D(\frac{a}{b},\frac{c}{d},\frac{p}{q})$ are actually reducible and thus standard.
\begin{thm}\label{thm_farey_trisections_std}\label{thm_farey_std}
Let $\{\frac{a}{b},\frac{c}{d},\frac{p}{q}\}\subset\mathbb{Q}\cup\{\frac{1}{0}\}$ with $d(x,y)\leq 1$ for each $x,y\in \{\frac{a}{b},\frac{c}{d},\frac{p}{q}\}$. If at least two of these fractions are distinct, then $D(\frac{a}{b},\frac{c}{d},\frac{p}{q})$ is handle slide equivalent to the standard diagram for $T\# S$ where $T\in \lbrace S^4,\CP2,\overline{\CP2}\rbrace$ and $S\in\lbrace S^2\times S^2, S^2\wt \times S^2\rbrace$.
\end{thm}
\begin{proof} Decompose $D(\frac{a}{b},\frac{c}{d},\frac{p}{q})$ into two pieces as suggested by Figure \ref{fig_proof_farey_trisections}; let $D'$ denote the thrice punctured torus component of this decomposition and let $P$ denote the thrice punctured sphere component. Notice that this $\star$-trisection $D'$ is a diagram for the complement of an embedded curve in a simply connected 4-manifold. Specifically, $D'$ is the result of taking the complement of the $c=\lambda$ curve in a genus 1 trisection with curves $\alpha=a\lambda+b\mu$, $\beta=c\lambda+d\mu$, and $\gamma=p\lambda+q\mu$ where the curve $c$ has been decomposed as suggested by the left side of Figure \ref{fig_proof_farey_trisections}. By Corollary \ref{standard_cor}, the decomposed curve $c$ is slide equivalent to an immersed decomposed curve $c'$ representing the trivial curve where some of the arcs twist around the boundary points, as in the last frame of Figure \ref{final_product2}. In particular, $\alpha$ and $\beta$ are disjoint from $c'$. Using the fact that $c'$ represents the trivial loop, we can slide $\gamma$ against $b_3$ until $\gamma$ is disjoint from $c'$ also. Thus there is a curve $\delta$ separating $c'$ from $\alpha,\beta,\gamma$. By surgering $D'$ along $\delta$ we get two components. Let $Q$ be the component coming from the side of $\delta$ containing $c'$ and let $T$ be the torus component containing $\alpha,\beta,\gamma$. Notice that $S=Q\cup_f P$, with the attaching map $f$ given by the $a_i$, is a genus two trisection of a closed 4-manifold with an intersection form of full rank. By the work of Meier and Zupan \cite{genus_two_std}, $S$ is a trisection diagram for $S^2\times S^2$ or $S^2\wt\times S^2$. The component $T$ is a genus one trisection of a simply connected closed 4-manifold: $S^4,\CP2$ or $\overline{\CP2}$.
\end{proof}
%%%%%%%%%%%%%%%%%%%%%%%%%%%%%%%%%%%%%%%%%%%%%%%%%%%%%%%%%
\begin{remark}[Spun lens spaces]
In \cite{spun_trisections}, Meier asked if the diagrams $D(\frac{p}{q},\frac{p}{q},\frac{p}{q})$ and $D(\frac{1}{q},\frac{1}{q},\frac{1}{q})$ depict diffeomorphic trisections. 
With this in mind, we can proceed as in Theorem \ref{thm_farey_trisections_std} and see that Proposition \ref{prop_genus_one} implies that the diagrams $D(\frac{p}{q},\frac{p}{q},\frac{p}{q})$ and $D(\frac{1}{q},\frac{1}{q},\frac{1}{q})$ are handle slide equivalent to diagrams which are identical outside of the regular neighborhood of a $\gamma$ curve, say $\gamma_0$. 
In this annulus $\nu(\gamma_0)$, the diagrams differ by their $\alpha$ curves, where one twists $q$ times around this annulus and the other twists once (see Figure \ref{fig_spun_lens_space}). 
This motivates the following question about uniqueness of trisection diagrams for 1-surgeries.
\end{remark}
\begin{question} \label{question_Gluck}
Let $c$ be an embedded loop in a 4-manifold $X$ represented by a (possibly immersed) decomposed curve in the genus $g$ trisection surface. If $|a_i|=1$ in the decomposition of the curve, is the resulting genus $g+2$ trisection diagram for $\left(X-\eta(c)\right)\cup(S^2\times D^2)$, with a specific choice of framing in $\Z/2\Z$, unique up to handle slides and diffeomorphisms of the surface? 
In particular, are the trisections given by the diagrams in Figure \ref{fig_spun_lens_space} diffeomorphic?
\begin{figure}[h]
\centering
\includegraphics[scale=.07]{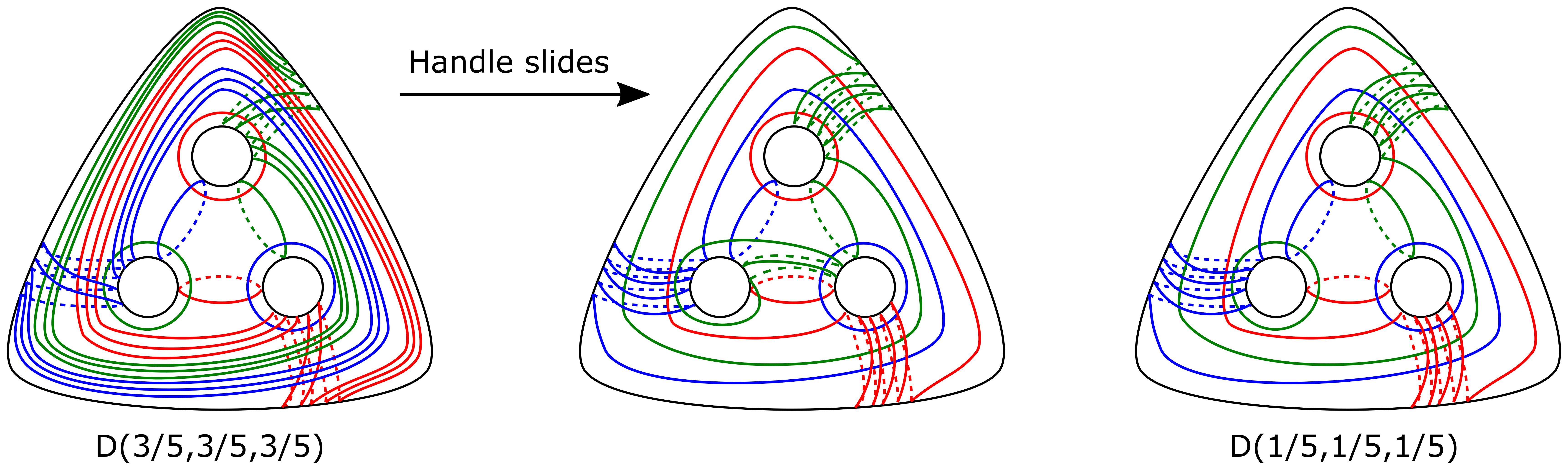}
\caption{How the handle slides described in Proposition \ref{prop_genus_one} change the diagram $D(\frac{3}{5},\frac{3}{5},\frac{3}{5})$. The middle and right diagrams difeer by removing an embedded $S^2\times D^2$ and regluing it using a power of the Gluck twist map. Question \ref{question_Gluck} asks if whether these two diagrams depict diffeomorphic trisections for even powers the Gluck twist map.}
\label{fig_spun_lens_space}
\end{figure}
\end{question}

%%%%%%%%%%%%%%%%%%%%%%%%%%%%%%%%%%%%%%%%%%%%%%%%%%%%%%%%%%
\section{Surface Surgery} 
\label{section_surface_surgery}
The main goal of this section is to use $\star$-trisections to draw diagrams for closed 4-manifolds obtained by various kinds of surgery. 
Drawings for these 4-manifolds could be theoretically derived using previous work on relative trisections. % cite(pasting and juanita's). 
However, the genus of  these diagrams will often be large. 
For example, in \cite{trisections_via_lefschetz}, a genus seven trisection of $T^2\times S^2$ was obtained by taking the double of a genus 3 relative trisection for $T^2\times D^2$. 
In Figure \ref{fig_trisection_T2xS2}, we used a genus one $\star$-trisection for $T^2\times D^2$ to draw a genus four trisection for $T^2\times S^2$. 
This section provides diagrams for the Cacime Surface (Section \ref{section_cacime}), and algorithms to perform Fintushel-Stern knot surgery (Section \ref{section_knot_surgery}) and torus surgeries (Section \ref{section_torus_surgery}) such as Logarithmic transforms and Luttinger transforms. 
The careful reader might observe that the diagrams for these transformations change by concatenating a fixed picture or by changing some loops in a high enough stabilization of the original trisection diagram. Thus to study the behavior of 4-manifold invariants under surface surgery, it could be worthwhile to explore each local modifications in detail. 

%One possible application of $\star$-trisections is to obtain low genus trisection diagrams for 4-manifolds obtained by gluing smalles pieces. One if this is the following: 
%In \cite{trisections_via_lefschetz} a genus seven trisection of $T^2\times S^2$ was obtained by taking the double of a genus 3 relative trisection for $T^2\times S^2$. In Figure \ref{fig_trisection_T2xS2}, we use $\star$-trisections to draw a genus four trisection for $T^2\times S^2$. 
%The main goal of this section is to use $\star$-trisections to draw trisection diagrams for 4-manifolds obtained by various surgeries on surfaces.

%%%%%%%%%%%%%%%%%%%%%%%%%%%%%%%%%%%%%%%%%%%%%%%%%%%%%%%%
\subsection{Embedded Surfaces and their Complements}\label{section_surface_complement}
Let $X=X_1\cup X_2\cup X_3$ be a $\star$-trisected 4-manifold and let $F\subset X$ be an embedded closed surface. Following Meier and Zupan \cite{bridge_trisections_4M}, we say that $F$ is in \textbf{$(g;c_1,c_2,c_3)$-bridge position} with respect to the $\star$-trisection if, for each $i \neq j$, $\mathcal{D}_i=F\cap X_i$ is a collection of $c_i$ trivial disks in $X_i$, and the arcs $\mathcal{D}_i\cap \mathcal{D}_j$ form a trivial $b$-tangle in the compression body $C_{ij}=X_i\cap X_j$. As a consequence, the intersection $F\cap (X_1\cap X_2\cap X_3)$ is a collection of $2b$ points. 

Given a $\star$-trisection diagram $(\Sigma;\alpha,\beta, \gamma)$, we can decode\footnote{This is a consequence of Lemma \ref{lem_unique_trivial_disks}.} a $(b;c_1, c_2, c_3)$-bridge trisection of $F$ by three sets of $b$ embeded arcs $s_{\alpha}$, $s_{\beta}$, $s_{\gamma}$ in $\Sigma$ corresponding to the shadows of the trivial tangles $F\cap C_{\varepsilon}$, $\varepsilon\in\{\alpha,\beta,\gamma\}$. The shadow arcs have $2b$ common endpoints $t=F\cap \Sigma$. We consider the arcs in $s_{\varepsilon}$ to be disjoint from the loops in $\varepsilon$. Thus, isotopy of the arcs $F\cap C_{\varepsilon}$ relative to their boundaries corresponds to sliding the shadows $s_{\varepsilon}$ over $\varepsilon$. 
For each pair $(\eps, \mu)\in \{(\alpha, \beta), (\beta, \gamma), (\gamma, \alpha)\}$, the arcs $s_{\eps}$, $s_\mu$ in $\Sigma$ determine a $c_{i}$-component unlink in bridge position with respect to the Heegaard splitting $C_\eps\cup_{\Sigma} C_\mu$, where $i\in \{1,2,3\}$ is the index corresponding to the pair $(\eps, \mu)$. 
For more details and examples of bridge trisections see \cite{bridge_trisection_S4,bridge_trisections_4M,rational_surfaces}. 

Given a bridge trisected surface $F\subset X$, there is an obvious $\star$-trisection for the complement $X-\eta(F)$ given by $\wt{X}_i=X_i-\eta(\mathcal{D}_i)$, $i=1,2,3$. Let $s_{\alpha}$, $s_{\beta}$, $s_{\gamma}$ be a set of shadows for $F$ in the $\star$-trisection diagram $(\Sigma;\alpha,\beta,\gamma)$. 
A $\star$-trisection diagram for $X-\eta(F)$ is given by $(\wt\Sigma;\wt\alpha,\wt\beta,\wt\gamma)$ where $\wt \Sigma=\Sigma-\eta(t)$ is a copy of $\Sigma$ with $|t|$ disks removed, and for each $\eps \in \{\alpha,\beta,\gamma\}$, $\wt\eps=\eps \cup \eps'$ where the extra loops in $\eps'$ are obtained from the non-boundary parallel components of $\partial \eta\left(s_\eps \cup \eta(t)\right)$. 
Using the notation of Remark \ref{remark_diagrams_standard_piece}, the curves in $\eps'$ correspond to new curves of type 2. 
%For each pair $(\eps, \mu)\in \{(\alpha, \beta), (\beta, \gamma), (\gamma, \alpha)\}$, the arcs $s_{\eps}$, $s_\mu$ in $\Sigma$ determine an unlink in bridge position with respect with the Heegaard diagram $(\Sigma; \eps, \mu)$. Thus, the new common curves in $\Delta_{\wt \eps \cap \wt \mu}$ are obtained by adding to $\Delta{\eps\cap \mu}$ one loop surrounding each 
See Figure \ref{fig_surface_complement} for a concrete example. 
In the rest of the section we will see how we can use $\star$-trisection diagrams to draw diagrams for surgeries along bridge trisected surfaces.

\begin{remark}
It is important to mention that this decomposition was previously discussed by Kim and Miller in \cite{surface_complements}. The authors of \cite{surface_complements} observed that the above $\star$-trisection is a classic trisection only when $F$ is a 2-sphere. To obtain a classical relative trisection, Kim and Miller performed a sequence of ``boundary-stabilizations" to the $\star$-trisection above. 
This procedure increases the complexity of the trisection surface in a controlled way, and in principle can also be used to perform surgery along surfaces following the methods in this section.
\end{remark}
\begin{figure}[!h]
\centering
\includegraphics[scale=.3]{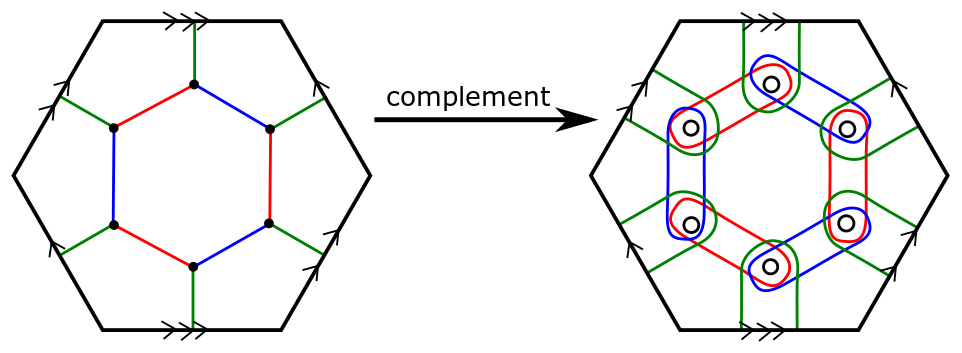}
\caption{(left) A bridge trisection for $T^2\times\{pt\}$ inside $T^2\times D^2$. (right) A $\star$-trisection for $T^2\times S^1\times[0,1]=T^3\times [0,1]$.}
\label{fig_surface_complement}
\end{figure} 
\begin{examp}[Trisecting the 4-torus]
Figure \ref{fig_surface_complement} shows a $\star$-trisection diagram for $T^3\times[0,1]$ obtained from removing an open neighborhood of $T^2\times\{0\}$ inside $T^2\times D^2$. We can poke this diagram in order to satisfy the boundary conditions of Lemma \ref{pasting_lemma_star}. To build $T^4$, we glue two copies of $T^3\times[0,1]$ along the identity in their boundaries. Diagram-wise, this corresponds to `double' the poked diagram for $T^3\times[0,1]$ and add some extra loops using Remark \ref{remark_diagramatics_pasting_lemma}. For each color $\eps \in \{\alpha, \beta, \gamma\}$, we draw embedded arcs disjoint from $\eps$ cutting the compressed surface $\Sigma_\eps$ into a disjoint union of disks. The new curves are given by gluing two copies of the arcs, one on each surface, along their endpoints as in Figure \ref{fig_trisecting_T4_2}. The resulting diagram represents a (10;4)-trisection for the 4-torus. 
\end{examp}
\begin{figure}[!h]
\centering
\includegraphics[scale=.3]{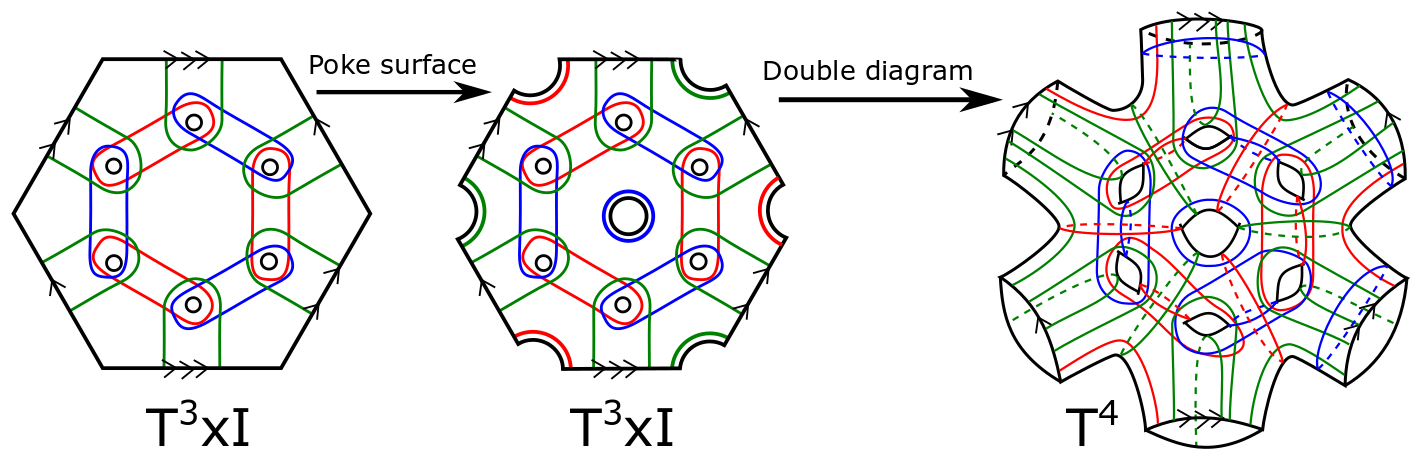}
\caption{Trisecting the 4-torus.}
\label{fig_trisecting_T4_1}
\end{figure} 
\begin{figure}[!h]
\centering
\includegraphics[scale=.25]{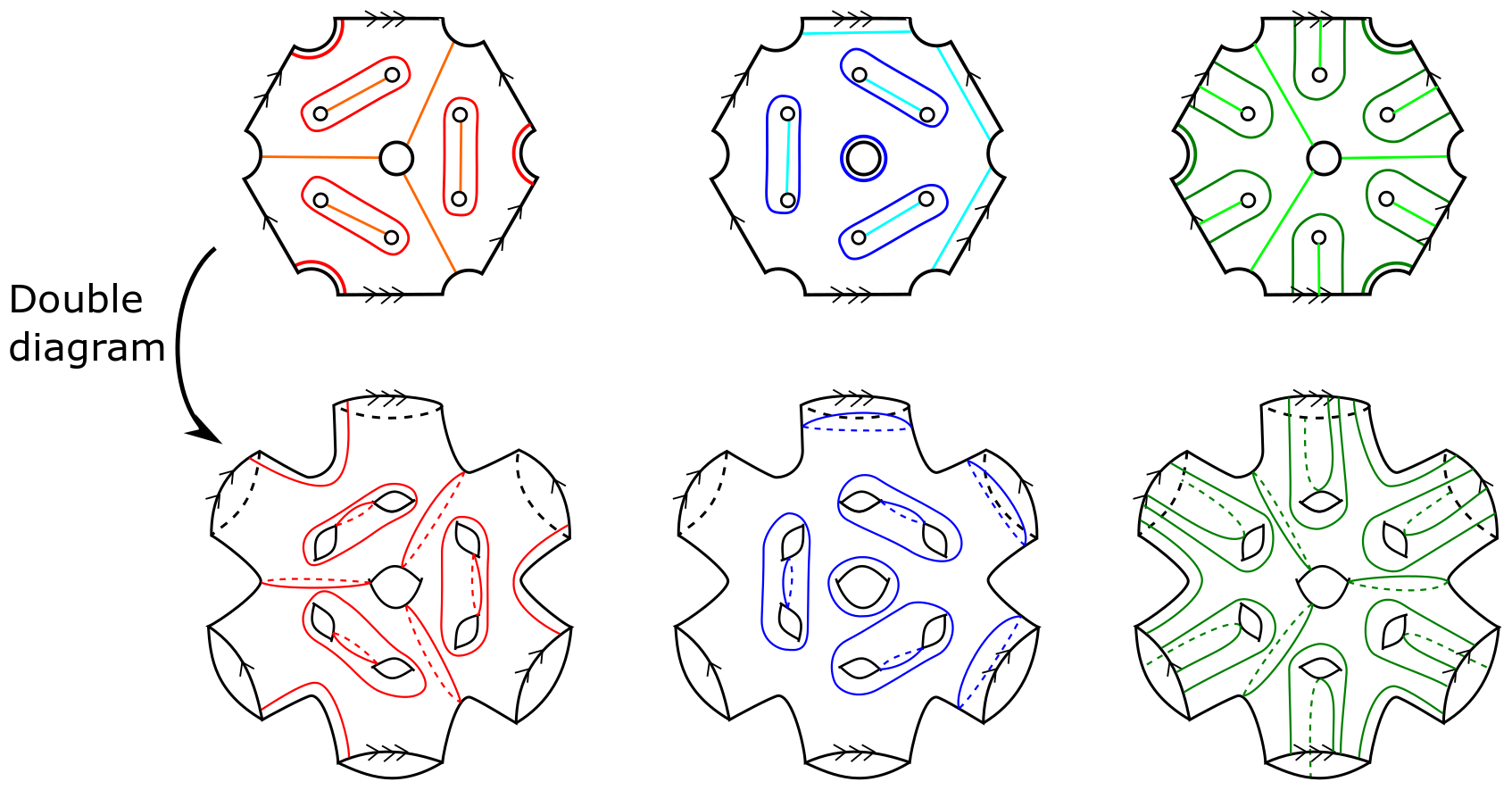}
\caption{How to double the diagram for $T^3\times[0,1]$.}
\label{fig_trisecting_T4_2}
\end{figure} 
The following proposition explains how obtain a trisection for fiber sums of two 4-manifolds. 
\begin{prop}\label{prop_fiber_sum}
Let $X$ and $X'$ be two closed connected 4-manifolds. Let $\tau=W_1\cup W_2 \cup W_3$ (resp. $\tau'$) be a $(g;k_1,k_2,k_3)$-trisection for $X$ (resp. $X'$). Suppose $F\subset X$ (resp. $F'\subset X'$) is a closed embedded surface in $(b;c_1,c_2,c_3)$-bridge position with respect to $\tau$ (resp. $\tau'$). Suppose $F\cdot F = F'\cdot F'$ and that the cell decompositions in $F$ and $F'$ induced by the bridge trisections agree; in other words, $b=b'$ and $c_i=c'_i$ for $i=1,2,3$. Then the fiber sum of $X$ and $X'$ along $F$ and $F'$, denoted by $X \natural X'$, admits a $(G; K_1, K_2, K_3)$-trisection where $G=g+g'+2b-1$ and $K_i=k_i + k'_i + c_i$.
\end{prop}
\begin{proof}
The condition $F\cdot F = F'\cdot F'$ is so that the gluing map $f: \partial \eta F \ra \partial \eta F'$ exists. 
To trisect the fiber sum of $X$ and $X'$ one must trisect the complements $X-\eta(F)$ and $X'-\eta(F')$ using the bridge trisections. Since $b=b'$ and $c_i=c'_i$ for $i=1,2,3$, the handle decompositions in $\partial \eta (F)$ and $\partial \eta (F')$ induced by the $\star$-trisections agree. Thus we can apply Lemma \ref{pasting_lemma_star} and obtain a trisection for $X\natural X'$. 
The $\star$-trisection surface for $X-\eta(F)$ is connected of genus $g$ and $2b$ boundary components; similarly for $X'-\eta(F')$. By Remark \ref{remark_diagramatics_pasting_lemma}, the trisection surface for $X\natural X'$ is a closed surface of genus $G=g+g'+ 2b-1$. 

To determine the value of each $K_i$ we need to dive into the proof of Lemma \ref{pasting_lemma_star}. Fix $i\in \{1,2,3\}$, the 4-manifold piece $W_i-\eta(F)$ is determined by the tuple $(\Sigma; \Delta^1, \Delta^0, \Delta_{all})$ as in Remark \ref{remark_diagrams_standard_piece}; similarly for $W'_i-\eta(F')$. By definition of bridge trisection, $F\cap \partial W_i$ is a $c_i$ component unlink in bridge position. 
By construction, $\Delta_{all}$ contains separating curves breaking $\Sigma$ as the connected sum of a genus $g$ closed surface and a $c_i$ spheres, each of them with positive even number of boundary components. Each sphere corresponds to a component of the unlink $F\cap \partial W_i$. 
The curves $(\Delta^0, \Delta^1)$ in the genus $g$ surface determine a 1-handlebody with $k_i$ 1-handles. For each planar surface, the curves $(\Delta^0,\Delta^1)$ are only of type 2 which determine 4-dimensional pieces that get glued with other 4-dimensional pieces coming from $W'_i-\eta(F')$. After gluing each planar surface determines a copy of $S^1\times B^3$ (see paragraph 4 of the proof of Lemma \ref{pasting_lemma_star}).
Thus, as explained in paragraph 3 of the proof of Lemma \ref{pasting_lemma_star}, the glued 4-manifold piece $W_i\cup W'_i$ is diffeomorphic to $\natural_{K_i} S^1\times B^3$ where $K_i=k_i+k'_i + c_i$. 
\end{proof}
%%%%%%%%%%%%%%%%%%%%%%%%%%%%%%%%%%%%%%%%%%%%%%%%%%%%%%%%%%
\subsection{Cacime Surface}\label{section_cacime}
Let $F_2$, $F_3$ be oriented surfaces of genus two and three, respectively. Define $\tau_i:F_i\ra F_i$ to be involutions as in Figure \ref{fig_involutions}. Define the Cacime Surface to be the quotient $C=F_2\times F_3/\tau_2\times \tau_3$. 
Following Chapter 4.2 of \cite{Akbulut_notes}, $C$ is diffeomorphic to a fiber sum of $F_2$ bundles over $T^2$ \[C=\left(F_2\times T^2\right) \natural \left( F_2\times S^1\times [0,1]\right)/\wt \tau_2,\]
here $\wt \tau_2(x,t) = (\tau_2(x),t)$ is a diffeomorphism of $F_2\times S^1$.
\begin{figure}[h]
\centering
\includegraphics[scale=.3]{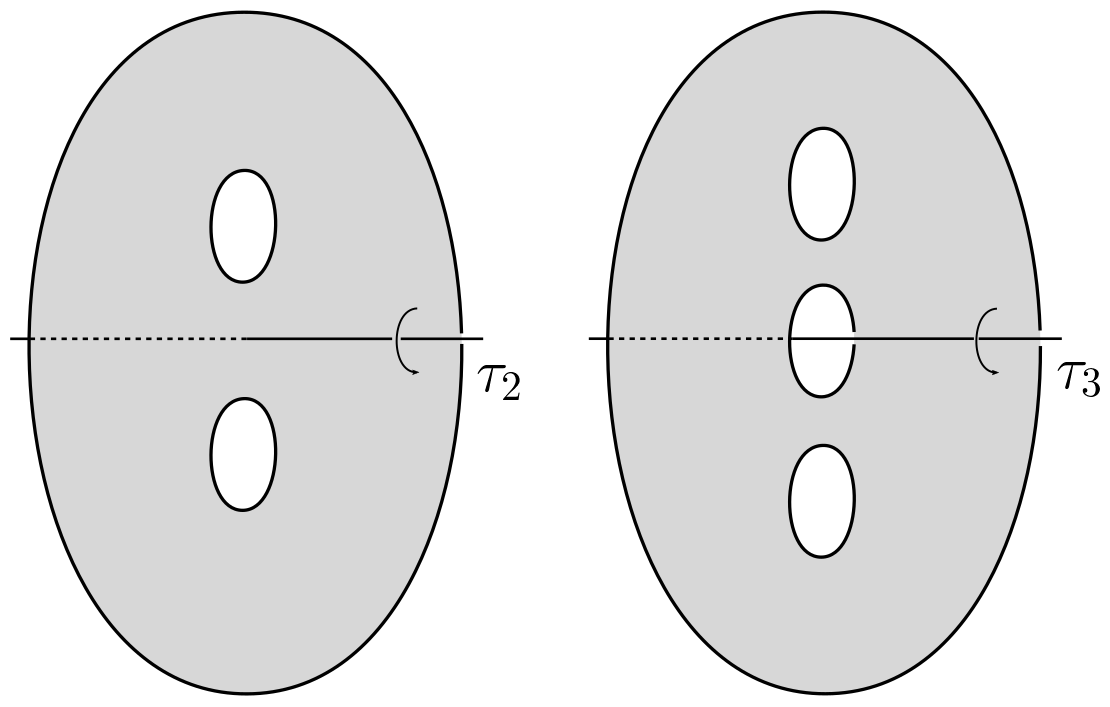}
\caption{The involutions $\tau_2$ an $\tau_3$.}
\label{fig_involutions}
\end{figure}
%In the Figures \ref{Cacime_1_fig}, \ref{Cacime_2_fig} and \ref{Cacime_3_fig} one can see the procedure to obtain a trisection diagram for $C$ as a fiber sum: 
To obtain a trisection diagram for $C$ we follow the following three steps drawn in figures \ref{Cacime_1_fig}, \ref{Cacime_2_fig}, \ref{Cacime_3_fig}, \ref{Cacime_4_fig}, \ref{Cacime_5_fig} and \ref{Cacime_6_fig}
\begin{enumerate}
\item Draw a genus 5 Heegaard splitting for $F_2\times S^1$ and perform Koenig's algorithm \cite{Dale_thesis} to draw a $(21;6,6,11)$-trisection diagram for $X_f=(F_2\times S^1)\times[0,1]/f$ for $f\in\{id,\wt \tau_2\}$.
\item Notice that $F_2\times\{pt\}$ can be seen in the Heegaard splitting for $F_2\times S^1$ and use this to draw a system of shadows of a $(5;1,1,1)$-bridge trisection for $F_2\times\{pt\}\times\{pt\}$.
\item Apply the method in Section \ref{section_surface_complement} to $\star$-trisect the complements $X_f-\eta(F_2)$ and tube the corresponding boundaries using the Pasting Lemma to obtain a $(51;13,13,23)$-trisection diagram for $C$ (See Proposition \ref{prop_fiber_sum}).
\end{enumerate} 
Koenig showed in \cite{Dale_thesis} that his algorithm always results in a trisection which can be destabilized. Here, we can destabilize ten times (five on each $X_f$) and get a $(41;13)$-trisection for the Cacime Surface. This trisection will yield a handle decomposition with one 0-handle, 13 1-handles, 28 2-handles, 13 3-handles and one 4-handles. This picture resembles the handle diagram in Figure 4.17 of \cite{Akbulut_notes}. Thus we think of the Pasting Lemma as the trisection analog of the Roping Method for handle decompositions. 
\begin{question} 
Is there an interpretation for the phrase ``upside-down trisection"? If so, is there a different method of gluing two $\star$-trisected 4-manifolds besides the one in Lemma \ref{pasting_lemma_star}? % which realizes the implied notion of ``doubling" that comes with this interpretation?
\end{question}
\begin{figure}[h!]
\centering
\includegraphics[scale=.55]{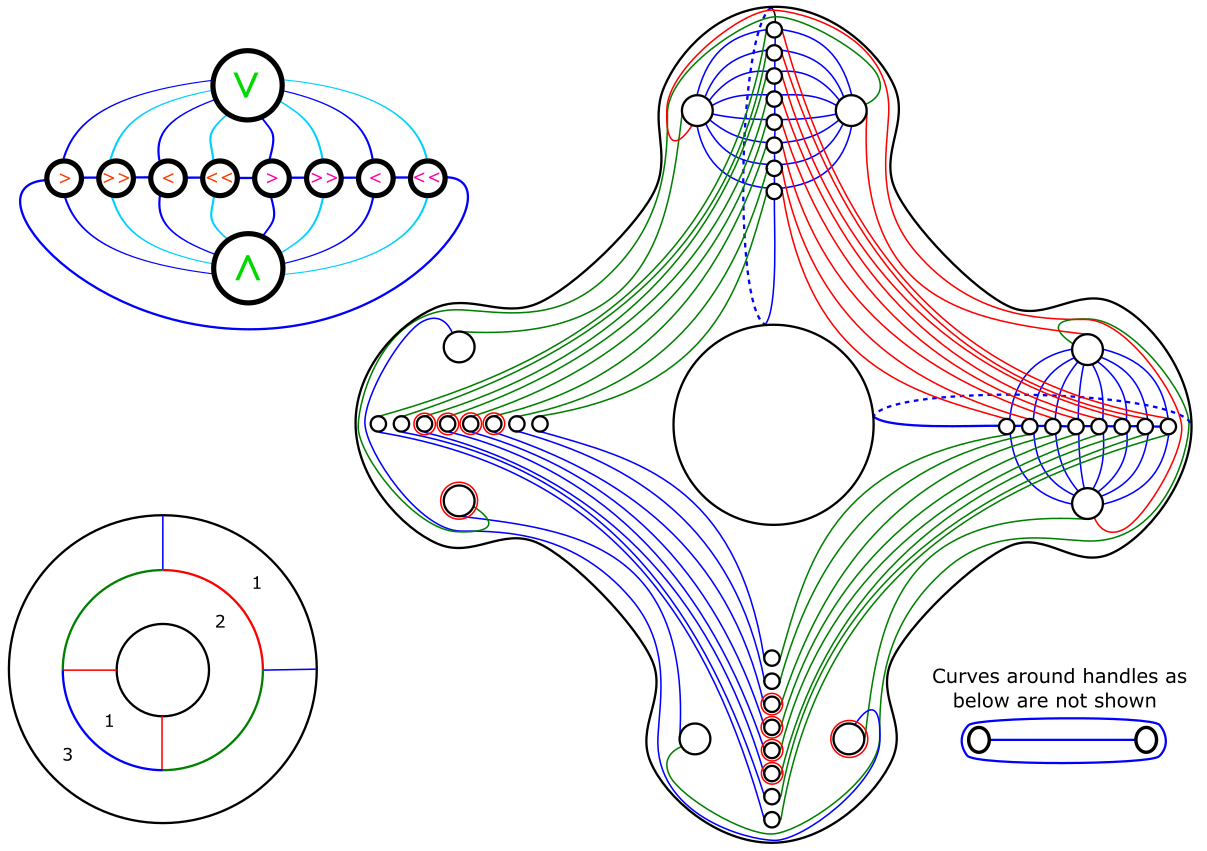}
\caption{Genus five Heegaard splitting for $F_2\times S^1$ and a genus 21 trisection diagram for $X_{id}=F_2\times T^2$. Both diagrams are drawn in punctured surfaces with the correct identifications on the boundaries. The bottom left annulus is a diagrammatic representation of the trisection for $X_{id}$. The colored arcs in the core of the annulus correspond to thickened punctured Heegaard surfaces, and the rest of the arcs are copies of the 3-dimensional handlebodies of the original Heegaard splitting. For more detailes see \cite{Dale_thesis}.}
\label{Cacime_1_fig}
\end{figure}

\begin{figure}[h!]
\centering
\includegraphics[scale=.55]{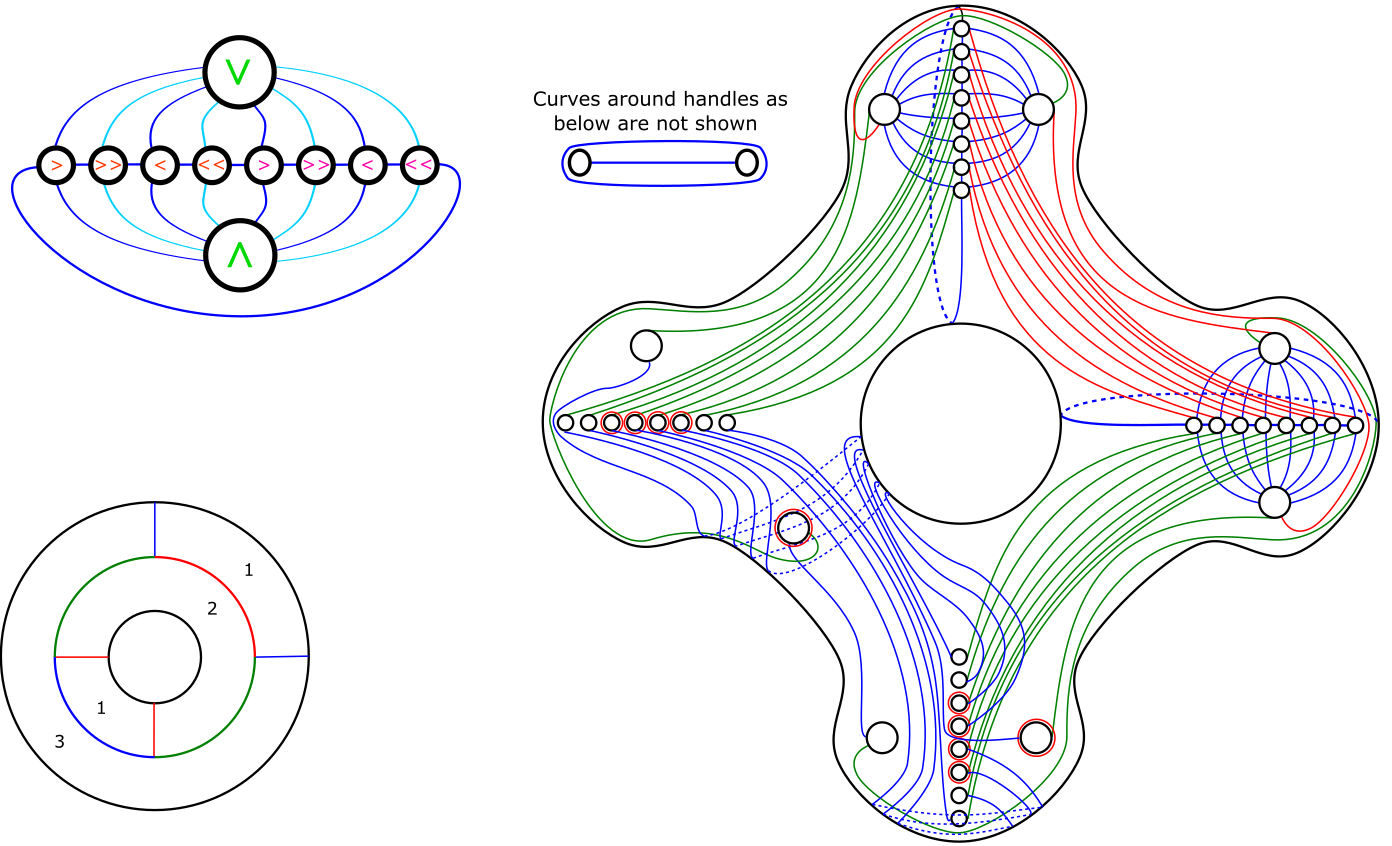}
\caption{A genus 21 trisection diagram for $X_{\tau_2}$. Notice the twist on some of the blue arcs is decoding the action of $\tau_2$ in the Heegaard splitting of $F_2\times S^1$.}
\label{Cacime_2_fig}
\end{figure}

\begin{figure}[h!]
\centering
\includegraphics[scale=.05]{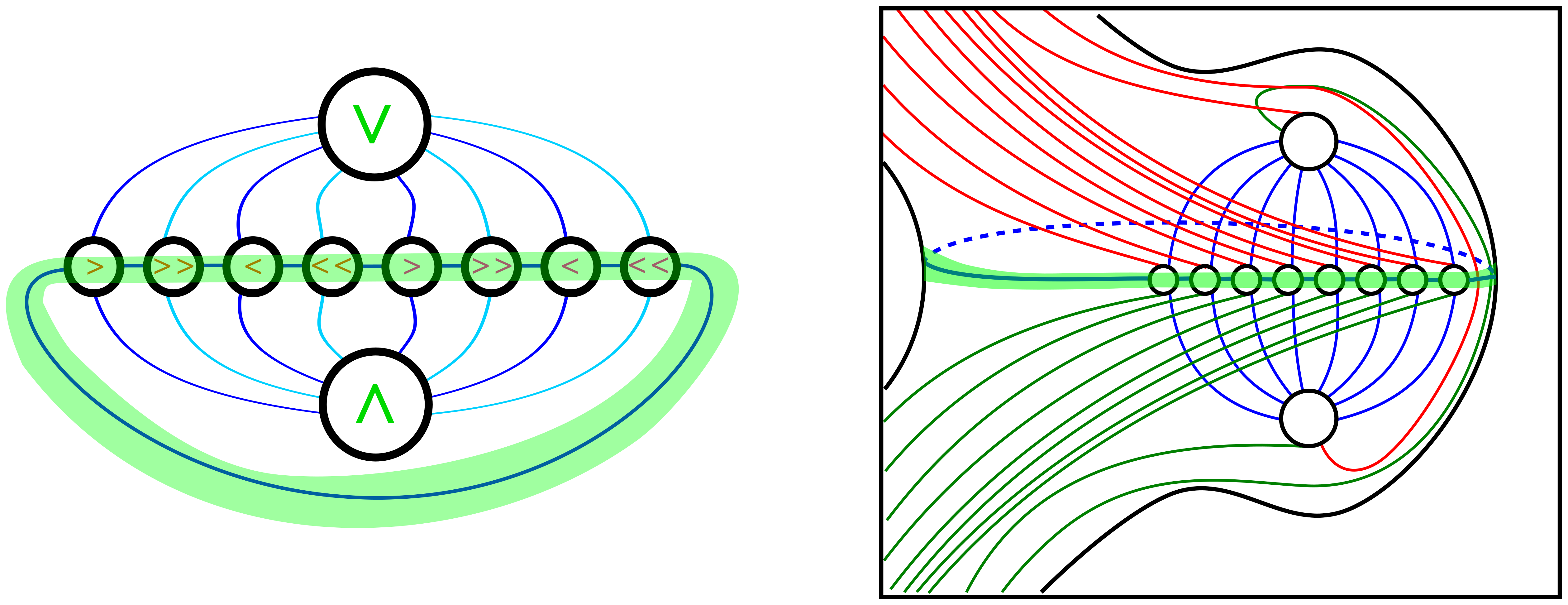}
\caption{The green shaded area corresponds to a copy of $F_2\times\{pt\}$ inside the Heegaard splitting for $F_2\times S^1$ (left) and $F_2\times \{pt\}\times\{pt\}$ in $X_f$.}
\label{Cacime_3_fig}
\end{figure}

\begin{figure}[h!]
\centering
\includegraphics[scale=.06]{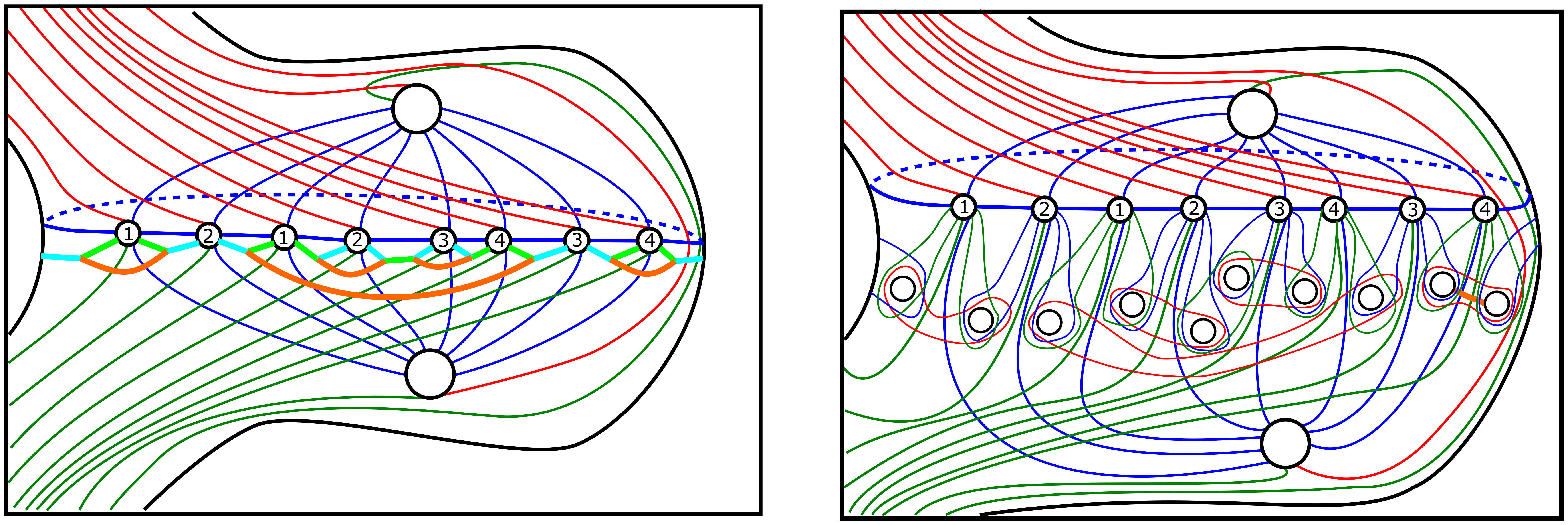}
\caption{Local models for the bridge trisection of $F_2\times\{pt\}\times\{pt\}$ in $X_f$ (left), and the $\star$-trisection diagram for its complement (right).}
\label{Cacime_4_fig}
\end{figure}

\begin{figure}[h!]
\centering
\includegraphics[scale=.05]{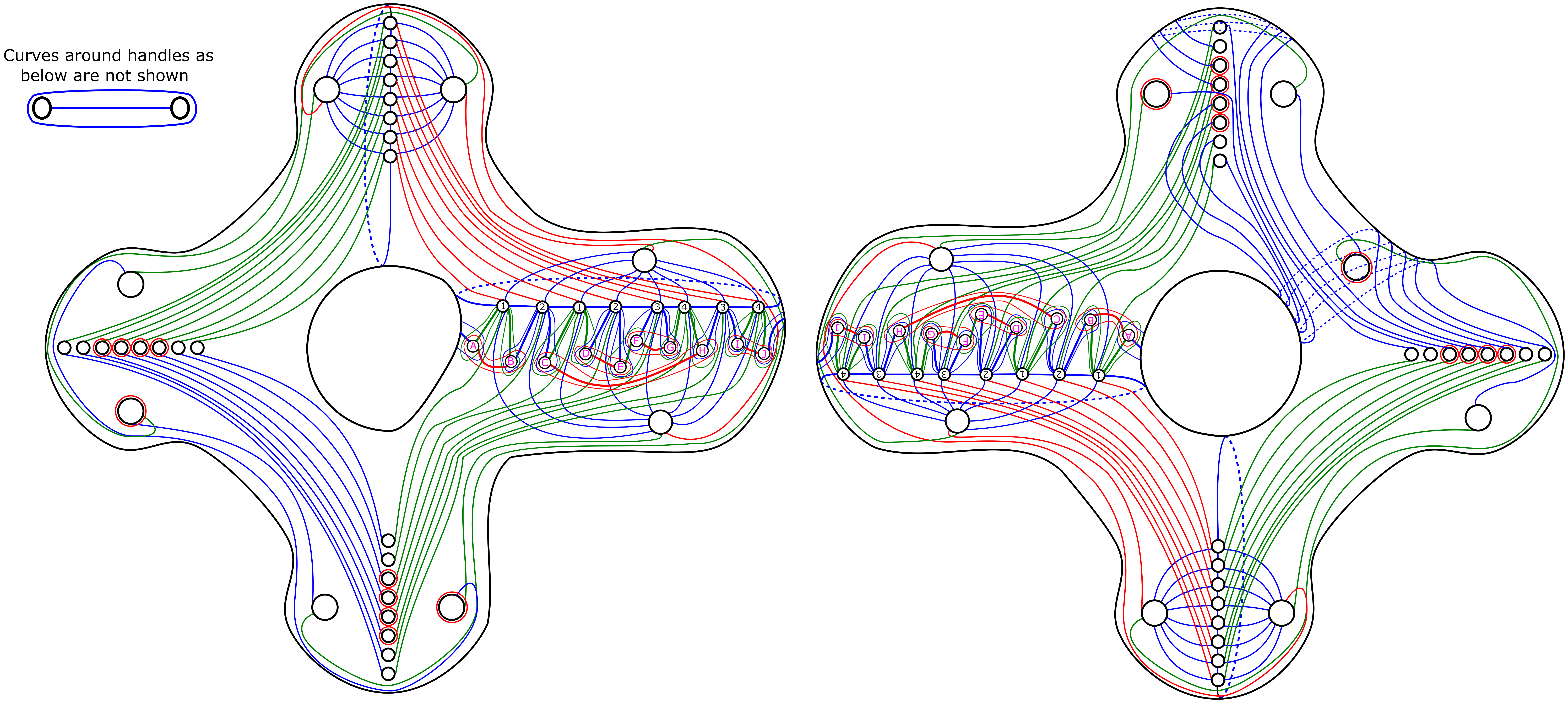}
\caption{Genus 51 trisection diagram of the Cacime surface.}
\label{Cacime_5_fig}
\end{figure}

\begin{figure}[h!]
\centering
\includegraphics[scale=.2]{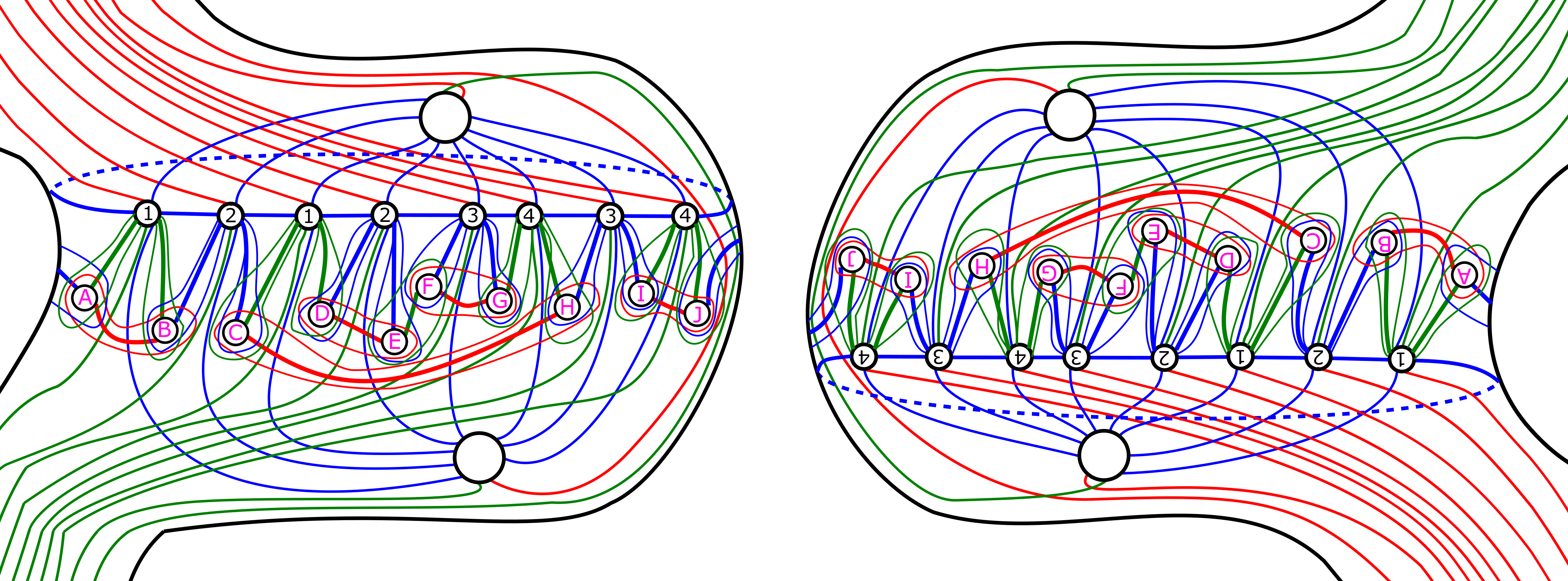}
\caption{A close-up of the above trisection. The disks with common labels are identified as shown. Notice that some curves can be erased as in Figure \ref{fig_knot_surgery_5}.}
\label{Cacime_6_fig}
\end{figure}

%%%%%%%%%%%%%%%%%%%%%%%%%%%%%%%%%%%%%%%%%%%%%%%%%%%%%%%%%%
\subsection{Knot Surgery} \label{section_knot_surgery}
Let $K$ be a knot in $S^3$, and let $m$ denote a meridian of $K$. Let $M_K$ be the 3-manifold obtained by 0-surgery along $K$. Notice that $m$ can be viewed as a circle in $M_K$ and that the torus $T_m=m\times S^1\subset M_K\times S^1$ has self-intersection zero. Let $X$ be a 4-manifold containing an embedded torus $T$ with self-intersection zero. Denote by $X_K$ the fiber sum 
\[ X_K = X\natural_{T=T_m} (M_K\times S^1).\] 
Here, we glue the complement of the corresponding thickened tori along a diffeomorphism preserving $\{pt\}\times \partial D^2$. 
Fintushel and Stern introduced the knot surgery operation in \cite{FS_knot_surgery} to build exotic copies of smooth 4-manifolds by controlling the change of the Seiberg-Witten invariants using the Alexander polynomial of $K$. 

The goal of this subsection is to describe how to draw trisection diagrams for $X_K$. Figures \ref{fig_knot_surgery_1}, \ref{fig_knot_surgery_2}, \ref{fig_knot_surgery_3}, \ref{fig_knot_surgery_4} and \ref{fig_knot_surgery_5} show the steps taking $K$ to be the trefoil knot.

Let $K$ be a knot in $S^3$. Find a Heegaard splitting for $S^3$ such that $K$ can be isotoped to be inside one of the handlebodies intersecting one meridian disk exactly in one point. 
In order to do this one can consider a tunnel system for $K$ as in Figure \ref{fig_knot_surgery_1}. 
Project $K$ onto the Heegaard surface $F$ in such a way that is an embedded circle in $F$ and the framing induced by the surface is the 0-framing on $K$. 
By construction, we can find a Heegaard diagram $(F;a,b)$ such that an isotopic copy of $m$ belongs to $a$ and $K$ is disjoint from all other elements of $a$. 
The pink circle in Figure \ref{fig_knot_surgery_1} are possible choices for the meridian $m$ as subsets of $F$.
A Heegaard diagram for $M_K$ is given by $(F;a',b)$ where $a'=(a-m)\cup K$. Furthermore, the loop $m$ as a subset of $F$ corresponds to the meridian of $K$ inside $M_K$ as in Figure \ref{fig_knot_surgery_1}.

Now perform Koenig's algorithm \cite{Dale_thesis} to draw a trisection diagram $(\Sigma;\alpha,\beta,\gamma)$ for $M_K\times S^1$ using $(F;a',b)$. Our choice of $m$ as a subset of $F$ allows us to see a bridge position for $T_m=m\times S^1$. To see this recall that $\Sigma$ is obtained by four copies of $F$ tubed as in Figure \ref{fig_knot_surgery_2}. Draw $m$ on each copy of $\Sigma$ and pick four distinct points on each circle. Then push-off $T_m$ away from $F\times S^1$ fixing the 16 selected points. This procedure gives us the bridge trisection of $T_m$ with $8$ bridges as in Figure \ref{fig_knot_surgery_2}. 

Now let $T$ be a torus with self-intersection zero embedded in a 4-manifold $X$. Suppose $T$ is in bridge position with respect to some trisection of $X$. 
There are two approaches we can take in order to draw a trisection diagram for $X_K$. 

\textbf{The first approach to trisect $X_K$} is to perturb\footnote{See \cite{bridge_trisections_4M}.} both bridge trisections for $T_m$ and $T$ until the new bridge trisections induce the same cell decomposition on both $T_m$ and $T$. Then, to draw a trisection for $X_K$ we have to draw the $\star$-trisection diagrams for the corresponding surface complements following Section \ref{section_surface_complement} and tube them using the Pasting Lemma as we did for the Cacime Surface. 

\textbf{The second approach to trisect $X_K$} is to glue a copy of $T^3\times [0,1]$ in such a way that the new boundary has a nice $S^1$-fibration with fiber a copy of the surface $T_m\times \{pt\}$ (similarly for $T$). In order to do this, draw the cell decomposition induced by the bridge trisection on the torus $T_m$ (see Figure \ref{fig_knot_surgery_3}). This picture can be thought as a bridge trisection for $T_m\times \{0\}$ inside $T_m\times D^2$. 
Thus we can draw a $\star$-trisection diagram for $T_m\times S^1\times[0,1]$ with one boundary having the same handle decomposition as the $\star$-trisection in Figure \ref{fig_knot_surgery_2} and other boundary a $S^1$-fibration with fiber $T_m\times\{pt\}$. This new trisection is drawn in Figure \ref{fig_knot_surgery_3}. 
Now tube this new $\star$-trisection with the trisection for the complement of $T_m$ in $M_K\times S^1$ to obtain a classical relative trisection (with empty binding) with a copy of $T_m$ as the fiber on its boundary, as desired. Notice the appearance of sphere components in the compressed surfaces $\Sigma_{\alpha}$, $\Sigma_{\beta}$ and $\Sigma_{\gamma}$, thus some curves are redundant (see Figure \ref{fig_knot_surgery_4}). The final trisection diagram is depicted in Figure \ref{fig_knot_surgery_5}. After performing a similar process to the bridge trisection of $T$ in $X$ one, in theory, can perform Pasting Lemma one last time to draw a trisection for $X_K$. 

\textbf{The advantage of the second method} is that any diffeomorphism of the form $f\times id_{S^1}:T_m\times \partial D^2\ra T\times \partial D^2$ can be used to perform the fiber sum.

\begin{figure}[h]
\centering
\includegraphics[scale=.1]{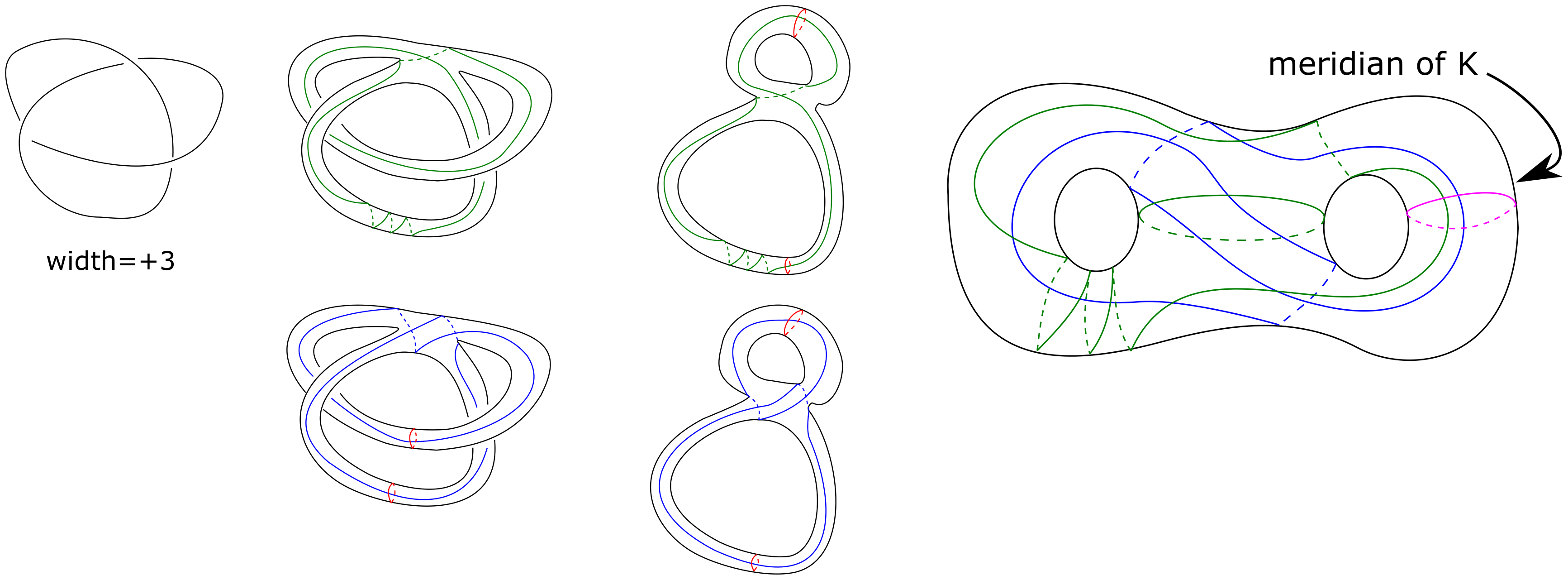}
\caption{Heegaard diagram for $M_K$ with $K$ a trefoil knot.}
\label{fig_knot_surgery_1}
\end{figure}
\begin{figure}[h]
\centering
\includegraphics[scale=.04]{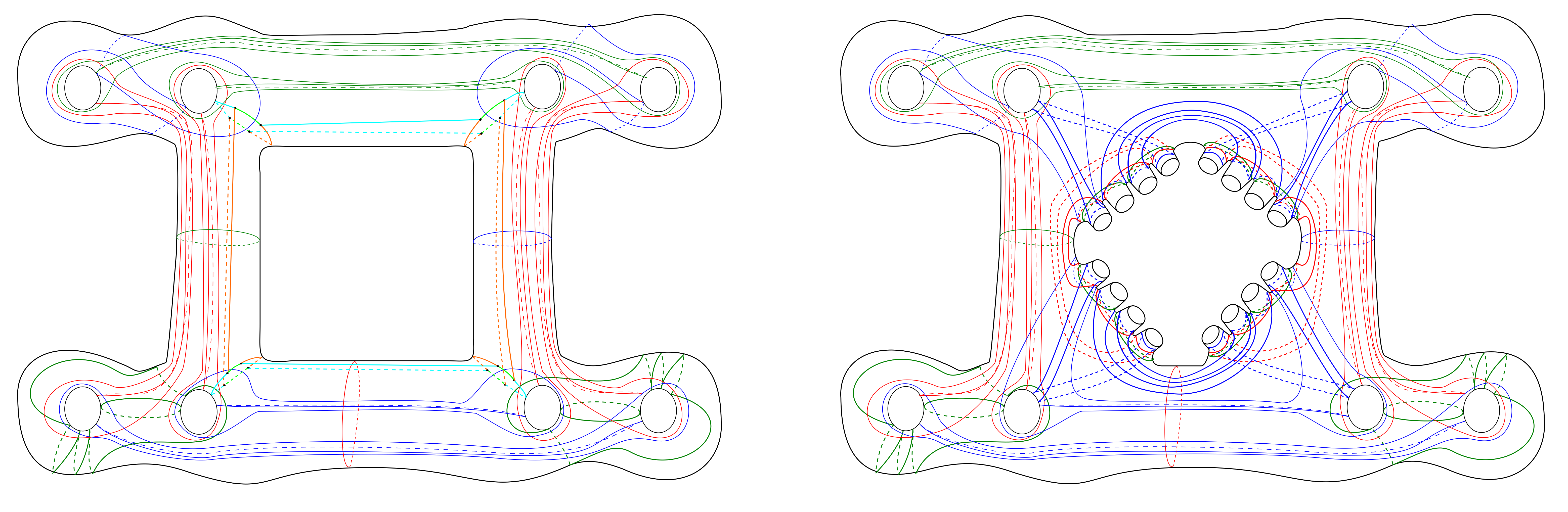}
\caption{Bridge trisection diagram for $T_m\subset M_K\times S^1$. After an isotopy, we draw a $\star$-trisection diagram for the complement $(M_K\times S^1)-(T_m\times D^2)$.}
\label{fig_knot_surgery_2}
\end{figure}
\begin{figure}[h]
\centering
\includegraphics[scale=.075]{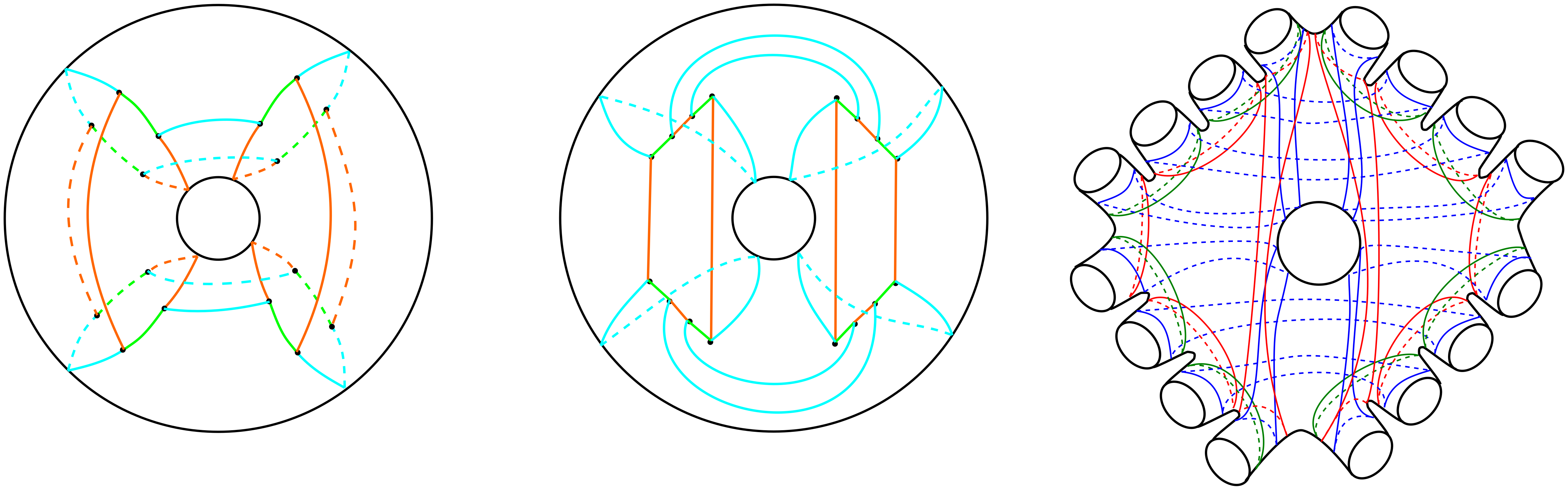}
\caption{(left and middle) The cell decomposition induced by the bridge trisection on $T_m$, notice that also describes a bridge trisection for $T_m\times\{0\}\subset T_m\times D^2$. (right) The associated $\star$-trisection for the complement of this bridge trisected surface.}
\label{fig_knot_surgery_3}
\end{figure}
\begin{figure}[h]
\centering
\includegraphics[scale=.04]{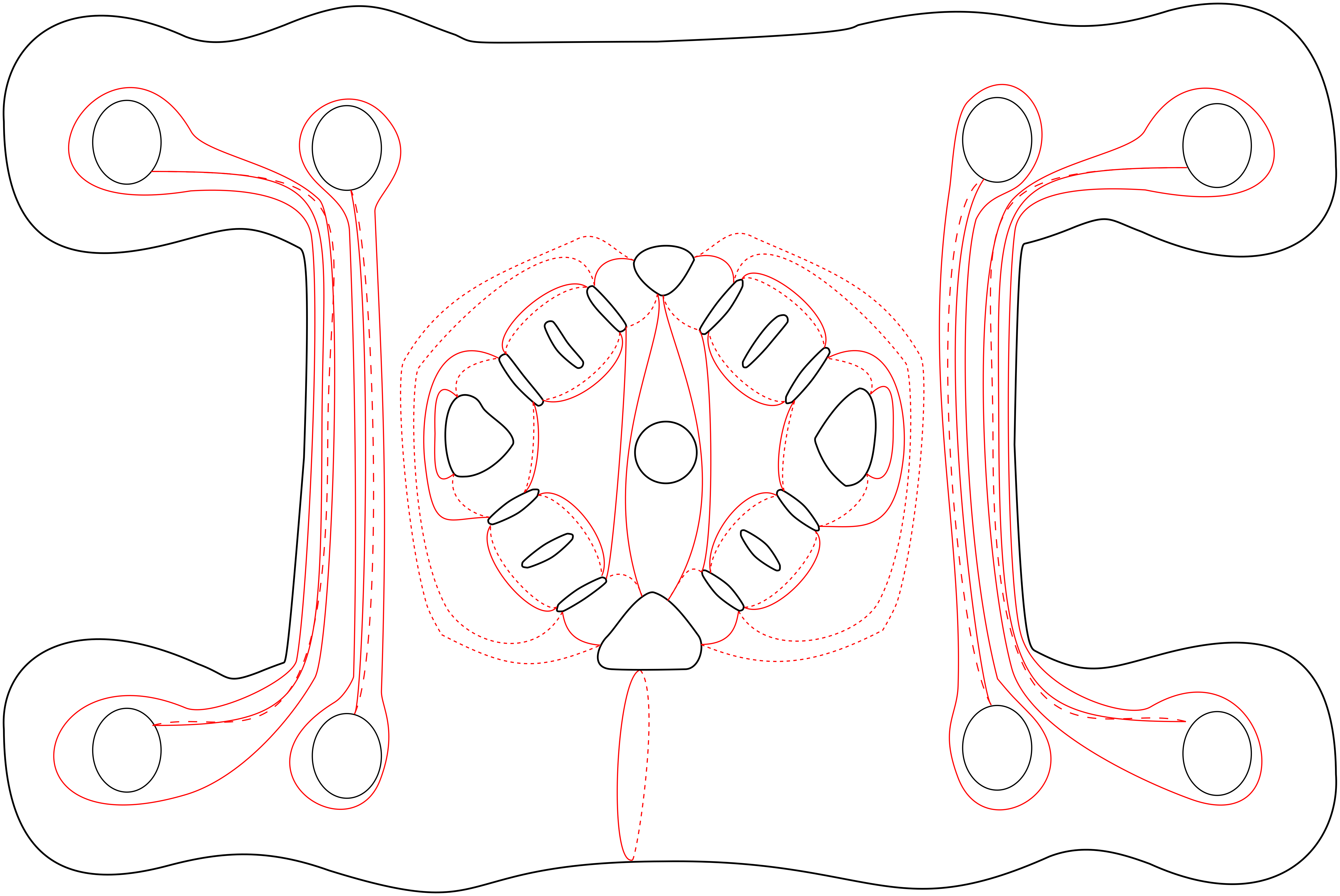}
\caption{Red loops after performing the pasting lemma. Notice that they are some redundancies.}
\label{fig_knot_surgery_4}
\end{figure}
\begin{figure}[h]
\centering
\includegraphics[scale=.075]{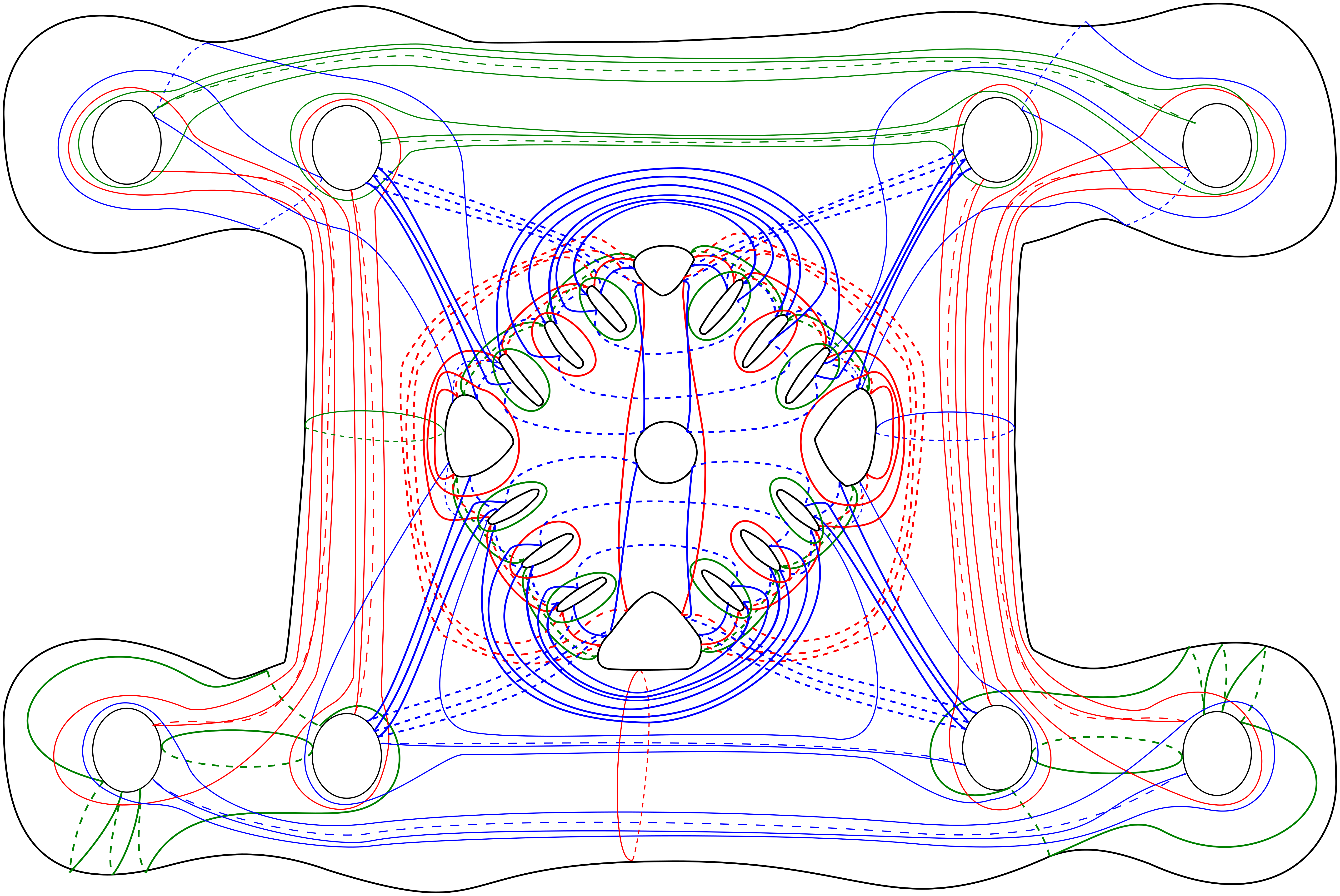}
\caption{A trisection diagram for the complement of $T_m\times S^1$ inside $M_K\times S^1$ with boundary admitting a $S^1$-foliation with fiber $T_m\times\{pt\}$. To perform knot surgery, we must attach this diagram to the complement of a torus in $X$.}
\label{fig_knot_surgery_5}
\end{figure}

%%%%%%%%%%%%%%%%%%%%%%%%%%%%%%%%%%%%%%%%%%%%%%%%%%%%%%%%%%

\subsection{Torus Surgery} \label{section_torus_surgery}
Let $F$ be an embedded torus with trivial tubular neighborhood in a $\star$-trisected 4-manifold $X$. The goal of this subsection is to describe how to draw trisection diagrams for all the 4-manifolds $X_{F,g}:=\left(X-\eta(F)\right)\cup_g (T^2\times D^2)$ where $g:\partial \eta(F) \ra T^3$ is an isotopy class of homeomorphism of $T^3$. In principle, there are $SL_3(\mathbb{Z})$ many such maps. 

Suppose $F$ is in $(b;c_1,c_2,c_3)$-bridge position with respect to the given trisection of $X$. 
Begin by performing the construction of a $\star$-trisection for the complement of $F$ in $X$ as in Section \ref{section_surface_complement}. 
The handle decomposition on the boundary of the new $\star$-trisection for $X-\eta(F)$ is rather complicated. 
To overcome this problem, we can paste our diagram for $X-\eta(F)$ with a special diagram for $T^3\times [0,1]$. 
The bridge trisection on $F$ induces a cell decomposition on the 2-dimensional torus. Given any such cell decomposition, we can draw a bridge position for $T^2\times\{0\}$ in $T^2\times D^2$ inducing the same cell decomposition. We can then draw a $\star$-trisection diagram for $T^3\times[0,1]$ with one boundary having the same handle decomposition as the trisection for $X-\eta(F)$. We will refer to such $\star$-trisections by $\tau_0$. 
Figures \ref{fig_surface_complement} and \ref{fig_knot_surgery_3} show examples of the $\tau_0$ $\star$-trisections for specific cell decompositions of the 2-torus. 

Gluing $\tau_0$ to the $\star$-trisection of $X-\eta(F)$ maintains the diffeomorphism type of the complement fixed, and replaces the restrictive handle decomposition in the boundary with the $S^1$-foliation $F\times S^1$ with fiber an isotopic copy of $F\times\{pt\}$. We can now apply Pasting Lemma to this new trisection of $X-\eta(F)$ using diffeomorphisms of the form $g=f\times id_{S^1}$ for some homeomorphism $f:F\ra F$. 
We think of this kind of gluing $g$ as acting on the standard basis for $H_1(T^3,\mathbb{Z})$ via the matrix $\left( \begin{smallmatrix}a&b&0\\c&d&0\\0&0&1\end{smallmatrix}\right)$.

Denote by $\tau_{23}$ the relative trisection diagram of Example \ref{examp_cob}. This diagram yields a relative trisection for $T^3\times [0,1]$ with closed trisection surface such that the $S^1$-foliation induced in $\partial (T^3\times[0,1])$ has page $S^1\times S^1\times\{pt\}$ in $T^3\times\{0\}$ and page $S^1\times \{pt\}\times S^1$ in $T^3\times\{1\}$. We think of $\tau_{23}$ as acting on the standard basis for $H_1(T^3,\mathbb{Z})$ via a permutation matrix $\sigma_{23}=\left( \begin{smallmatrix}1&0&0\\0&0&-1\\0&1&0\end{smallmatrix}\right)$. By modifying the labelings in $\tau_{23}$, we can draw trisections diagrams $\tau_{31}$ and $\tau_{12}$ of $T^3\times [0,1]$ corresponding to 2-cycles $(3,1)$ and $(1,2)$ respectively. 

Any matrix in $SL_3(\mathbb{Z})$ can be written as a product of permutation matrices ($\tau_{ij}$) and matrices of the form $\left( \begin{smallmatrix}1&1&0\\0&1&0\\0&0&1\end{smallmatrix}\right)$ (special kind $f\times id_{S^1}$). 
Hece, equipped with the trisections $\tau_0$ and $\tau_{ij}$ we can, in theory, draw $\star$-trisection diagrams for all the 4-manifolds $X_{F,g}$. 
The procedure of performing torus surgery becomes then a linear algebra problem. We will explain this explicitly by showing methods to trisect Logarithmic transforms and Luttinger Transforms. 

%%%%%%%%%%%%%%%%%%%%%%%%%%
\begin{examp}[Logarithmic Transform]
Let $F$ be an embedded torus in $X$ with self-intersection number zero. Fix a basis for $H_1(F,\mathbb{Z})$ and a trivialization $\eta(F)\cong T^2\times D^2$. Following \cite{round_handles}, given a matrix $A=\big(\begin{smallmatrix}a_{11}&a_{12}\\a_{21}&a_{22}\end{smallmatrix}\big)\in SL_2(Z)$, we will denote as $A$-Logarithmic transform the 4-manifold $X_{F,A}=(X-\eta(F))\cup_{g}(T^2\times D^2)$, where $g:T^2\times D^2\ra T^2\times D^2$ is a homeomorphism given by the matrix $\Big(\begin{smallmatrix}1&0&0\\0&a_{11}&a_{12}\\0&a_{21}&a_{22}\end{smallmatrix}\Big)$.
Suppose that $F$ is in bridge position with respect to some trisection $\tau$ of $X$. Notice that 
\[ A =\left(\begin{matrix}0&0&1\\0&1&0\\1&0&0\end{matrix}\right)\cdot \left(\begin{matrix}a_{22}&a_{21}&0\\a_{12}&a_{11}&0\\0&0&1\end{matrix}\right) \cdot \left(\begin{matrix}0&0&1\\0&1&0\\1&0&0\end{matrix}\right).\]
Thus to perform $A$-logarithmic transform along $F$ we can take the $\star$-trisection $\wt \tau$ for the complement of $F$ in $X$, and concatenate it with specific trisections as follows,
\[X_{F,A}\quad:\quad \wt \tau \bigcup_{id} \tau_0 \bigcup_{id} \tau_{31} \bigcup_{A\times id_{S^1}} \tau_{31} \bigcup_{id} \tau_{\emptyset}.\] 
Here $\tau_{\emptyset}$ is the relative trisection for $T^2\times D^2$ given by the empty diagram on a closed torus.

Integral logarithmic transform is given by the matrix $A_p=\big(\begin{smallmatrix}0&1\\-1&p\end{smallmatrix}\big)$. In particular $0$-logarithmic transform is given by the map gluing map corresponding with the permutation $(2,3)$ and so the trisection diagram for the $0$-logarithmic transform $X_{F,0}$ can be simplified as follows: 
\[X_{F,0}\quad:\quad \wt \tau \bigcup_{id} \tau_0 \bigcup_{id} \tau_{23} \bigcup_{id} \tau_{\emptyset}.\] 
\end{examp}
%%%%%%%%%%%%%%%%%%%%%%%%%%%

\begin{examp}[Luttinger Surgery]
For an embedded torus $F$ in $X$ with self-intersection zero, a Luttinger surgery is an operation $X \mapsto X_{m,n}$ where $X_{m,n}$ is torus surgery along $F$ via a homeomorphism given by the matrix $A_{m,n}=\left(\begin{smallmatrix}1&0&m\\0&1&n\\0&0&1\end{smallmatrix}\right)$. Since $A_{m,n}$ factors as follows, 
\[ A =\left(\begin{matrix}1&0&0\\0&0&1\\0&1&0\end{matrix}\right)\cdot \left(\begin{matrix}1&m&0\\0&1&0\\0&0&1\end{matrix}\right) \cdot \left(\begin{matrix}0&0&1\\0&1&0\\1&0&0\end{matrix}\right)\cdot \left(\begin{matrix}1&n&0\\0&1&0\\0&0&1\end{matrix}\right)\cdot \left(\begin{matrix}0&0&1\\0&1&0\\1&0&0\end{matrix}\right) \cdot \left(\begin{matrix}1&0&0\\0&0&1\\0&1&0\end{matrix}\right).\]
We can preform Luttinger surgery along $F$, after fixing a trivialization $\eta(F)\cong T^2\times D^2$ and a basis for $H_1(F,\mathbb{Z})$, by concatenating the following trisections. 
\[ X_{F,m,n} \quad : \quad \wt \tau \bigcup_{id} \tau_0 \bigcup_{id} \tau_{23} \bigcup_{ \left(\begin{smallmatrix}1&m\\0&1\end{smallmatrix}\right) \times id} \tau_{13} \bigcup_{\left(\begin{smallmatrix}1&n\\0&1\end{smallmatrix}\right) \times id}\tau_{13} \bigcup_{id} \tau_{23} \bigcup_{id} \tau_{\emptyset}.\]
Here, $\wt \tau$ is the $\star$-trisection for the complement of $F$ and $\tau_{\emptyset}$ is the empty trisection diagram for $T^2\times D^2$. 
\end{examp}

%%%%%%%%%%%%%%%%%%%%%%%%%%%%%%%%%%%%%%%%%%%%%%%%%%%%%%%%%%
\appendix
\section{Classic Diagrams}
\label{section_classic_diagrams}
Let $X$ be a $\star$-trisected 4-manifold with non-empty boundary. If the trisection is a classic relative trisection, the compression bodies given by the pairwise intersection satisfy  $C_i^0=C_i^1$ for all $i=1,2,3$. In particular, any diagram $(\Sigma; \alpha, \beta, \gamma)$ of such trisections satisfies that each pair of loops is slide equivalent to the loops like in Figure \ref{standard_rel_picture}.
In this case, there is a open book decomposition on $\partial X$ induced by the trisection with binding having $b=|\partial \Sigma|$ components. In the existing literature algorithms have only been developed when the trisection surface has boundary ($b>0$). The goal of this appendix is to extend these results to the case $b=0$. 
\begin{figure}[h]
\centering
\includegraphics[scale=.1]{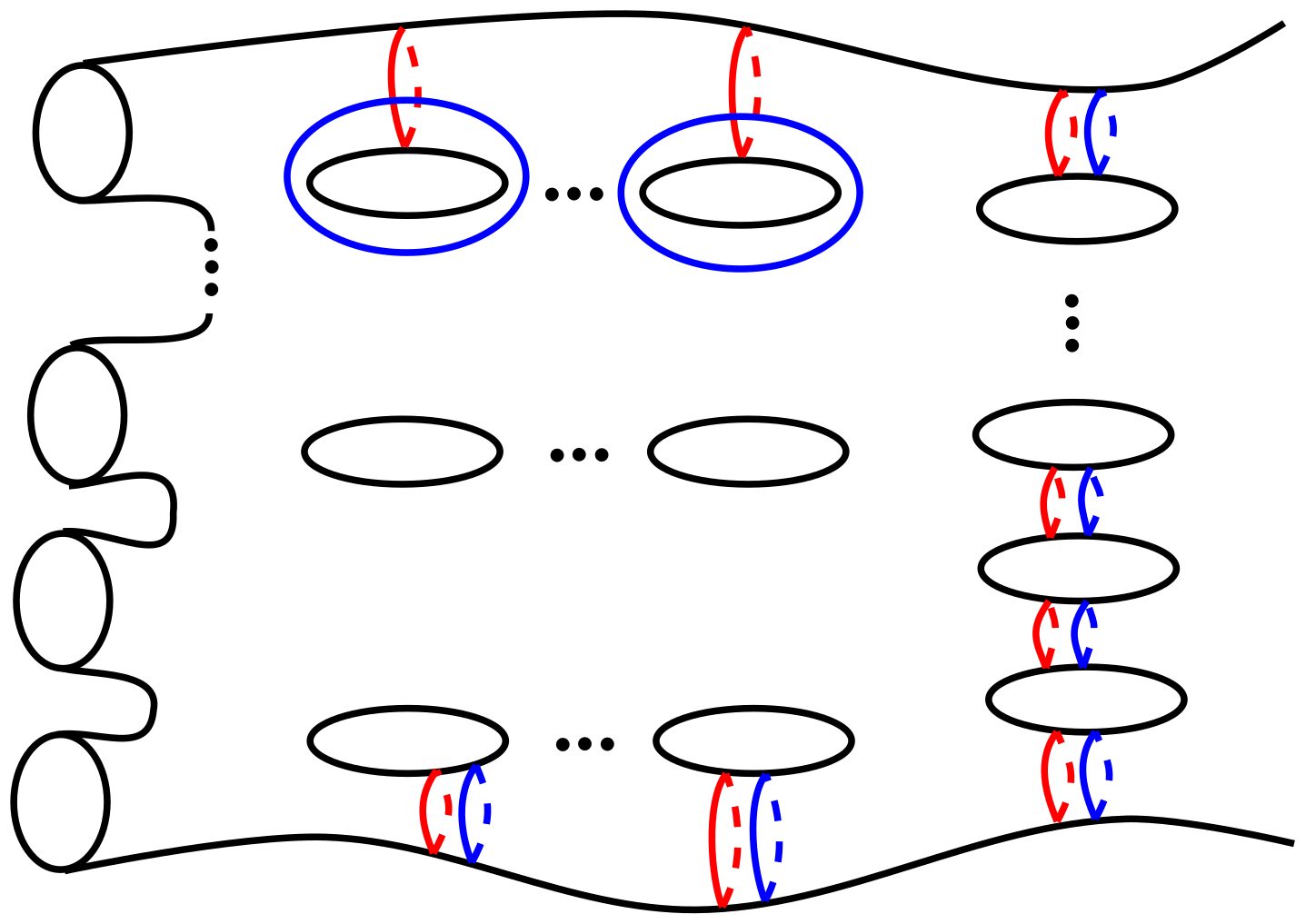}
\caption{The standard picture for ``classic" relative trisection diagrams.}
\label{standard_rel_picture}
\end{figure}

%%%%%%%%%%%%%%%%%%%%%%%%%%%%%%%%%%%%%%%%%%%%%%%%%%%%%%%%%
\subsection{Relative Trisections from Kirby Diagrams} \label{section_rel_trisections_from_Kirby}

Let $X$ be a connected 4-manifold with connected boundary. In \cite{relative_trisections_2} the authors showed how to draw a trisection diagram from a Kirby diagram of $X$ if a page $P$ for an open book decomposition of $\partial X$ is given in the Kirby diagram. As expected, the proper modification of this result holds if $P$ is the page of a fibration of $\partial X$ over $S^1$; i.e., if $\partial P=\emptyset$.  
We will state the result in full generality. 

\begin{thm} [Adaptation from Main Theorem of \cite{relative_trisections_2}]
Take a handle decomposition of $X$ with one 0-handle, some 1-handles, 2-handles and 3-handles described explicitly in the form of a Kirby diagram. Let $P$ be the page of an open book decomposition \textbf{or} a fibration over $S^1$ of $\partial X$. 
Suppose that $P$ is explicitly drawn in the Kirby diagram. Then there is an algorithm to draw a trisection diagram for $X$ described as follows: 
\begin{enumerate} 

\item Isotope $P$ in the diagram so that $P$ has a 2-dimensional handle decomposition induced by the 0-handle and some 1-handles \textbf{and 2-handles} of $X$. You might need to add 1/2-cancelling pairs to do so. 
\item If not all the 1-handles of $X$ were used to build $P$, add genus to $P$ by tubing it as in Figure \ref{tubing}. Call this new surface $\wt{\Sigma}$. 
\begin{figure}[h]
\centering
\includegraphics[scale=.05]{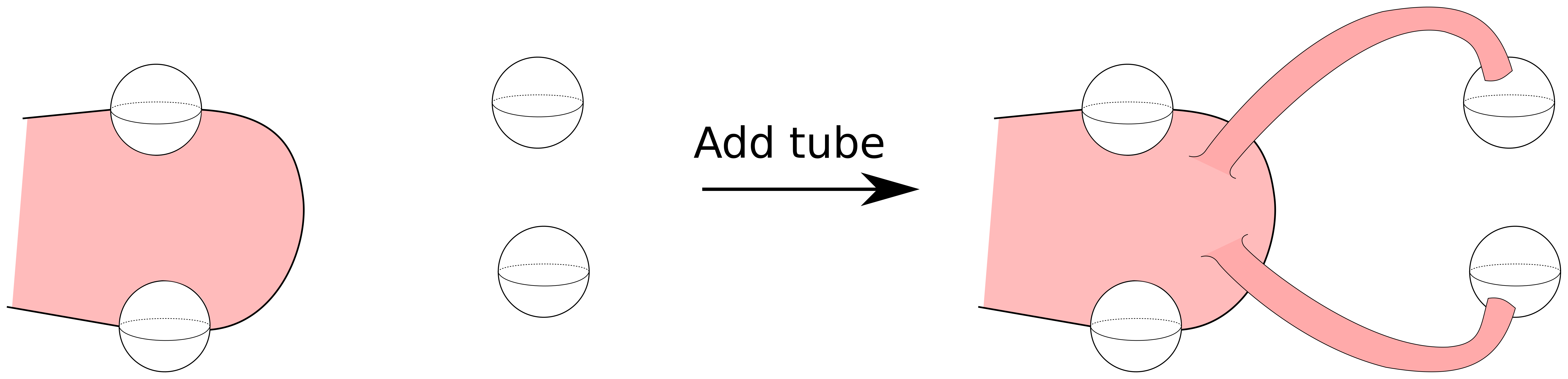}
\caption{Tubing the page.}
\label{tubing}
\end{figure}
\item Project the attaching regions of the rest of the 2-handles onto $\wt\Sigma$. With the help of Reidemeister I moves ensure that the framing of the handles is given by the surface; and use Reidemeister II moves to ensure that every loop has at least one overcrossing in the link projection. 
\item Stabilize $\wt \Sigma$ so that the link above has no crossing following Figure \ref{fix_crossing}. Call this new surface $\Sigma$. Let $\gamma$ be the loops in $\Sigma$ arising from the link projection, let $\alpha$ and $\beta$ be the red and blue curves in $\Sigma$ coming from the stabilizations.
\begin{figure}[h]
\centering
\includegraphics[scale=.05]{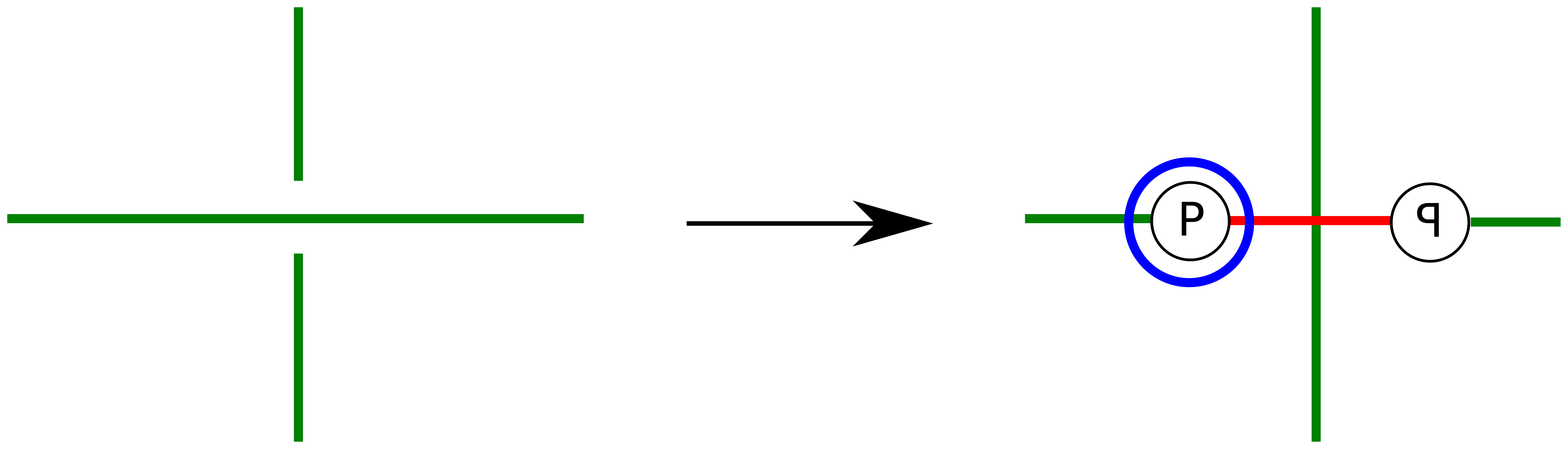}
\caption{How to fix a crossing.}
\label{fix_crossing}
\end{figure}
\item By construction $|\alpha|=|\beta|\geq |\gamma|$. If the inequality is strict, we do the following: 
For each component $\gamma_i$, by construction we can pick a loop $\beta_{J_i}$ intersecting $\gamma_i$ transversely in one point and disjoint from other $\gamma$ curves. Take a $\beta_j$ not in the selected set $\{ \beta_{J_i}\}_i$; $\beta_j$ intersects a unique $\gamma$ curve in one point, say $\gamma_{i_0}$. Slide $\beta_j$ over $\beta_{J_{i_0}}$ using an arc of $\gamma_{i_0}$; denote the resulting curve by $\gamma_j$. 
\end{enumerate} 
The tuple $(\Sigma;\alpha, \beta, \gamma)$ is a relative trisection diagram for $X$ inducing the given fibration on the boundary. 
\end{thm} 

\begin{proof} 
The decomposition of $X$ will be given as follows: 
Divide the 2-handles of $X$ by $h^2=h^2_P\cup h^2_r$ where $h^2_P$ are the ones used to build $P$ and $h^2_r$ the rest of the 2-handles. Define $X_1=B^4[h^1\cup h^2_P]$. We can see $\Sigma$ as embedded in $\partial X_1$ by stabilizing the standard circular decomposition in $\partial X_1$ as in Subsection \ref{subsection_boundary_Z}, say 
\[\partial X_1= \bigslant{\Big( C_\alpha \cup_{\Sigma} C_\beta\Big)}{\Big(\partial_-C_\alpha =_{id} \partial_-C_\beta\Big)}.\]
Define $X_2 = \eta_X(C_\beta)[h^2_r]$ and $X_3 = X-int(X_1\cup X_2)$. 
The proof that $X=X_1\cup X_2\cup X_3$ is indeed a relative trisection is the same as in Theorem 1 of \cite{relative_trisections_2}. 
\end{proof} 

\begin{examp}
Figures \ref{Kirby_part_1}, \ref{Kirby_part_2} and \ref{Kirby_part_3} describe how to draw a relative trisection diagram for the complement of an unknotted torus in $S^4$.
\end{examp} 
\begin{figure}[h]
\centering
\includegraphics[scale=.045]{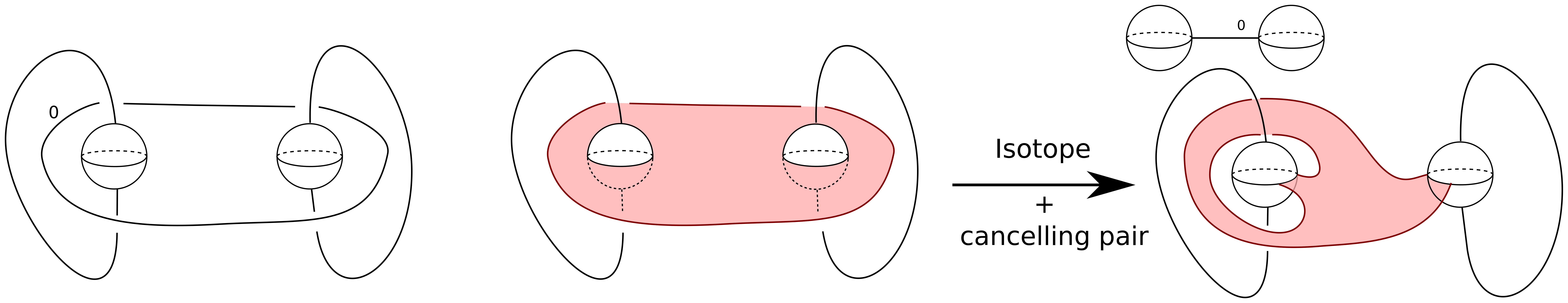}
\caption{Left: A Kirby diagram for the complement of the unknotted torus in $S^4$. Right: The shaded surface is an embedding of a torus page for the fibration of $T^3$.}
\label{Kirby_part_1}
\end{figure}
\begin{figure}[h]
\centering
\includegraphics[scale=.05]{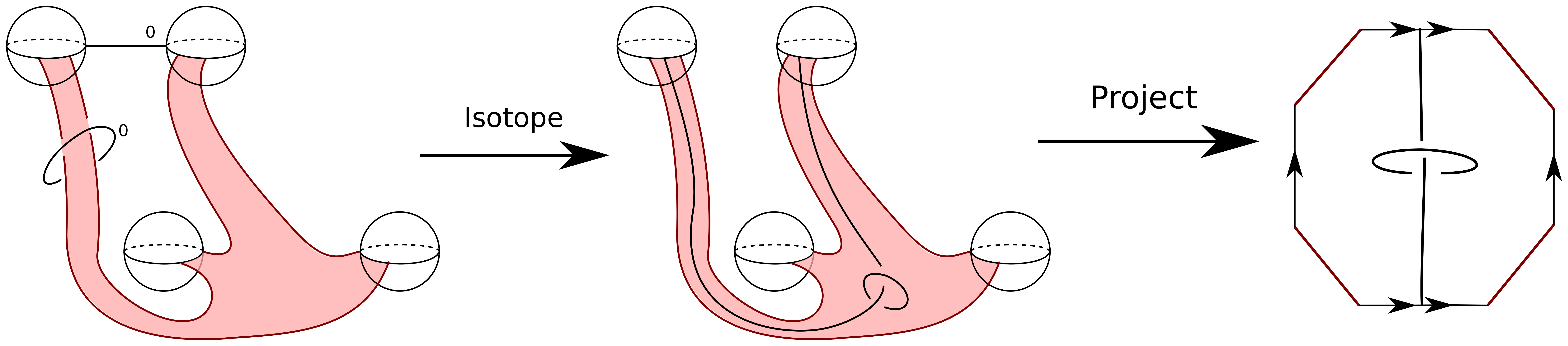}
\caption{After sliding, notice that the page has a handle decomposition induced by the 0-handle, the pair of 1-handles and one 2-handle of the 4-manifold.}
\label{Kirby_part_2}
\end{figure}

\begin{figure}[h]
\centering
\includegraphics[scale=.05]{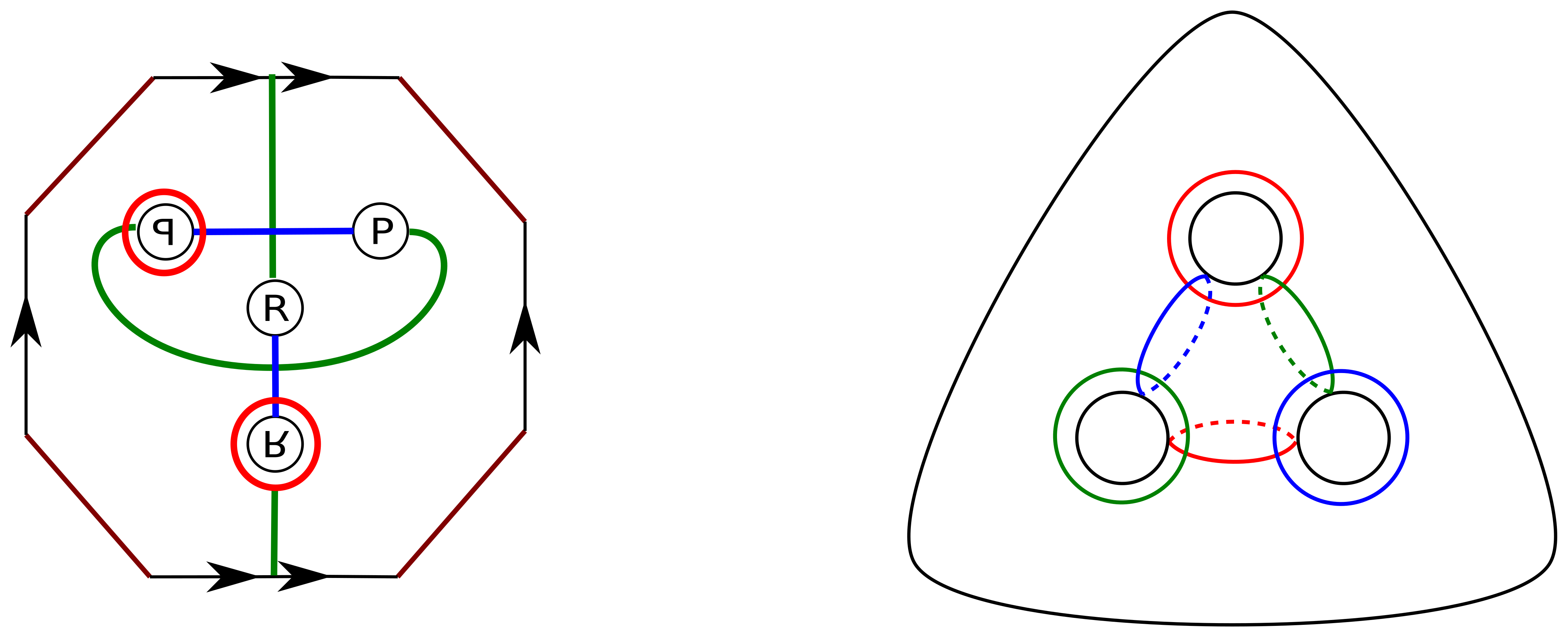}
\caption{By resolving the crossings as in Step 4, since we have the same number of loops of each color, we obtain a relative trisection diagram for $X$ (left). We get the diagram in the right by a diffeomorphism of the surface.}
\label{Kirby_part_3}
\end{figure}
\begin{examp}\label{examp_cob}
Figures \ref{fig_cob_1} and \ref{fig_cob_2} %and \ref{fig_cob_3} 
show how to trisect the thickened 3-torus $T^3\times[0,1]$ in such a way that on one side the $S^1$-foliation has fiber $S^1\times S^1\times\{pt\}$ and in the other boundary the fiber is $S^1\times \{pt\}\times S^1$. 
\end{examp}

\begin{figure}[h]
\centering
\includegraphics[scale=.14]{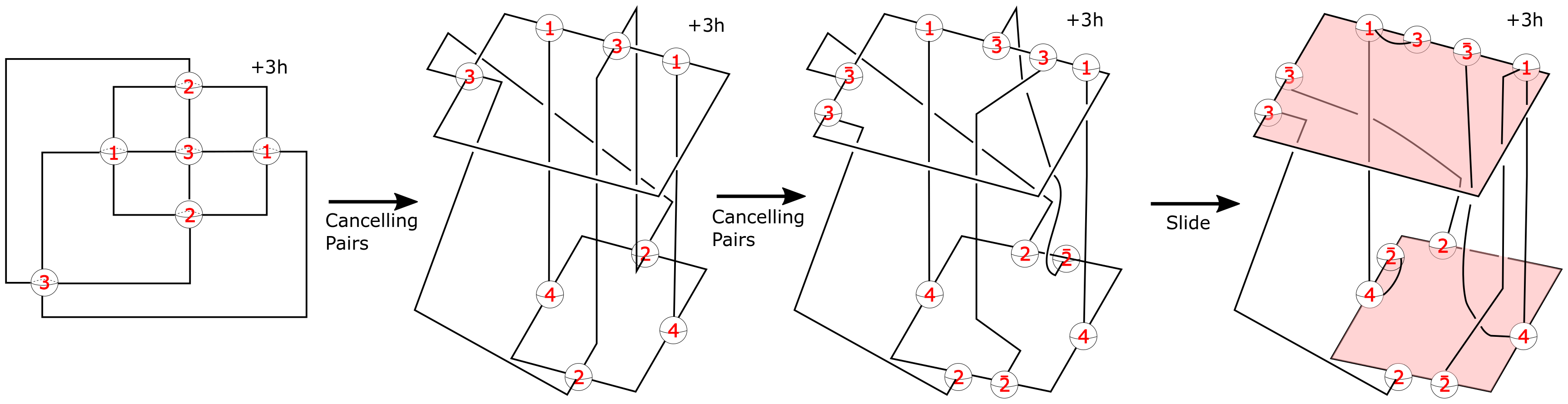}
\caption{A Kirby diagram for $T^3\times [0,1]$ obtained by thickening a Heegaard diagram for the 3-torus. After adding a 1/2-cancelling pair, you can see two embedded tori (shaded in pink) corresponding to $S^1$-fibers $T_{12}=S^1\times S^1\times\{pt\}\times \{0\}$ and $T_{13}=S^1\times\{pt\}\times S^1\times\{1\}$. Notice that the pages $T_{12}$ and $T_{13}$ have handle decompositions induced by the 0-handle, the 1-handles and some 2-handle of the 4-manifold.}
\label{fig_cob_1}
\end{figure}

\begin{figure}[h]
\centering
\includegraphics[scale=.17]{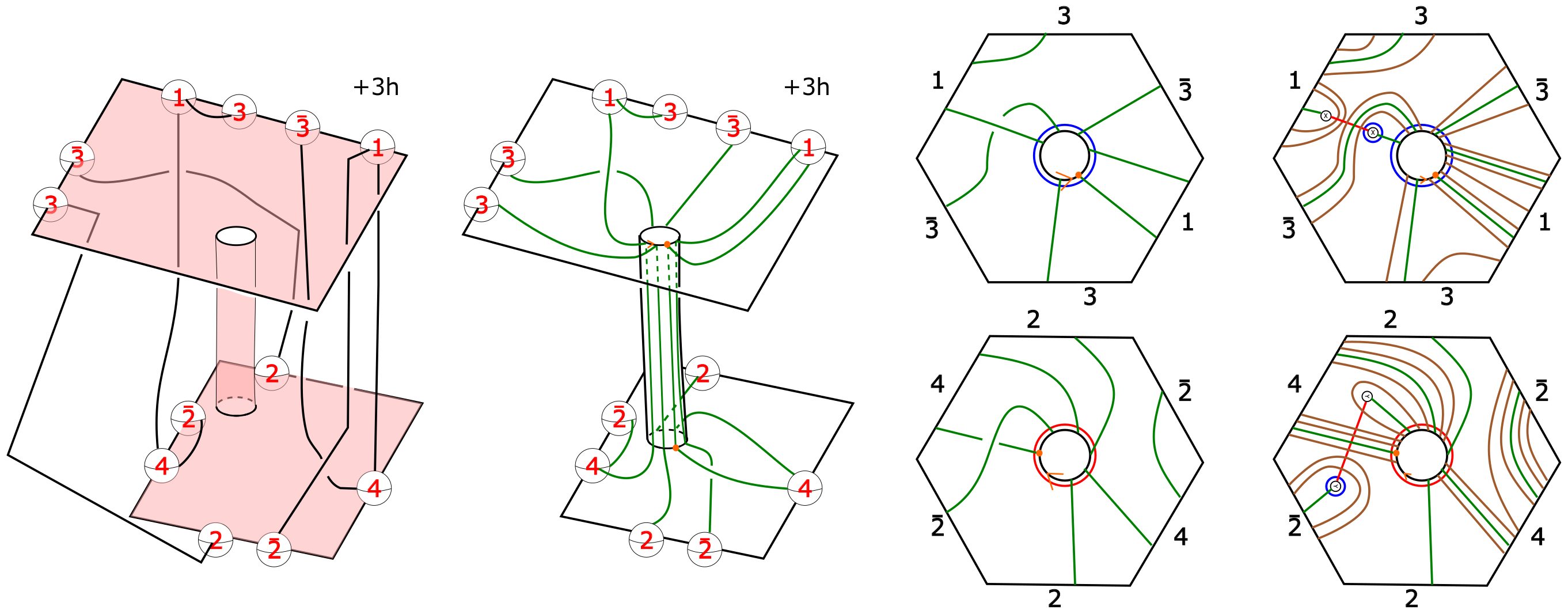}
\caption{After tubing the two tori, we draw a link diagram for the rest of the 2-handles in a genus two surface. By resolving the crossings as in Step 4, we obtain a diagram with fewer $\gamma$-curves (green). We perform Step 5 in order to find $\gamma_3$ (brown). This final result is a relative trisection diagram for $T^3\times[0,1]$ with $S^1$-fibers on its boundary given by $T_{12}=S^1\times S^1\times\{pt\}\times \{0\}$ and $T_{13}=S^1\times\{pt\}\times S^1\times\{1\}$.}
\label{fig_cob_2}
\end{figure}
%\begin{figure}[h]
%\centering
%\includegraphics[scale=.55]{Images/fig_cob_3.png}
%\caption{The trisection diagram of Figure \ref{fig_cob_2} after identification.}
%\label{fig_cob_3}
%\end{figure}

%%%%%%%%%%%%%%%%%%%%%%%%%%%%%%%%%%%%%%%%%%%%%%%%%%%%%%%%%
\subsection{The Monodromy Induced on $\partial X$} \label{section_monodromy}
In \cite{relative_trisections}, the authors described an algorithm to compute the monodromy of an open book decomposition induced by a trisection when the diagram has boundary ($b>0$). If the trisection surface is closed, the trisection will induce a fibration over $S^1$ and the monodromy can also be computed following a suitable modification of the original algorithm. We now describe the algorithm in general. The key idea is to take properly embedded 1-manifolds in the trisection surface that cut a page into a disk and traverse the boundary of the trisection using the correct handle slides. 

\begin{thm}[Adaptation from of Theorem 5 of \cite{relative_trisections}]
A relative trisection diagram encodes an open book decomposition or a fibration over $S^1$ on $\partial X$ with page given by $\Sigma_\alpha$, the surface resulting from $\Sigma$ by compressing along the $\alpha$ curves, and monodromy $\mu:\Sigma_\alpha \ra \Sigma_\alpha$ determined as follows: 
\begin{enumerate} 
\item Choose an ordered collection of properly embedded arcs or\footnote{We could have both arcs and curves simultaneously.} simple closed curves $a$ on $\Sigma$, disjoint from $\alpha$ and such that the corresponding 1-manifolds in $\Sigma_\alpha$ cut $\Sigma_\alpha$ into a disk\footnote{As many disks as boundary components of $X$.}. 

\item There exists a collection of properly embedded 1-manifolds $a_1$ and simple closed curves $\beta'$ in $\Sigma$ such that $(\alpha,a_1)$ is handle slide equivalent to $(\alpha,a)$, $\beta'$ is handle slide equivalent to $\beta$, and $a_1$ and $\beta'$ are disjoint. We claim that in this step we do not need to slide $\alpha$ curves over $\alpha$ curves, only $a$ 1-manifolds over $\alpha$ curves and $\beta$ curves over $\beta$ curves. Choose such an $a_1$ and $\beta'$. 

\item There exists a collection of properly embedded 1-manifolds $a_2$ and simple closed curves $\gamma'$ in $\Sigma$ such that $(\beta',a_2)$ is handle slide equivalent to $(\beta',a_1)$, $\gamma'$ is handle slide equivalent to $\gamma$, and $a_2$ and $\gamma'$ are disjoint. Again we claim that we do not need to slide $\beta'$ curves over $\beta'$ curves. Choose such an $a_2$ and $\gamma'$. 

\item There exists a collection of properly embedded 1-manifolds $a_3$ and simple closed curves $\alpha'$ in $\Sigma$ such that $(\gamma',a_3)$ is handle slide equivalent to $(\gamma',a_2)$, $\alpha'$ is handle slide equivalent to $\alpha$, and $a_3$ and $\alpha'$ are disjoint. Again we claim that we do not need to slide $\gamma'$ curves over $\gamma'$ curves. Choose such an $a_3$ and $\alpha'$. 

\item The pair $(\alpha',a_3)$ is handle slide equivalent to $(\alpha,a_*)$ for some collection of 1-manifolds. Choose such an $a_*$. Note that now $a$ and $a_*$ are both disjoint from $\alpha$ and thus we can compare the corresponding 1-manifolds in $\Sigma_\alpha$. 

\item The monodromy $\mu$ is the unique map (up to isotopy) such that \[ \mu  (\varphi_\alpha(a))=\varphi_\alpha(a_*),\] respecting the ordering of the 1-manifolds. 
\end{enumerate}
\end{thm} 

\begin{proof} 
The proof is the same as in Theorem 5 on \cite{relative_trisections}. The only observation is that the proof of Lemma 13 of \cite{relative_trisections}, a key lemma for this result, does not apply when $\Sigma_\alpha$ is closed. This problem can be solved by considering the annulus $a\times [-1,1]$ for any loop in $P$ instead of the disk in the proof of Lemma 13. The proof then works. 
\end{proof}

\begin{examp}
Figures \ref{fig_monodromy_1}, \ref{fig_monodromy_2} and \ref{fig_monodromy_3} show how to run the algorithm for the monodromy in the concrete case of the trisection of the complement of the unknotted torus in $S^4$. 
\end{examp}

\begin{figure}[h]
\centering
\includegraphics[scale=.045]{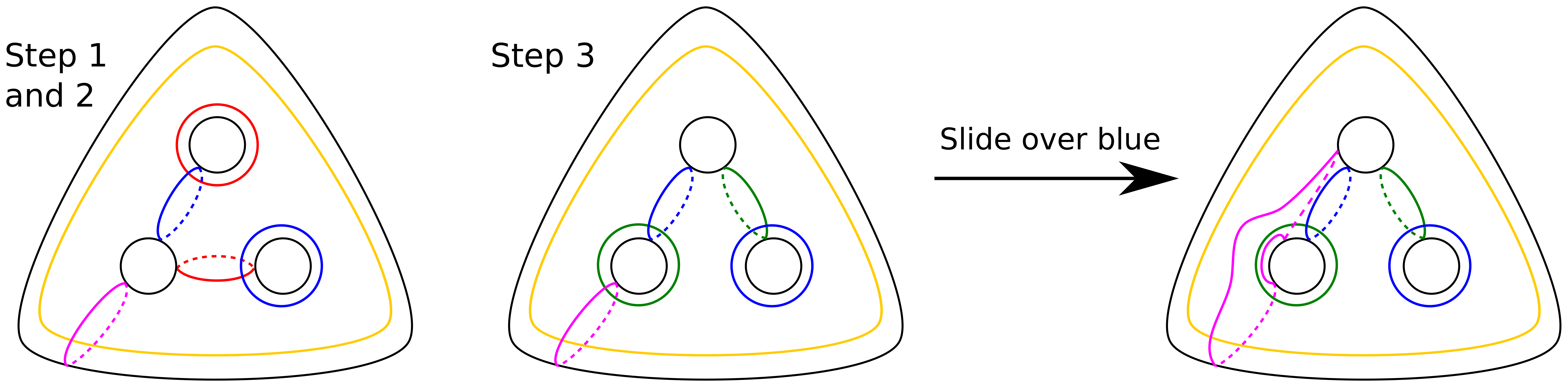}
\caption{Left: The outer pair of curves (yellow and pink) correspond to the 1-manifolds $a=a_1$ disjoint from the $\alpha$ (red) and $\beta$ (blue) loops. Right: After switching to the $\beta$, $\gamma$ (green) pair, we need to slide $a_1$ over $\beta$ to get 1-manifolds disjoint from $\gamma$, we call those $a_2$.}
\label{fig_monodromy_1}
\end{figure}

\begin{figure}[h]
\centering
\includegraphics[scale=.045]{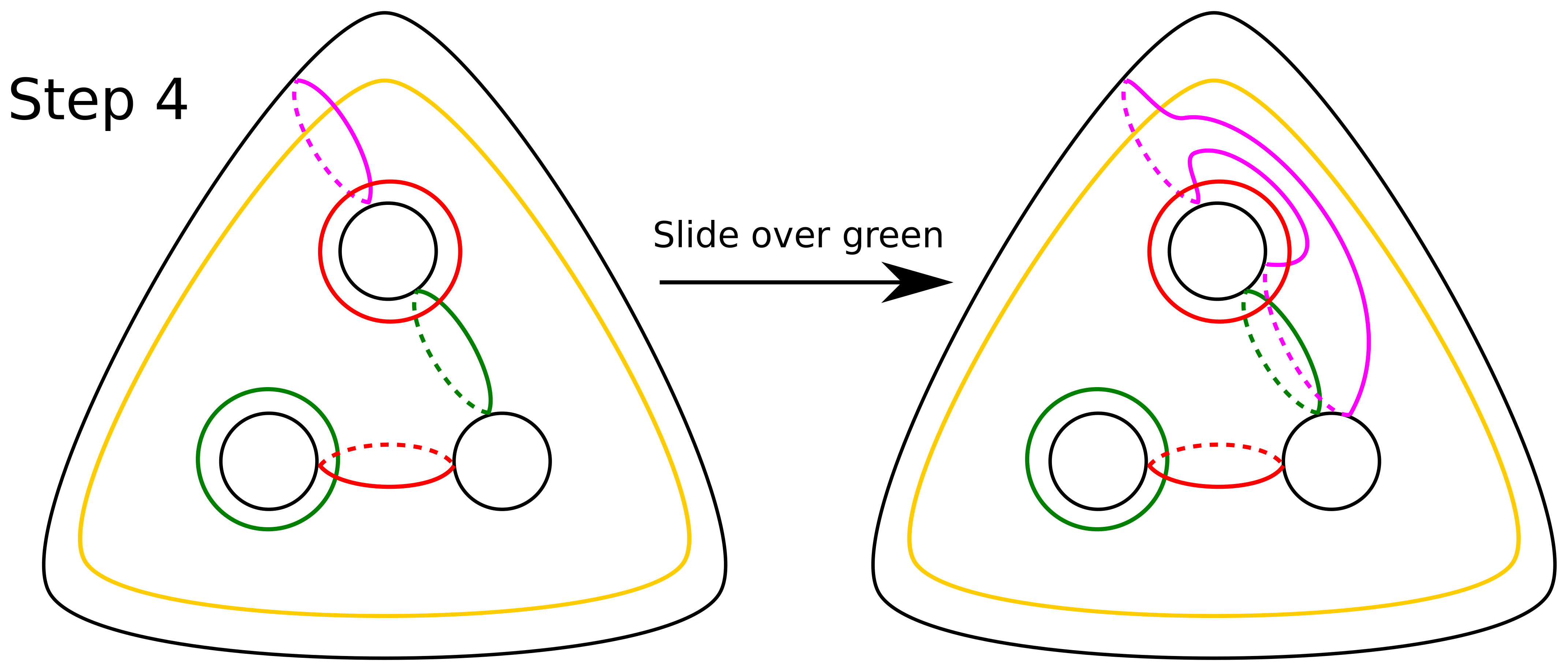}
\caption{After isotopy of $a_2$ and drawing now the pair $(\gamma,\alpha)$, we slide $a_2$ over $\gamma$ to get 1-manifolds disjoint from $\alpha$, we call those $a_3=a_*$.}
\label{fig_monodromy_2}
\end{figure}

\begin{figure}[h]
\centering
\includegraphics[scale=.045]{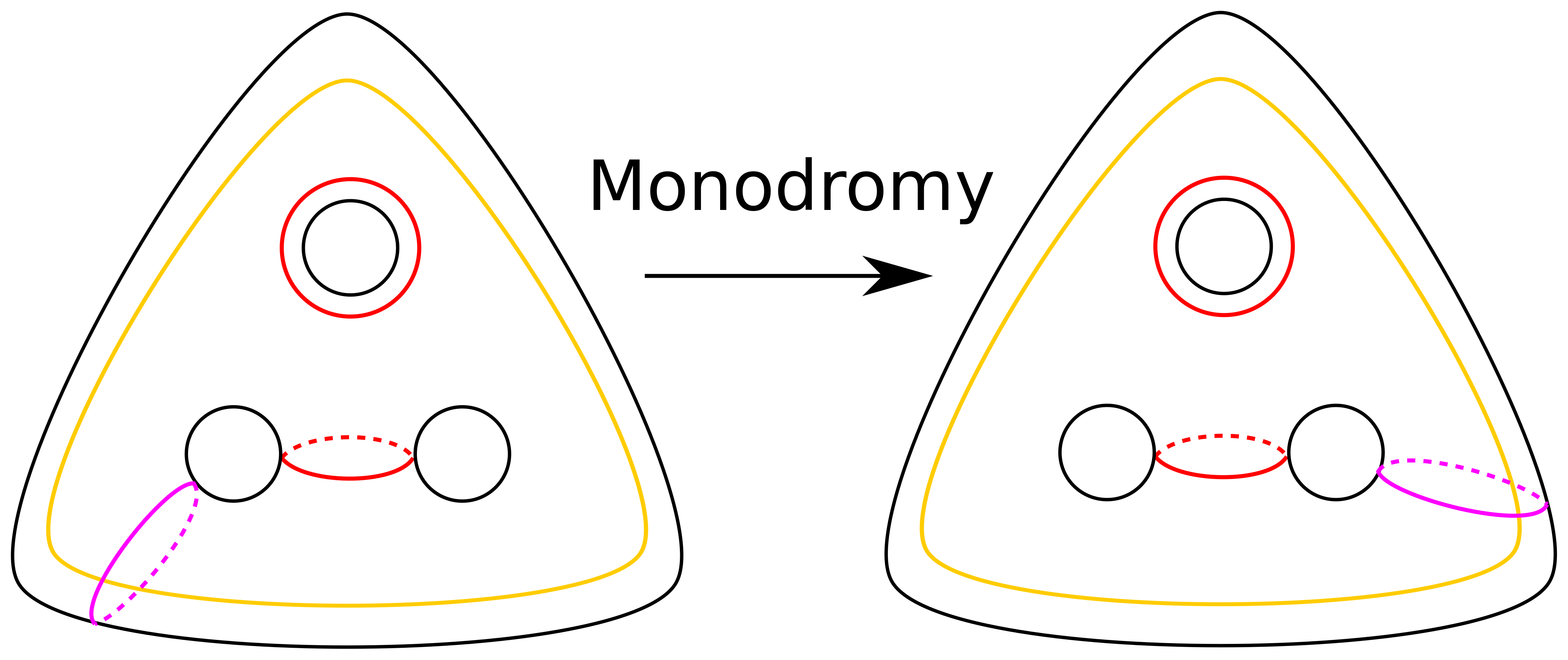}
\caption{The orientation preserving monodromy is defined in the torus obtained by compressing along the $\alpha$ loops and its determined by $a\mapsto a_*$. Notice that in this case we obtained the identity map, as expected.}
\label{fig_monodromy_3}
\end{figure}
\clearpage

%%%%%%%%%%%%%%%%%%%%%%%%%%%%%%%%%%%%%%%%%%%%%%%%%%%%%%%%

\begin{flushright}
\texttt{email: jesse.moeller@huskers.unl.edu}\\
\texttt{email: jose-arandacuevas@uiowa.edu}
\end{flushright}
\end{document}